\documentclass[11pt,reqno]{amsart}
\usepackage{amssymb}
\usepackage[all]{xy}
\usepackage{xcolor}
\usepackage{hyperref}
\newtheorem{thm}{Theorem}[section]
\newtheorem{lem}[thm]{Lemma}
\newtheorem{prop}[thm]{Proposition}
\newtheorem{cor}[thm]{Corollary}

\theoremstyle{definition}
\newtheorem{dfn}[thm]{Definition}
\newtheorem{ex}[thm]{Example}

\theoremstyle{remark}
\newtheorem{remark}[thm]{Remark}

\newtheorem{notations}[thm]{Notations}

\newcommand{\CA}{{\mathcal{A}}}

\newcommand{\CE}{{\mathcal{E}}}

\newcommand{\CF}{{\mathcal{F}}}
\newcommand{\CG}{{\mathcal{G}}}

\newcommand{\CU}{{\mathcal{U}}}
\newcommand{\CV}{{\mathcal{V}}}

\newcommand{\CI}{{\mathcal{I}}}

\newcommand{\CJ}{{\mathcal{J}}}

\newcommand{\CL}{{\mathcal{L}}}

\newcommand{\CZ}{{\mathcal{Z}}}
\newcommand{\CB}{{\mathcal{B}}}
\newcommand{\CR}{{\mathcal{R}}}

\newcommand{\CO}{{\mathcal{O}}}

\newcommand{\CW}{{\mathcal{W}}}

\newcommand{\af}{\alpha}
\newcommand{\bt}{\beta}
\newcommand{\gm}{\gamma}
\newcommand{\dt}{\delta}

\newcommand{\ld}{\lambda}

\newcommand{\sm}{\sigma}

\newcommand{\Ld}{\Lambda}

\newcommand{\Z}{{\mathbb{Z}}}

\newcommand{\N}{{\mathbb{N}}}

\newcommand{\WCB}{{\widehat{\mathcal{B}}}}

\makeatletter
\@namedef{subjclassname@2020}{%
  \textup{2020} Mathematics Subject Classification}
\makeatother

\begin{document}

\title[Boundary path groupoids of GBDS and their $C^*$-algebras]
{Boundary path groupoids of generalized Boolean dynamical systems and their $C^*$-algebras }

\author[Gilles G. de Castro]{Gilles G. de Castro}
\address[Gilles G. de Castro]{Departamento de Matem\'atica, Universidade Federal de Santa Catarina, 88040-970 Florian\'opolis SC, Brazil.} \email{gilles.castro@ufsc.br}

\author[E. J. Kang]{Eun Ji Kang}
\address[Eun Ji Kang]{Department of Mathematics, Research Institute for Natural Sciences,  Hanyang University, Seoul 04763, Korea} \email{kkang3333\-@\-gmail.\-com}

\thanks{The first named author was partially supported by Capes-PrInt Brazil grant number 88881.310538/2018-01. The second named author was  supported by Basic Science Research Program through the National Research Foundation of Korea(NRF) funded by the Ministry of Science and ICT (NRF-2020R1A4A3079066) and by  the Ministry of Education (NRF-2020R1I1A1A01072970).}

\subjclass[2020]{Primary: 46L55, Secondary: 37B99, 46L05}

\keywords{Generalized Boolean dynamical systems, Inverse semigroups, Groupoids, Tight groupoids, Boundary path groupoids, Topological correspondences}

\begin{abstract} 
In this paper, we provide two types of boundary path groupoids from a generalized Boolean dynamical system $(\CB,\CL, \theta, \CI_\af)$.  For the first groupoid, we associate an inverse semigroup to a generalized Boolean dynamical system and use the tight spectrum $\mathsf{T}$ as the unit space of a groupoid $\Gamma(\CB,\CL, \theta, \CI_\af)$ that is isomorphic to the tight groupoid $\CG_{tight}$. The other one is defined as the Renault-Deaconu groupoid $\Gamma(\partial E, \sm_E)$ arising from a topological correspondence $E$ associated with a generalized Boolean dynamical system.  We then  prove  that the tight spectrum $\mathsf{T} $ is homeomorphic to the boundary path space $\partial E$ obtained from the topological correspondence.
 Using this result, we prove that  the groupoid $\Gamma(\CB,\CL, \theta, \CI_\af)$ equipped with the topology  induced from  the topology on  $\CG_{tight}$
is isomorphic to  $\Gamma(\partial E, \sm_E)$
 as a topological groupoid.
Finally, we show that their $C^*$-algebras are isomorphic to the $C^*$-algebra of the generalized Boolean dynamical system.

\end{abstract}
 
\maketitle

\setcounter{equation}{0}

\section{Introduction}

Ever since a class of $C^*$-algebras associated to directed graphs was introduced in \cite{KPRR, KPR}, there have been various generalizations of graph algebras. 
The $C^*$-algebras associated with topological graphs  \cite{Ka2004},  higher rank graphs \cite{KP1}, labeled spaces \cite{BP1}, Boolean dynamical systems \cite{COP} and generalized Boolean dynamical systems \cite{CaK2} are those generalizations among many others. 

 When studying $C^*$-algebras, it is  useful  to describe them as groupoid $C^*$-algebras. In \cite{KPRR}, 
graph  algebras associated to row-finite graphs with no sources were introduced  as Renault-Deaconu groupoid $C^*$-algebras.  
Based on this work, a groupoid model of most of  the classes mentioned above was constructed.   
 For example, 
  Katsura proved in \cite{Ka2009} that when vertex and edge spaces of a topological graph  are both compact and the range map is surjective,  the topological graph algebra 
 is isomorphic to a Renault-Deaconu groupoid $C^*$-algebra. Then, Yeend proved in \cite{Yeend2006} that every topological higher-rank graph $C^*$-algebra is an $\acute{e}$tale groupoid $C^*$-algebra. Similar to Yeend's results,  Kumjian and Li showed  in \cite{KL2017} that  
every twisted topological graph algebra
can be realized as  a Renault-Deaconu groupoid $C^*$-algebra.
Also, the first  author together with Boava and Mortari provided a kind of  Renault-Deaconu groupoid model for weakly left-resolving normal labeled spaces in \cite{BCM3}.

One of  the purposes of  this paper is to construct  several groupoid models  for generalized Boolean dynamical systems. 
 The class of $C^*$-algebras of generalized Boolean dynamical systems was introduced in \cite{CaK2} in order to provide a common approach to the $C^*$-algebras of weakly left-resolving normal labeled spaces and  Boolean dynamical systems for which  each action has compact range and closed domain. 
 We first construct a kind of  Renault-Deaconu groupoid  by extending the  results known in the context of labeled spaces. 
Precisely, motivated by the work bringing inverse semigroup theory to the study of labeled space given in \cite{BCM1}, we  define an inverse semigroup $S_{(\CB,\CL, \theta, \CI_\af)}$ associated with generalized Boolean dynamical system $(\CB,\CL, \theta, \CI_\af)$ and characterize the tight spectrum $\mathsf{T}$ of
the idempotent semilattice of this inverse semigroup.
We then define a groupoid $\Gamma(\CB,\CL, \theta, \CI_\af)$ having $\mathsf{T}$ as its unit space and show that $\Gamma(\CB,\CL, \theta, \CI_\af)$  is isomorphic to the tight groupoid $\CG_{tight}$ arising from the inverse semigroup $S_{(\CB,\CL, \theta, \CI_\af)}$. 
By using some cutting and gluing map, we build a  local homeomorphism $\sm$ on a subset of $\mathsf{T}$ and prove that the  
groupoid $\Gamma(\CB,\CL, \theta, \CI_\af)$ equipped with the topology  induced from  the topology on  $\CG_{tight}$ can be seen as a Renault-Deaconu groupoid.

Carlsen, Ortega and Pardo characterize Boolean dynamical system for which  each action has compact range and closed domain as 0-dimensional topological graph in \cite{COP}. Motivated their work, we construct a topological correspondence $E:=E_{(\CB,\CL,\theta,\CI_\alpha)}$ from a generalized Boolean dynamical system. We then consider the  boundary path space $\partial E$ analogous to that of \cite[Definition 3.1]{KL2017} from the topological correspondence.  Thereby, we have a Renault-Deaconu groupoid $\Gamma(\partial E, \sm_E)$  associated to a shift map $\sm_E$  on the boundary path space. 

The second goal and main task of this paper is to examine the relations between the introduced groupoid models. We   prove  that the tight spectrum $\mathsf{T} $ is homeomorphic to the boundary path space $\partial E$ and that  the groupoid $\Gamma(\CB,\CL, \theta, \CI_\af)$ 
is isomorphic to  $\Gamma(\partial E, \sm_E)$
 as  topological groupoids. This simplifies the description of the tight spectrum found in \cite{BCM1} for labeled spaces and also unifies the groupoids models found in \cite{BCM3} for labeled spaces and \cite{COP} for Boolean dynamical systems.
 
The final goal of this paper is to prove that the C*-algebra of the above groupoids are isomorphic to the C*-algebra of the corresponding generalized Boolean dynamical system, which again unifies some of the results of \cite{BCM3} and \cite{COP}.

The structure of the paper is as follows: in Section \ref{Preliminaries}, we review some of the concepts needed for this paper. In Section \ref{section:inverse.semigroup}, we define an inverse semigroup from a generalized Boolean dynamical system and we describe the sets of filters, ultrafilters and tight filters of the corresponding semilattice of idempotents. In Section \ref{filter surgery}, we define maps that we call cutting and gluing maps that will be used to define one of the groupoid models as well as the topological correspondence of a generalized Boolean dynamical system. In Section \ref{section:tight.groupoid}, we study the tight groupoid of the aforementioned inverse semigroup and give it a Renault-Deaconu groupoid. In Section \ref{section:C*-isomorphism}, we prove that C*-algebra of the tight groupoid is isomorphic to the C*-algebra of the generalized Boolean dynamical system. Finally, in Section \ref{section:topological.correspondence}, we define a topological correspondence from a generalized Boolean dynamical, we prove that its C*-correspondence is isomorphic to the C*-correspondence found in \cite{CaK2} and that its groupoid is isomorphic to the tight groupoid. As a corollary, we have the C*-algebra of the topological correspondence is also isomorphic to the C*-algebra of the generalized Boolean dynamical system.

\section{Preliminaries}\label{Preliminaries}

\subsection{Filters}

\begin{dfn}\label{filter}
 A {\em filter} in a partially ordered set $P$ with least element 0 is a non-empty subset  $\xi$ of $P$ such that 
\begin{enumerate}
\item[(i)] $0 \notin \xi$,
\item[(ii)] if $ x \in \xi$ and  $x \leq y$, then $y \in \xi$,
\item[(iii)] if $x,y \in \xi$, there exists $z \in \xi$ such that $z \leq x$ and $z \leq y$.
\end{enumerate}
An {\it ultrafilter} is a filter which is not properly contained in any filter. 
\end{dfn}
If $P$ is a (meet) semilattice, condition (iii) may be replaced by $x \wedge y$ if $x,y \in \xi$.

For $x \in P$, we denote the principal filter generated by x by $\uparrow x$, that is, 
$$\uparrow x=\{y \in P: x \leq y\}.$$
When $X, Y \subseteq P$, we write
$\uparrow X:=\{b \in P: a \leq b ~\text{for some}~a \in A\}$
and $\uparrow_Y X:= Y \cap \uparrow X$. 

\subsection{Inverse semigroups}\label{Inverse semigroups} We here briefly summarize some basic facts of inverse semigroups. 
The reader is referred to  \cite[Sect. 4]{Ex1} or \cite{EP} for more details. 
 
 \vskip 1pc 
\begin{enumerate}
\item Let $S$ be an inverse semigroup. Then $E:=E(S)=\{e \in S: e^2=e\}$ is the {\it idempotent semi-lattice} of $S$ under the order relation 
$$e \leq f \iff ef=e$$
 for all $e,f \in E$ 
and the greatest lower bound $e \wedge f := ef$. The order relation on $E$ is extended to an order relation on $S$, defined by
$$s \leq t \iff ts^*s=s=ss^*t.$$

\item Given $e,f \in E$, we say that $e$ is {\it orthogonal } to f, in symbols $e \perp f$, when $e  f=0$. Otherwise, we say that $e$ {\it intersects} $f$.
Given any $F \subseteq E$, we say that $Z \subseteq F$ is a cover for $F$ if for every $0 \neq x \in F$ there exists $z \in Z$ such that $z x \neq 0 $. $Z$ is cover for $y$ if it is a cover for $F=\{x \in E: x \leq y\}$.
\item A {\it character} on $E$ is a nonzero function  $\phi$ from $ E$ to the Boolean algebra $ \{0,1\}$ such that $$ \phi(0)=0  ~\text{and}~  \phi(ef)=\phi(e)\phi(f)$$ for all $e,f \in E$. We denote by $\hat E_0$ the set of all characters on $E$. We view $\hat E_0$ as a topological space equipped with the product topology inherited from $\{0,1\}^{E}$. 
   So, given finite subsets $X, Y \subseteq E$,  the set 
$$U(X,Y):=\{\phi \in \hat E_0: \phi=1 ~\text{on}~ X  ~\text{and} ~\phi=0 ~ ~\text{on}~ Y\}$$ is an open set in $\hat E_0$ and the collection of all such is a 
  basis of the topology of $\hat E_0$.
  It is easy to see that $\hat E_0$ is a locally compact totally disconnected Hausdorff space. 
\item Given a filter $\eta$ in $E$, the map $$\phi_{\eta}: E \rightarrow \{0,1\} ~\text{ given by}~ \phi_{\eta}(e)=[e \in \eta]$$ is a character. Conversely, for a character $\phi$ on $E$, the set $$\eta_{\phi}=\{e \in E: \phi(e)=1\}$$ is a filter. These correspondences are mutually inverses.

A character $\phi$ of $E$ is called an  {\it ultra-character} if the corresponding filter $\xi_{\phi}$ is an ultrafilter. 
We denote by $\hat E_{\infty}$ the set of all ultra-characters.

\item (\cite[Definition 2.21]{BCM1}) A character $\phi$ of $E$ is {\it tight} if for every $x \in E$ and every finite cover $Z$ for $x$, we have 
$$ \bigvee_{z \in Z} \phi(z) = \phi(x).$$
The set of all tight characters is denoted by $\hat E_{tight}$, and 
 called the {\it tight spectrum} of $E$. It is a closed subspace of $\hat E_0$ containing $\hat E_{\infty}$ as a dense subspace. 
 
  A filter $\xi$ is a {\it tight filter} if its associated character $\phi_{\xi}$ is a tight character.
 \end{enumerate}
 
 The following is frequently used in section 4.2.

\begin{prop}\label{tight filter:how to check}(\cite[Proposition 2.23]{BCM1}) A filter $\xi$ in E is tight if and only if for every $x \in \xi$ and every finite cover $Z$ for $x$, one has $Z \cap \xi \neq \emptyset$. 
\end{prop}

\subsection{The tight groupoid of \texorpdfstring{$S$}{S}}\label{tight groupoid} In this subsection, we recall the construction of the tight groupoid of an inverse semigroup as a groupoid of germs as done in \cite{Ex1}. 
 Let $S$ be a inverse semigroup with 0 and $E:=E(S)$ be its idempotent semilattice.   For each idempotent $e \in E$, we let $$D_e:=\{ \phi \in \hat E_{tight}: \phi(e)=1\}.$$
 The standard action $\rho$ of $S$ on $\hat{E}_{tight}$ is defined as follows:  
 given $s \in S$,   $\rho_s:D_{s^*s} \to D_{ss^*}$ is given by 
 $\rho_s(\phi)(e)=\phi(s^*es).$ 

\begin{dfn}(\cite[Chapter 4]{Ex1}) \label{def:tight groupoid}
Put $$\Omega:=\{ (s,\phi) \in S \times \hat{E}_{tight} : \phi \in D_{s^*s} \}$$
 and define an equivalence relation on $\Omega$ by 
$$(s,\phi) \sim (t,\psi) ~\text{if}~ \phi=\psi ~\text{and  there exists}~ e \in E ~\text{such that}~ \phi \in D_e ~\text{and}~ se=te.$$ The equivalence class of $(s,\phi)$ is called {\it germ} of $s$ at $\phi$ and denoted by $[s,\phi]$.
Let  $\CG_{tight}:=\CG_{tight}(S)= \Omega/\sim$ be the set of all germs.
Define 
$$\CG^{(2)}_{tight}:=\{([s,\phi],[t,\psi]) \in \CG_{tight} \times \CG_{tight}: \phi =\rho_t(\psi)\} \subseteq \CG_{tight} \times \CG_{tight}.$$
For $([s,\phi],[t,\psi]) \in \CG^{(2)}_{tight}$ define
$$[s,\rho_t(\psi)] \cdot [t,\psi]=[st,\psi]$$
and
$$[s,\phi]^{-1}=[s^*, \rho_s(\phi)].$$
Then $\CG_{tight}$ is a groupoid, which we call $\CG_{tight}$ {\it the tight groupoid (of germs)}. 
\end{dfn}

\begin{remark} Let $\CG_{tight}$ be the groupoid with the operation defined above. 
\begin{enumerate}
\item 
The unit space is the set
$$\CG_{tight}^{(0)}=\{[e,\phi]: \phi \in D_e \},$$
which may be identified with $\hat E_{tight}$ via the bijection $[e, \phi] \mapsto \phi$ (\cite[Equation (3.9)]{EP}).
Under the above identification, we write  the domain and source of $[s,\phi]$  by  
$$d([s,\phi])=\phi ~\text{and}~r([s,\phi])= \rho_s(\phi).$$
\item (\cite[Proposition 4.13]{Ex1}) Given $s \in S$ and any open set $\CU \subseteq D_{s^*s}$, let $$\Theta(s, \CU)=\{[s,\phi]: \phi \in \CU\}.$$ Then the collection of all  $\Theta(s, \CU)$ forms a basis for a topology on $\CG_{tight}$.
  \item  (\cite[Proposition 4.15 and 4.18]{Ex1}) Given $s \in S$, let  $\CU \subseteq D_{s^*s}$ be an open set.  Then  
  the map $$\pi: \CU \rightarrow \Theta(s, \CU) ~\text{given by}~ \pi(\phi)=[s,\phi]$$ is a homeomorphism. This implies that the set $\Theta(s,\CU)$ is a compact-open bisection of $\CG_{tight}$. Thus, $\CG_{tight}$ is a locally compact étale groupoid.
  \item (\cite[Proposition 4.11 and Corollary 4.16]{Ex1}) 
The  identification given (1) is a homeomorphism, and hence 
 $\CG_{tight}^{(0)} $ is a locally compact totally disconnected Hausdorff space. 
  \end{enumerate}
  \end{remark}

\subsection{Boolean algebras}

A {\em Boolean algebra} is a set $\CB$ with a distinguished element $\emptyset$ and maps $\cap: \CB \times \CB \rightarrow \CB$, $\cup: \CB \times \CB \rightarrow \CB$ and $\setminus: \CB \times \CB \rightarrow \CB$ such that $(\CB,\cap,\cup)$ is a distributive lattice, $A\cap\emptyset=\emptyset$ for all $A\in\CB$, and $(A\cap B)\cup (A\setminus B)=A$ and $(A\cap B)\cap (A\setminus B)=\emptyset$ for all $A,B\in\CB$. The Boolean algebra $\CB$ is called {\em unital} if there exists $1 \in \CB$ such that $1 \cup A = 1$ and $1 \cap A=A$ for all $A \in \CB$ (often, Boolean algebras are assumed to be unital and what we here call a Boolean algebra is often called a \emph{generalized Boolean algebra}).

We call $A\cup B$ the \emph{union} of $A$ and $B$, $A\cap B$ the \emph{intersection} of $A$ and $B$, and $A\setminus B$ the \emph{relative complement} of $B$ with respect to $A$. A subset $\CB' \subseteq \CB$ is called a {\em Boolean subalgebra} if $\emptyset\in\CB'$ and $\CB'$ is closed under taking unions, intersections and relative complements. A Boolean subalgebra of a Boolean algebra is itself a Boolean algebra.

We define a partial order on $\CB$ as follows: for $A,B \in \CB$,
\[
A \subseteq B ~~~\text{if and only if}~~~A \cap B =A.
\]
Then $(\CB, \subseteq)$ is a partially ordered set, and $A\cup B$ and $A\cap B$ are the least upper-bound and the greastest lower-bound of $A$ and $B$ with respect to the partial order $\subseteq$. If a family $\{A_{\ld}\}_{\ld \in \Lambda}$ of elements from $\CB$ has a least upper-bound, then we denote it by $\cup_{\ld \in \Lambda} A_\ld$. If $A\subseteq B$, then we say that $A$ is a \emph{subset of} $B$.

A non-empty subset $\CI$ of $\CB$ is called  an {\em ideal} \cite[Definition 2.4]{COP} if 
\begin{enumerate}
\item[(i)] if $A, B \in \CI$, then $A \cup B \in \CI$,
\item[(ii)] if $A \in \CI$ and $ B \in \CB$, then   $A \cap B \in \CI$. 
\end{enumerate}

An ideal $\CI$ of a Boolean algebra $\CB$ is a Boolean subalgebra. For $A \in \CB$, the ideal generated by $A$ is defined by $\CI_A:=\{ B \in \CB : B \subseteq A\}.$

A filter in $\CB$ is {\em prime} if for every $B,B' \in \CB$ with $B \cup B'\in\xi$, we have that either $B \in \xi$ or $B' \in \xi$. We note that $\xi$ is an ultrafilter if and only if it is a prime filter, and we use this equivalence throughout the paper without further mention.

 Given a Boolean algebra $\CB$, we write $\widehat{\CB}$  for the set of all ultrafilters of $\CB$. Notice that if $A\in\CB\setminus\{\emptyset\}$, then $\{B\in\CB:A\subseteq B\}$ is a filter, and it then follows from Zorn's Lemma that there is an ultrafilter $\eta\in\widehat{\CB}$ that contains $A$. For $A\in\CB$, we let $$Z(A):=\{\xi\in\widehat{\CB}:A\in\xi\}$$ and we equip $\widehat{\CB}$ with the topology generated by $\{Z(A): A\in\CB\}$. Then $\widehat{\CB}$ is a totally disconnected locally compact Hausdorff space, $\{Z(A): A\in\CB\}$ is a basis for the topology, and each $Z(A)$ is compact and open.

\subsection{Boolean dynamical systems}

A map $\phi: \CB \rightarrow \CB'$ between two Boolean algebras is called a {\em Boolean homomorphism} (\cite[Definition 2.1]{COP}) if $\phi(A \cap B)=\phi(A) \cap \phi(B)$, $\phi(A \cup B)=\phi(A) \cup \phi(B)$, and $\phi(A \setminus B)=\phi(A) \setminus \phi(B)$ for all $A,B \in \CB$.

A map $\theta: \CB \rightarrow \CB $ is called an {\it action} (\cite[Definition 3.1]{COP})  on a Boolean algebra $\CB$ if it is a Boolean homomorphism with $\theta(\emptyset)=\emptyset$.

Given a set $\CL$ and any $n \in \N$, we define $\CL^n:=\{(\af_1, \dots, \af_n): \af_i \in \CL\}$,  $\CL^{\geq 1}=\cup_{n \geq 1} \CL^n$  and $\CL^*:=\cup_{n \geq 0} \CL^n$, where $\CL^0:=\{\emptyset \}$. For $\alpha\in\CL^n$, we write  $|\af|:=n$. For $\af=(\af_1, \dots, \af_n), \beta=(\beta_1,\dots,\beta_m) \in \CL^*$, we will usually write $\af_1 \dots \af_n$ instead of $(\af_1, \dots, \af_n)$ and use $\alpha\beta$ to denote the word $\af_1 \cdots \af_n\beta_1\dots\beta_m$ (if $\alpha=\emptyset$, then $\alpha\beta:=\beta$; and if $\beta=\emptyset$, then $\alpha\beta:=\alpha$). 
  For $1\leq i\leq j\leq |\af|$, we also denote by $\af_{i,j}$ the sub-word $\af_i\cdots \af_j$ of  $\af=\af_1\af_2\cdots\af_{|\af|}$, where $\af_{i,i}=\af_i$. If $j < i$, set $\af_{i,j} =\emptyset$.
  
 We also  let $\CL^\infty$ denote the set of infinite sequences with entries in $\CL$. 
If $x=(x_1,x_2,\dots)\in\CL^\infty$ and $n\in\N$, then we let $x_{1,n}$ denote the word $x_1x_2\cdots x_n\in\CL^n$. 

 Given $\af,\bt \in \CL^*$, we say $\af$ is a {\it beginning} of $\bt$ if $\bt=\af\bt'$ for some $\bt' \in \CL^*$. When $\af$ is a  beginning of $\bt$, then we also call $\af$ the {\it initial path or initial segment} of $\bt$.
 We say that $\af$ and $\bt$ are {\it comparable} if $\af$ is a beginning of $\bt$ or $\bt$ is a beginning of $\af$ (\cite[Definition 2.7]{BCM1}).

A {\em Boolean dynamical system} is a triple $(\CB,\CL,\theta)$ where $\CB$ is a Boolean algebra, $\CL$ is a set, and $\{\theta_\af\}_{\af \in \CL}$ is a set of actions on $\CB$. For $\af=\af_1 \cdots \af_n \in \CL^{\geq 1}$, the action $\theta_\af: \CB \rightarrow \CB$ is defined as $\theta_\af:=\theta_{\af_n} \circ \cdots \circ \theta_{\af_1}$.  We  also define $\theta_\emptyset:=\text{Id}$.

For $B \in \CB$, we define
\[
\Delta_B^{(\CB,\CL,\theta)}:=\{\af \in \CL:\theta_\af(B) \neq
\emptyset \} ~\text{and}~  \ld_B^{(\CB,\CL,\theta)}:=|\Delta_B^{(\CB,\CL,\theta)}|.
\]
We will often just write $\Delta_B$ and $\ld_B$ instead of
$\Delta_B^{(\CB,\CL,\theta)}$ and $\ld_B^{(\CB,\CL,\theta)}$.

We say that $A \in \CB$ is {\em regular} (\cite[Definition 3.5]{COP})
if for any $\emptyset \neq B \in \CI_A$, we have $0 < \ld_B < \infty$.
If $A \in \CB$ is not regular, then it is called a {\em singular} set.
We write $\CB^{(\CB,\CL,\theta)}_{reg}$ or just $\CB_{reg}$ for the
set of all regular sets. Notice that $\emptyset\in\CB_{reg}$.

\subsection{Generalized Boolean dynamical systems and their \texorpdfstring{$C^*$}{C*}-algebras}\label{GBDS}

Let $(\CB,\CL,\theta)$ be a Boolean dynamical system and let 
\[
\mathcal{R}_\alpha^{(\CB,\CL,\theta)}:=\{A\in\mathcal{B}:A\subseteq\theta_\alpha(B)\text{ for some }B\in\mathcal{B}\}
\]
for each $\alpha \in \mathcal{L}$. Note that each $\CR_\af^{(\CB,\CL,\theta)}$ is an ideal of $\CB$. 
We will often, when it is clear which Boolean dynamical system we are working with, just write $\CR_\af$ instead of $\CR_\af^{(\CB,\CL,\theta)}$.
\

\begin{dfn}(\cite[Definition 3.2]{CaK2})\label{def:GBDS} 
A {\em generalized Boolean dynamical system} is a quadruple  $(\CB,\CL,\theta,\CI_\alpha)$ where  $(\CB,\CL,\theta)$ is  a Boolean dynamical system  and  $\{\CI_\alpha:\alpha\in\CL\}$ is a family of ideals in $\CB$ such that $\CR_\alpha\subseteq\CI_\alpha$ for each $\alpha\in\CL$.
\end{dfn}

\begin{dfn}\label{def:representation of RGBDS} 
Let $(\CB,\CL,\theta, \CI_\af)$ be a  generalized Boolean dynamical system. A {\it  $(\CB, \CL, \theta, \CI_\af)$-representation (or  Cuntz--Krieger representation of $(\CB, \CL,
\theta,\CI_\af)$)} is a family of projections $\{P_A:A\in\mathcal{B}\}$ and a family of partial isometries $\{S_{\alpha,B}:\alpha\in\mathcal{L},\ B\in\mathcal{I}_\alpha\}$ such that for $A,A'\in\mathcal{B}$, $\alpha,\alpha'\in\mathcal{L}$, $B\in\mathcal{I}_\alpha$ and $B'\in\mathcal{I}_{\alpha'}$,
\begin{enumerate}
\item[(i)] $P_\emptyset=0$, $P_{A\cap A'}=P_AP_{A'}$, and $P_{A\cup A'}=P_A+P_{A'}-P_{A\cap A'}$;
\item[(ii)] $P_AS_{\alpha,B}=S_{\alpha,  B}P_{\theta_\af(A)}$;
\item[(iii)] $S_{\alpha,B}^*S_{\alpha',B'}=\delta_{\alpha,\alpha'}P_{B\cap B'}$;
\item[(iv)] $P_A=\sum_{\af \in\Delta_A}S_{\af,\theta_\af(A)}S_{\af,\theta_\af(A)}^*$ for all  $A\in \mathcal{B}_{reg}$. 
\end{enumerate}
\end{dfn}

Given a $(\CB, \CL, \theta, \CI_\af)$-representation $\{P_A, S_{\af,B}\}$ in a $C^*$-algebra $\CA$, we denote by $C^*(P_A, S_{\af,B})$ the $C^*$-subalgebra of $\CA$ generated by $\{ P_A,  S_{\af,B}\}$.
It is shown in \cite{CaK2} that there exists  a universal $(\CB, \CL, \theta, \CI_\af)$-representation $\{p_A, s_{\af,B}: A\in \CB, \af \in \CL ~\text{and}~ B \in \CI_\af\}$   in a  $C^*$-algebra. 
 We write $C^*(\mathcal{B},\mathcal{L},\theta, \CI_\af)$ for $C^*(p_A,s_{\af,B})$ and    call it the {\it  $C^*$-algebra of $(\CB,\CL,\theta,\CI_\alpha)$}. 
 When $(\CB,\CL,\theta)$ is a Boolean dynamical system, then we write
$C^*(\CB,\CL,\theta)$ for $C^*(\CB,\CL,\theta, \CR_\af)$ and call it
the \emph{$C^*$-algebra of $(\CB,\CL,\theta)$}.

\vskip1pc

\begin{dfn}\label{def of I}(\cite{CaK2})
For $\af=\af_1\af_2 \cdots \af_n \in \CL^{\geq 1}$, we define
\begin{align*}
\CI_\af:=\{A \in \CB : A \subseteq \theta_{\af_2 \cdots \af_n}(B)~\text{for some }~ B \in \CI_{\af_1}\}.
\end{align*}
 For $\af =\emptyset$, we  define $\CI_\emptyset := \CB$ 
\end{dfn}

It is easy to check that $\CI_\af$ is an ideal of $\CB$ for  $\af \in \CL^{\geq 1}.$
For $\af\in \CL^{\geq 1}$ and $\CI \subseteq \CB$, we put $\theta_{\af}(\CI):=\{ \theta_\af(B) : B \in \CI\} $.

\begin{lem}\label{properties of I} Let $\af,\bt \in \CL^*$.  Then we have 
\begin{enumerate}
\item[(i)] $\theta_{\af}(\CB) \subseteq \CI_\af$.
\item[(ii)]  $\CI_{\af\bt} \subseteq \CI_\bt$.
\item[(iii)] If $A\in\CI_\af$, then $\theta_\bt(A)\in\CI_{\af\bt}$.
\end{enumerate}
\end{lem}

\begin{proof}(i) Since $\theta_\af(B)=\theta_{\af_{2,|\af|}}(\theta_{\af_1}(B))$ and $\theta_{\af_1}(B) \in \CR_{\af_1} \subseteq \CI_{\af_1} $ for any $B \in \CB$, we have 
$\theta_{\af}(B) \in  \CI_\af$ for any $B \in \CB$.

(ii) If $D \in \CI_{\af\bt}$, then $D \subseteq \theta_{\af_{2,|\af|}\bt}(D')$ for some $ D' \in \CI_{\af_1}$.  We then have that  $D \in \CI_{\bt}$ since $\theta_{\af_{2,|\af|}\bt}(D')=\theta_{\bt_{2,|\bt|}}(\theta_{\af_{2,|\af|}\bt_1}(D'))$ and $\theta_{\af_{2,|\af|}\bt_1}(D')=\theta_{\bt_1}(\theta_{\af_{2,|\af|}}(D')) \in \CI_{\bt_1}$. 

(iii) If $A\in \CI_{\af}$, then $A\subseteq \theta_{\af_{2,|\af|}}(A')$ for some $A'\in\CI_{\af_1}$. Hence, $\theta_\bt(A)\subseteq \theta_\bt(\theta_{\af_{2,|\af|}}(A'))=\theta_{\af_{2,|\af|}\bt}(A')$, so that $\theta_{\bt}(A)\in\CI_{\af\bt}$.
\end{proof}

\begin{remark}
Let $\{P_A,\ S_{\alpha,B}: A\in\CB, \alpha\in\CL ~\text{and}~ B\in\CI_\alpha\}$
be a $(\CB,\CL,\theta,\CI_\alpha)$-representation. 
\begin{enumerate}
\item  For $\af, \bt \in \CL^*$, $A \in \CI_\af$ and $B \in \CI_\bt$, we have the equality 
\[
S_{\af,A}^*S_{\bt,B}= \left\{ 
\begin{array}{ll}
    P_{A \cap B} & \hbox{if\ }\af =\bt \\
    S_{\af', A \cap \theta_{\af'}(B)}^* & \hbox{if\ }\af =\bt\af' \\
    S_{\bt',B \cap \theta_{\bt'}(A)}   & \hbox{if\ } \bt=\af\bt' \\
    0 & \hbox{otherwise.} \\
\end{array}
\right.
\]
 We then have
that
\begin{align}
C^*(P_A,S_{\alpha,B})&=\overline{\operatorname{span}}\{
S_{\af,A}S_{\bt,B}^*: ~\af,\bt \in \CL^* ~\text{and}~ A \in \CI_\af, B
\in \CI_\bt\}\label{eq:2}\\
&=\overline{{\rm \operatorname{span}}}\{S_{\af,A}S_{\bt,A}^*: \af,\bt
\in \CL^* ~\text{and}~ A \in \CI_\af\cap \CI_\bt \}.\label{eq:3}
\end{align}
\item It follows from the universal property of $C^*(\CB,\CL,\theta,
\CI_\af)=C^*(p_A, s_{\af,B})$ that there is a strongly continuous
action $\gm:\mathbb T\to {\rm Aut}(C^*(\CB,\CL,\theta, \CI_\af))$,
which we call the {\it gauge action}, such that

\[
\gm_z(p_A)=p_A   \ \text{ and } \ \gm_z(s_{\af,B})=zs_{\af,B}
\]
for $A\in \CB$, $\af \in \CL$ and $B \in \CI_\af$.
 
\end{enumerate}

\end{remark}

\begin{ex}
Let $(\CE,\CL,\CB)$ be a  labeled space where $\CL:\CE^1 \to \CA$ is onto and put $C^*(\CE, \CL, \CB)=C^*(p_A, s_\af)$. Then $\CB$ is a Boolean algebra  and for each $\af \in \CA$, the map $\theta_{\af}:\CB \to \CB$ defined by 
$\theta_\af(A):=r(A,\af)$
is an action on $\CB$ (\cite[Example 11.1]{COP}).  Put $\CR_\af=\{A \in \CB: A \subseteq r(B,\af) ~\text{for some}~ B \in \CB\}$ and let $$\CI_{r(\af)}=\{A \in \CB : A \subseteq r(\af)\}.$$
It is clear that $\CR_\af \subseteq \CI_{r(\af)}$  for each $\af \in \CL$. Then $(\CB, \CA,\theta, \CI_{r(\af)})$ is a generalized Boolean dynamical system and we have 
$$   C^*(\CE,\CL,\CB) \cong C^*(\CB, \CA,\theta, \CI_{r(\af)}) $$
by \cite[Example 4.2]{CaK2}.  In this case, we  call $(\CB, \CA,\theta, \CI_{r(\af)})$ {\it the generalized Boolean dynamical system associated to $(\CE, \CL,\CB)$}. 
 \end{ex}
 
\section{Inverse semigroups associated with generalized Boolean dynamical systems and its tight spectrum}\label{section:inverse.semigroup}

Motivated by the inverse semigroup treatment given to labeled spaces in \cite{BCM1}, in this section, we associate each generalized Boolean dynamical system with an inverse semigroup and characterize the tight spectrum of the inverse semigroup. 
We first give a description of filters and ultrafilters in the idempotent semilattice of this inverse semigroup in section 3.2. and section 3.3, respectively. Using these results, we  give a complete characterization of the tight filters  in section 3.4.

\subsection{An inverse semigroup}\label{An inverse semigroup}
 Let $(\CB, \CL,\theta, \CI_\af)$ be a generalized Boolean dynamical system and let
$$S_{(\CB, \CL,\theta, \CI_\af)}:=\{(\af, A, \bt): \af,\bt \in \CL^* ~\text{and}~  A \in \CI_\af \cap \CI_\bt ~\text{with}~ A \neq \emptyset \} \cup \{0\}. $$
To simplify the notation, we write $S=S_{(\CB, \CL,\theta, \CI_\af)}$ when it is clear which generalized Boolean dynamical system we are working with.

Define a binary operation on $S$ is given as follows: $s \cdot 0 = 0 \cdot s=0$ for all $s \in S$ and 
for $(\af, A, \bt)$ and $(\gm, B, \dt)$ in $S$, 

$$(\af, A, \bt) \cdot (\gm, B, \dt)=\left\{
 \begin{array}{ll}    (\af, A \cap B, \dt) & \hbox{if\ }\bt =\gm  ~\text{and}~ A\cap B \neq \emptyset, \\
 (\af\gm', \theta_{\gm'}(A)\cap B, \dt) & \hbox{if\ } \gm=\bt\gm' ~\text{and}~ \theta_{\gm'}(A)\cap B \neq \emptyset,\\
 (\af, A\cap \theta_{\bt'}(B) ,\dt\bt') & \hbox{if\ } \bt=\gm\bt'  ~\text{and}~ A\cap \theta_{\bt'}(B) \neq \emptyset,\\
          0 & \hbox{otherwise.}
                      \end{array}
                    \right.
$$
That the operation is well-defined follows from Lemma~\ref{properties of I}. If for a given $s=(\af, A, \bt) \in S$ we define $s^*=(\bt, A, \af)$, then the set $S$, endowed with the operation above, 
 is an  inverse semigroup with zero element 0 (\cite[Sect. 2.3]{BCM1}), whose semilattice of  idempotents is 
$$E(S):=\{(\af, A, \af): \af \in \CL^* ~\text{and}~ \emptyset \neq A \in \CI_\af\}\cup \{0\}.$$ 

The natural order in the semilattice $E(S)$ is given below.

\begin{lem}\label{lem:order}(\cite[Proposition 4.1]{BCM1}) Let $\af,\bt \in \CL^*$, $A \in \CI_\af$ and $B \in \CI_\bt$. Then 
 $(\af, A, \af)\leq (\bt, B, \bt)$ if and only if $\af=\bt\af'$ and $A \subseteq\theta_{\af'}(B)$.
\end{lem}

\subsection{Filters in \texorpdfstring{$E(S)$}{E(S)}}  In this section, we characterize filters in $E(S)$. In  \cite{BCM1}, it is shown  for a labeled space $(\CE, \CL, \CB)$ (i.e., the generalized Boolean dynamical system $(\CB, \CA,\theta,\CI_{r(\af)} )$) that there is a bijective correspondence between filters in $E(S)$ and pairs $(\af, \{\CF_n\}_{0 \leq n \leq |\af|})$, where $\af \in \CL^{\leq \infty}$ (in \cite{BCM1},  $\CL^\infty=\{\af \in \CA^\infty: \af_{1,n} \in \CL(\CE^n), \forall n\}$ and $\CL^{\leq \infty} = \{\omega\} \cup (\cup_{n \geq 1 }\CL(\CE^n)) \cup \CL^{
\infty}  )$ and $ \{\CF_n\}_{0 \leq n \leq |\af|}$ is a complete family for $\
\af$. 
We generalize this results to generalized Boolean dynamical systems $(\CB,\CL,\theta,\CI_\af)$. 
The main difference between the generalized Boolean dynamical system $(\CB, \CA,\theta,\CI_{r(\af)} )$ and a generalized Boolean dynamical system $(\CB, \CL,\theta,\CI_{\af} ) $ is that, for all $\af \in \CL^{\geq 1}$, $\CI_{r(\af)}$ is a unital Boolean algebra with unit $r(\af)$, but $\CI_\af$ might not be unital, in general.
We mention that the results  in \cite[Section 4]{BCM1}  proved without using the unit 
  are easily generalized  just by replacing $r(-,\af)$ to $\theta_\af(-)$, and  that  a proof of  Proposition \ref{filter gives complete family}  should be given  because 
 it is a  generalization of \cite[Proposition 4.8]{BCM1}, which is proved  using the unit  of $\CI_{r(\af)}$. We  provide all proofs  for completeness. 
  
 We also  remark that given a  generalized Boolean dynamical system $(\CB,\CL,\theta,\CI_\af)$, $\CI_\af$ can consist of only the empty set for some $\af \in \CL^*$, in which case we would have no filters in $\CI_\af$. 
So, to exclude this case, we define  $\CW^*=\{\alpha\in\CL^*:\CI_\alpha\neq \{\emptyset\}\}$, $\CW^{\infty}=\{\alpha\in\CL^{\infty}:\alpha_{1,n}\in\CW^* ~\text{for all}~ n \geq 1\}$ and $\CW^{\leq\infty}=\CW^*\cup \CW^{\infty}$ ($\CW^{\leq\infty}$ plays the same role as $\CL^{\leq\infty}$ in \cite{BCM1}).

\begin{dfn}(\cite[Definition 4.5]{BCM1}) Let $\af \in \CW^{\leq \infty}$ and $\{\CF_n\}_{0 \leq n \leq |\af|}$  (understanding that $0 \leq n \leq |\af|$ means $0 \leq n < \infty$ when $\af \in \CW^\infty$) be a family such that $\CF_n$ is a filter in $\CI_{\af_{1,n}}$ for every $n >0 $ and $\CF_0$ is a filter in $\CB$ or $\CF_0=\emptyset$ (the latter only allowed if $|\af|>0$).
The family $\{\CF_n\}_{0 \leq n \leq |\af|}$ is said to be {\it admissible for $\af$} if 
$$\CF_n \subseteq \{A \in \CI_{\af_{1,n}}: \theta_{\af_{n+1}}(A) \in \CF_{n+1}\}$$ 
for all $0\leq n<|\af|$, and  is said to be {\it complete for $\af$} if 
$$\CF_n=\{A \in \CI_{\af_{1,n}}: \theta_{\af_{n+1}}(A) \in \CF_{n+1}\}$$ 
for all $0\leq n<|\af|$. 
\end{dfn}
\begin{remark} Let  $\{\CF_n\}_{0 \leq n \leq |\af|}$ be a complete family for $\af$ with $|\af|>0$. If $\CF_1 \cap \CR_{\af_1} =\emptyset$, then $\CF_0 =\emptyset$. If   $\CF_1 \cap \CR_{\af_1} \neq \emptyset$, then $\CF_0$ is a filter in $\CB$.
\end{remark}

\begin{remark}\label{rmk:complete.family} Let $\af \in \CW^{\leq \infty}$ and $\{\CF_n\}_{0 \leq n \leq |\af|}$   be a family such that $\CF_n$ is a filter in $\CI_{\af_{1,n}}$ for every $n >0 $ and $\CF_0$ is a filter in $\CB$ or $\CF_0=\emptyset$.
\begin{enumerate}
\item If  $\af \in \CW^{\infty}$, then  $\{\CF_n\}_{n \geq 0}$  is complete for $\af$ if and only if 
$$\CF_n=\{A \in \CI_{\af_{1,n}}: \theta_{\af_{n+1, m}}(A) \in \CF_m \}$$
for all $n \geq 0$ and all $m >n$. 
\item If $\{\CF_n\}_{0 \leq n \leq |\af|}$ is a  complete family for $\af \in \CW^*$, then $$\CF_n=\{A \in \CI_{\af_{1,n}}: \theta_{\af_{n+1,|\af|}}(A) \in \CF_{|\af|} \}$$ for all $0 \leq n \leq  |\af|$.
\end{enumerate}
\end{remark}

\begin{lem}\label{lem:assoc.word}
Let $\xi$ be a filter in $E(S)$. Then, there exists a unique $\alpha\in\CW^{\leq\infty}$ such that for all $(\beta,A,\beta)\in \xi$, we have that $\beta$ is an initial segment of $\alpha$.
\end{lem}

\begin{proof}
Let $M=\{|\alpha|:(\alpha,A,\alpha)\in\xi\}$. If $M$ has a maximum, we let $\alpha$ be such that $(\alpha,A,\alpha)\in E(S)$ for some $A\in\CI_\af$ and $|\alpha|=\max M$. If $(\beta,B,\beta)\in\xi$, then $|\beta|\leq |\alpha|$, and by Lemma~\ref{lem:order}, we have that $\beta$ is an initial segment of $\alpha$.

Suppose now that $M$ does not have a maximum. For each $n >0$, choose $p=(\bt, B, \bt) \in \xi$ such that $|\bt|\geq n$ and define $\af_n=\bt_n$. Since any two elements in $\xi$ have comparable words, then $\af_n$ is well defined. Hence, we have an infinite word $\af=\af_1\af_2 \cdots \in \CW^\infty$. By construction if $(\beta,B,\beta)\in\xi$, we have that $\beta$ is an initial segment of $\alpha$.

The uniqueness of such $\alpha$ is immediate.
\end{proof}

\begin{dfn}
Let $\xi$ be a filter in $E(S)$ and let $\alpha\in\CW^{\leq\infty}$ as in Lemma~\ref{lem:assoc.word}. We say that $\alpha$ is the \emph{word associated with} $\xi$. If $|\alpha|<\infty$, we also say that $\alpha$ is the \emph{largest word} for $\xi$.
\end{dfn}

\begin{lem}(\cite[Proposition 4.3]{BCM1})\label{word to finite} Let $\af \in \CW^{*}$ and $\CF$ be a filter in $\CI_\af$. Then 
\begin{align*}\xi&= \bigcup_{A \in \CF} \uparrow(\af, A, \af) \\
&=\{(\af_{1,i}, A, \af_{1,i}) \in E(S): 0 \leq i \leq |\af| ~\text{and}~\theta_{\af_{i+1, |\af|}}(A) \in \CF\}
\end{align*}
is a filter  in $E(S)$ with largest word $\af$. 
\end{lem}

\begin{proof} First we show that the sets in the statement are equal: let $A \in \CF$ and $p=(\bt, B, \bt) \in \uparrow(\af, A, \af)$. Then $\af=\bt\af'$ and $A \subseteq \theta_{\af'}(B)$. So,  $\bt=\af_{1,i}$ for some $i$ and $\theta_{\af_{i+1,|\af| }}(B) =\theta_{\af'}(B) \in \CF $ since $A \in \CF$
 and $\CF$ is a filter.  

Conversely, let $p=(\af_{1,n}, A, \af_{1,n})$ be such that $\theta_{\af_{i+1, |\af|}}(A) \in \CF$. We then have $p \in \bigcup_{A \in \CF} \uparrow(\af, A, \af) $ since  $(\af, \theta_{\af_{i+1, |\af|}}(A), \af ) \leq (\af_{1,i}, A, \af_{1,i})$.

Now, we show that $\xi$ is a filter: it clearly satisfies (i),(ii) in Definition \ref{filter}. For (iii), let $p, q \in \xi$ and choose $A$ and $B$ in $\CF$
 such that $p \geq (\af, A, \af)$ and $q \geq (\af, B, \af)$. Since $A \cap B \in \CF$,  we see that $r:=(\af, A \cap B, \af) \in \xi$. Then $r \leq p$ and $r \leq q$. 
\end{proof}
\begin{remark} Let $\af \in \CL$. If $\CF$ is a filter in $\CI_\af$ such that $\CF \cap \CR_\af \neq \emptyset $, then the  filter  $\xi$  in $E(S)$ with largest word $\af$ given in Lemma \ref{word to finite}  is 
\begin{align*}\xi =\{(\emptyset, A, \emptyset): \theta_{\af}(A) \in \CF\} \cup \{(\af, A, \af): A \in \CF\}.
\end{align*} 
If $\CF$ is a filter in $\CI_\af$ such that $\CF \cap \CR_\af = \emptyset $, then 
\begin{align*}\xi = \{(\af, A, \af): A \in \CF\}.
\end{align*}
	\end{remark}

\begin{prop}\label{complete family gives filter}(\cite[Proposition 4.7]{BCM1}) Let $\af \in \CW^{\leq \infty}$ and $\{\CF_n\}_{0 \leq n \leq |\af|}$  be an   admissible family  for $\af$. Define
$$\xi= \bigcup_{n=0}^{|\af|} \bigcup_{A \in \CF_n} \uparrow (\af_{1,n}, A, \af_{1,n}).$$
Then $\xi$ is a filter in $E(S)$. 
\end{prop}

\begin{proof} By Lemma \ref{word to finite}, for each $n \geq 0$, 
$$\xi^n:=\bigcup_{A \in \CF_n} \uparrow(\af_{1,n}, A, \af_{1,n}) $$ is a filter in E(S)
 (except if $\CF_0=\emptyset$). We claim that $\xi^n \subseteq \xi^{n+1}$ for each $0 \leq n < |\af|$. Let $p \in \xi^n$. Then $(\af_{1,n}, A, \af_{1,n}) \leq p$ for some $A \in \CF_n$.
Since the family is admissible, we have $\theta_{\af_{n+1}}(A) \in \CF_{n+1}$, and hence, $(\af_{1,n+1}, \theta_{\af_{n+1}}(A), \af_{1,n+1}) \in \xi^{n+1}$.
Since $\xi^{n+1}$ is filter and $ (\af_{1,n+1}, \theta_{\af_{n+1}}(A), \af_{1,n+1})  \leq (\af_{1,n}, A, \af_{1,n}) \leq p$, we see that $p \in \xi_{n+1}$. This claim says that the filters $\xi^n$ are nested, from where we conclude that their union is a filter.  
\end{proof}

\begin{cor}\label{complete to finite} Let $\af \in \CW^{*}$ and $\{\CF_n\}_{0 \leq n \leq |\af|}$  be a complete family  for $\af$. If  $\xi$ is  the filter given by Proposition \ref{complete family gives filter}, then 
\begin{align*} \xi = \bigcup_{A \in \CF_{|\af|}} \uparrow(\af, A, \af).
\end{align*}

\end{cor}

\begin{proof} Put $\eta:=\bigcup_{A \in \CF_{|\af|}} \uparrow(\af, A, \af) $. It then is   clear that $\eta \subseteq \xi$. If $p \in \xi$, then $(\af_{1,n}, A, \af_{1,n}) \leq p$ for some $0 \leq n \leq |\af|$ and $A \in \CF_n$. Since $(\af, \theta_{\af_{n+1, |\af|}}(A), \af) \leq (\af_{1,n}, A, \af_{1,n}) \leq p$, we have  $p \in \eta$. Thus, $\xi=\eta$.
\end{proof}
 In Proposition \ref{complete family gives filter}, we construct a filter in $E(S)$ from a complete family. We now go in the opposite direction.

\begin{prop}\label{filter gives complete family}
Let $\xi$ be a filter in $E(S)$ and $\alpha$ its associated word. For each $0 \leq n \leq |\af|$ (understanding that $0 \leq n \leq |\af|$ means $0 \leq n < \infty$ when $\af \in \CW^\infty$) define
$$\CF_n=\{A \in \CI_{\af_{1,n}}: (\af_{1,n}, A, \af_{1,n}) \in \xi\},$$
then $\{\CF_n\}_{0\leq n\leq |\alpha|}$ is a complete family for $\af$.
\end{prop}

\begin{proof}
If $|\af|=0$, then it follows that $\CF_0$ is a filter by Lemma~\ref{lem:order}, and hence $\{\CF_n\}_{0\leq n\leq |\alpha|}$ is a complete family for $\af$.

Suppose now that $|\af|>0$ and let $0 < n \leq |\af|$. 
  It is easy to see that $\CF_n$ is a filter in $\CI_{\af_{1,n}}$  if $\CF_n$ is nonempty. We claim that $\CF_n \neq \emptyset$. We can choose $m \geq n$ such that there is $p=(\af_{1,m}, A, \af_{1,m}) \in \xi$ for some $A \in \CI_{\af_{1,m}}$ (such $m$ always exits by the construction of the associated word in the proof of Lemma~\ref{lem:assoc.word}). Since $A \in \CI_{\af_{1,m}}$, $A \subseteq \theta_{\af_{2,m}}(B)$ for some $B \in \CI_{\af_1}$. We then see that $p=(\af_{1,m}, A, \af_{1,m}) \leq (\af_{1,n},   \theta_{\af_{2,n}}(B) , \af_{1,n}) $ since $A \subseteq \theta_{\af_{n+1,m}}(\theta_{\af_{2,n}}(B))=\theta_{\af_{2,m}}(B)$. Thus, we have   $(\af_{1,n},   \theta_{\af_{2,n}}(B) , \af_{1,n}) \in \xi$ since $\xi$ is a filter. This shows that $\theta_{\af_{2,n}}(B) \in \CF_n$.

 It remains to show that the family $\{\CF_n\}_{0\leq n\leq |\alpha|}$ is complete for $\af$. 
First, fix $0<n<|\af|$ and put $$\CG=\{A \in \CI_{\af_{1,n}}: \theta_{\af_{n+1}}(A) \in \CF_{n+1}\}.$$ Let $A \in \CF_n$. Note first that $A \subseteq \theta_{\af_{2,n}}(B)$ for some $B \in \CI_{\af_1}$. 
 Choose $C \in \CF_{n+1} (\neq \emptyset)$. Since $p=(\af_{1,n}, A, \af_{1,n})$ and $q=(\af_{1,n+1}, C, \af_{1,n+1})$ belong to the filter $\xi$, $pq=(\af_{1,n+1}, \theta_{\af_{n+1}}(A)\cap C, \af_{1,n+1}) \in \xi$. 
This says that $ \emptyset \neq \theta_{\af_{n+1}}(A)\cap C \in \CF_{n+1}$. 
We also have  $\theta_{\af_{n+1}}(A) \in \CI_{\af_{1,n+1}}$ since 
 $\theta_{\af_{n+1}}(A) \subseteq \theta_{\af_{2,n+1}}(B)$ for $B \in \CI_{\af_1}$. 
It thus follows that  $\theta_{\af_{n+1}}(A)  \in \CF_{n+1}$
since $\CF_{n+1}$ is a filter in $\CI_{\af_{1,n+1}}$. Therefore, $A \in \CG$. 
On the other hand, let $A \in \CG$ and observe that $p=(\af_{1,n+1}, \theta_{\af_{n+1}}(A), \af_{1,n+1}) \in \xi$. Clearly, $p \leq q=(\af_{1,n}, A, \af_{1,n})$
Thus, $q \in \xi$ since $\xi$ is a filter. So, $A \in \CF_n$. 

For $n=0$, we observe that $\CF_0 \neq \emptyset$ if and only if  $\CF_1 \cap \CR_{\af_1} \neq \emptyset$. In fact, if $\CF_1 \cap \CR_{\af_1} \neq \emptyset$, then $(\af_1, \theta_{\af_1}(A), \af_1) \in \xi$ for some $A \in \CB$. Since $(\af_1, \theta_{\af_1}(A), \af_1) \leq (\emptyset, A, \emptyset)$, we have $ (\emptyset, A, \emptyset) \in \xi$, and hence, $A \in \CF_0$. On the other hand, if $A\in\CF_0$, $\theta_{\af_1}(A)\in\CF_1\cap \CR_{\af_1}$.

Using the same arguments, one also can show that if $\CF_1 \cap \CR_{\af_1}\neq \emptyset$, then 
$\CF_0=\{A \in \CB: \theta_{\af_1}(A) \in \CF_1\}.$ And if $\CF_1 \cap \CR_{\af_1} = \emptyset$, then $\CF_0=\emptyset=\{A \in \CB: \theta_{\af_1}(A) \in \CF_1\}$.

Hence, $\{\CF_n\}_{0\leq n\leq|\af|}$ is a complete family for $\af$. 
\end{proof}

Combining the above results, we have the main result of this section. 

\begin{thm}\label{filter-bijective-pair}(\cite[Theorem 4.13]{BCM1}) Let $(\CB, \CL,\theta, \CI_\af)$ be a generalized Boolean dynamical system and $S$ be its associated inverse semigroup. Then there is a bijective correspondence between filters in $E(S)$ and pairs $(\af, \{ \CF_n\}_{0 \leq n \leq |\af|})$, where $\af \in \CW^{\leq \infty}$ and $\{\CF_n\}_{0 \leq n \leq |\af|}$ is a complete family for $\af$.
\end{thm}

\begin{proof} Let $\xi$ be a filter in $E(S)$ with the associated word $\af \in \CW^{\leq \infty}$ and consider $(\af, \{\CF_n\}_{0 \leq n \leq |\af|})$ as in Proposition \ref{filter gives complete family}. Let $\eta=\bigcup_{n=0}^{|\af|} \bigcup_{A \in \CF_n} \uparrow (\af_{1,n}, A, \af_{1,n}) $ be the filter given by Proposition \ref{complete family gives filter} from $(\af, \{\CF_n\}_{0 \leq n \leq |\af|})$.  We claim that $\xi=\eta$.  If $p \in \xi$, then $p=(\af_{1,n}, A, \af_{1,n})$ for some $n \geq 0$ and some $A \in \CI_{\af_{1,n}}$ by Lemma \ref{lem:assoc.word}. Then clearly, $p \in \eta$ by the definition of $\eta$. So,  $\xi \subseteq \eta$. For the reverse inclusion, let $p \in \eta$ and choose $n \geq 0$ and $A \in \CF_n$ such that $(\af_{1,n}, A, \af_{1,n}) \leq p$. Since $(\af_{1,n}, A, \af_{1,n}) \in \xi$, we have $p \in \xi$. Thus, $\eta \subseteq \xi$.

Conversely, let $\af \in \CW^{\leq \infty}$ and $\{\CF_n\}_{0 \leq n \leq |\af|}$ be a complete family for $\af$, consider $\xi$ constructed from them as in Proposition \ref{complete family gives filter} and let $(\bt, \{\CG_n\}_{0 \leq n \leq |\af|})$ be given from $\xi$ by Proposition \ref{filter gives complete family}. We show that $(\af, \{\CF_n\}_{0 \leq n \leq |\af|})=(\bt, \{\CG_n\}_{0 \leq n \leq |\af|})$. By construction we have $\af=\bt$.  Fix $m \geq 0$ and let $A \in \CF_m$. Then $(\af_{1,m}, A, \af_{1,m}) \in \xi$, hence $A \in \CG_m$.

On the other hand, if $B \in \CG_m$, then $(\af_{1,m}, B, \af_{1,m}) \in \xi$. So, there exists $k \geq m$ and $C \in \CF_k$ such that $(\af_{1,k}, C, \af_{1,k}) \leq (\af_{1,m}, B, \af_{1,m})$. This says that $C \subseteq \theta_{\af_{m+1,k}}(B)$, and hence, $ \theta_{\af_{m+1,k}}(B) \in \CF_k$.
Since $\{\CF_n\}_{0 \leq n \leq |\af|}$ is complete for $\af$, we see that $B \in \CF_m$. 
\end{proof}

Filters in $E(S)$ are classified in two types: filters are of {\it finite type} if they are associated with pairs $(\af, \{\CF_n\}_{0 \leq n \leq |\af|})$ for which $|\af|< \infty$, and of {\it infinite type} otherwise.

A filter $\xi$ in $E(S)$ with associated $\af \in \CW^{\leq \infty}$ is sometimes denoted by $\xi^\af$ to stress $\af$; in addition, the filters in the complete family associated with $\xi^\af$ will be  denoted by  $\xi^\af_n$ (or simply $\xi_n$). Specifically, 
$$\xi^\af_n=\{A \in \CI_{\af_{1,n}}: (\af_{1,n}, A, \af_{1,n}) \in\xi^\af\}$$
and the family $\{\xi_n^\af\}_{0 \leq n \leq |\af|}$ satisfies 
$$\xi^\af_n =\{A \in \CI_{\af_{1,n}}: \theta_{\af_{n+1,m}}(A) \in \xi_m^\af\}$$
for all $0 \leq n < m \leq |\af|$.

\begin{notations} For future reference, we denote by $\mathsf{F}$ the set of all filters in $E(S)$ and by  $\mathsf{F}_{\af}$ the  set of all filters  in $E(S)$ associated with  $\af \in \CW^{\leq \infty}$.
\end{notations}

Lastly, we remark in below that a filter $\xi$  of finite type is completely determined by a finite word $\af \in \CW^*$ and a filter $\CF$ in $\CI_{\af}$. 

\begin{remark} There is a bijective correspondence between filters in $E(S)$ of finite type and pairs $(\af, \CF)$, where $\af$ is a finite word and $\CF$ is a filter in $\CI_\af$. 
\end{remark}

\begin{proof}
If $\xi=\xi^\af$ is a filter of finite type, we then have a filter $\CF_{|\af|}=\{A \in \CI_\af: (\af, A, \af) \in \xi\}$ in $\CI_\af$ from Proposition \ref{filter gives complete family}. Let $\eta=\bigcup_{A \in \CF_{|\af|}} \uparrow(\af, A, \af)$ be the filter given by Corollary \ref{complete to finite}. Then clearly, $\eta \subseteq \xi$. On the other hand, let $p=(\bt,B,\bt) \in \xi$ and choose $A \in \CI_\af$ such that $q=(\af, A, \af) \in \xi$. By Lemma \ref{lem:assoc.word},  $\bt$ must be an initial segment of $\af$, say $\af=\bt\af'$. Since $pq=(\af, A \cap \theta_{\af'}(B), \af) \in \xi$, we have $A \cap \theta_{\af'}(B) \in \CF_{|\af|}$. Thus, $p \in \uparrow(\af, A \cap \theta_{\af'}(B), \af) \subseteq \eta$.

Give a pair $(\af, \CF)$  such that $\af$ is a finite word and $\CF$ is a filter in $\CI_\af$, let $\xi=\bigcup_{A \in \CF} \uparrow(\af, A, \af)$ be the filter given by Lemma \ref{word to finite}. Consider  $(\af, \{\CF_n\}_{0 \leq n \leq |\af|})$ as in Proposition \ref{filter gives complete family}. Then, clearly, $\CF \subseteq \CF_{|\af|}$.
If $A \in \CF_{|\af|}$, then $(\af, A, \af) \in \xi$. So, $(\af, B, \af) \leq (\af, A, \af)$ for some $B \in \CF$. Hence, $A \in \CF$. Thus, $\CF=\CF_{|\af|}$.
\end{proof}

\subsection{Ultrafilters in \texorpdfstring{$E(S)$}{E(S)}} We next describe the ultrafilters in $E(S)$. We start with an easy, but useful proposition. 
\begin{prop}\label{prop 1}(\cite[Proposition 5.1]{BCM1}) Let $\xi^\af, \eta^\bt$ be filters in $E(S)$. Then
\begin{enumerate} 
                  
\item If  $\af$ is a finite word, then  $\xi^\af \subseteq \eta^\bt$  if and only if $\af$ is a beginning of (or equal to) $\bt$ and $\xi_{|\af|} \subseteq \eta_{|\af|}$. 
\end{enumerate}
\end{prop}

\begin{proof}(1) Since $\xi_n=\{A \in \CI_{\af_{1,n}}: (\af_{1,n}, A, \af_{1,n}) \in \xi^\af\}$ and $\eta_n=\{A \in \CI_{\bt_{1,n}}: (\bt_{1,n}, A, \bt_{1,n}) \in \eta^\bt\}$, the results is clear.

(2) It is also clear since  $\xi^\af= \bigcup_{A \in \xi_{|\af|}} \uparrow (\af, A, \af)$.
\end{proof}

\begin{prop}(\cite[Proposition 5.4]{BCM1}) Let $\xi^\af $ be a filter of finite type in $E(S)$. Then $\xi^\af$ is an ultrafilter if and only if  $\xi_{|\af|} $ is an ultrafilter in $\CI_\af$  and for each $\bt \in \CL$, there exists $A \in  \xi_{|\af|}$ such that $\theta_\bt(A)=\emptyset$. 
\end{prop}

\begin{proof}($\Rightarrow$) Suppose  $\xi^\af$ is an ultrafilter and let $\CF $ be a filter in $\CI_\af$ such that $\xi_{|\af|} \subseteq \CF$. Then the filter $\eta$ associated with $(\af,\CF )$ is given by 
$\eta=\bigcup_{A \in \CF} \uparrow (\af, A, \af)$, so $\xi^\af \subseteq \eta$ by Proposition \ref{prop 1}. Thus, 
 $\xi^\af = \eta$. So, $\CF=\xi_{|\af|}$, whence $\xi_{|\af|}$ is an ultrafilter. 

Now, assume to the contrary that there exists $\bt \in \CL$ such that $\theta_\bt(A) \neq \emptyset$ for all $A \in \xi_{|\af|}$.
Consider $Y=\{\theta_\bt(A): A \in \xi_{|\af|}\} \subset \CI_{\af \bt}$.  Since $\emptyset \notin Y$ and $Y$ is closed under intersections, 
 $\uparrow_{\CI_{\af \bt}}Y$ is a filter in $\CI_{\af \bt}$ by \cite[Proposition 2.16]{BCM1}. Consider then the filter $\eta$ associated with the pair $(\af \bt, \uparrow_{\CI_{\af \bt}}Y)$. Since $\eta_{|\af|}= \{A \in \CI_\af : \theta_\bt(A) \in \uparrow_{\CI_{\af \bt}}Y\} \supseteq \xi_{|\af|} $, we would then have $\xi^\af \subsetneq \eta$. It contradicts the fact that $\xi^\af$ is an ultrafilter. 
 
 ($\Leftarrow$) Suppose $\xi^\af$ is not an ultrafilter, so that there is a filter $\eta^\bt$ such that $\xi^\af \subsetneq \eta^\bt$. Then by Proposition \ref{prop 1}, $\af$ is a beginning of $\bt$. If $\bt=\af$, then $\xi_{|\af|} \subsetneqq \eta_{|\af|}$ by Proposition \ref{prop 1}(2), hence $\xi_{|\af|}$ is not an ultrafilter in $\CI_\af$. If $\bt \neq \af$, then $\bt=\af\gm$ with $\gm \neq \emptyset$.
 Let $b$ be the first letter of $\gm$ and observe that $\xi_{|\af|} \subseteq \eta_{|\af|}=\{A \in \CI_\af: \theta_b(A) \in \eta_{|\af|+1}\}$.
  Since $\{\theta_b(A): A \in \xi_{|\af|}\} \subseteq \eta_{|\af|+1}$ and $\eta_{|\af|+1}$ is a filter, we have $\theta_b(A)\neq \emptyset$ for all $A \in \xi_{|\af|}$, a contradiction.    
\end{proof}

If a filter in a complete family is an ultrafilter, then all filters in the family coming before it are also ultrafilters as we see in below.

\begin{lem}\label{aa}Let $\xi^\af$ be a filter in $E(S)$. If $\xi_n$ is an ultrafilter for $n$ with $0 < n \leq |\af|$, then $\xi_m$ is an ultrafilter for every $0 < m < n$.  If furthermore $\xi_0$ is nonempty, then $\xi_0$ is also ultrafilter.
\end{lem}

\begin{proof}
By Remark \ref{rmk:complete.family}, we have that $\xi_m=\{A \in \CI_{\af_{1,m}}: \theta_{\af_{m+1,n}}(A) \in \xi_n\}$ for $0 < m < n$. Suppose that $A,B\in \CI_{\af_{1,m}}$ are such that $A\cup B\in \xi_m$. Then, $\theta_{\af_{m+1,n}}(A)\cup \theta_{\af_{m+1,n}}(B)=\theta_{\af_{m+1,n}}(A\cup B)\in\xi_n$. Since $\xi_n$ is an ultrafilter, we have that $\theta_{\af_{m+1,n}}(A)\in\xi_n$ or $\theta_{\af_{m+1,n}}(B)$, and hence $A\in\xi_m$ or $B\in\xi_m$. It follows that $\xi_m$ is an ultarfilter. The case $n=0$  is analogous.
\end{proof}

\begin{prop}\label{ultrafilter of infinite type}Let $\xi^\af$ be a filter of infinite type in $E(S)$. Then $\xi^\af$ is an ultrafilter if and only if $\xi_n$ is an ultrafilter in $\CI_{\af_{1,n}}$ for every $n >0$ and $\xi_0$ is either an ultrafilter in $\CB$ or the empty set.
\end{prop}

\begin{proof} ($\Rightarrow$) Assume  $\xi^\af$ is an ultrafilter and fix $n >0$. Let $A \in \xi_n$. Suppose that $A=B \cup B'$ for $B, B' \in \CI_{\af_{1,n}}$. If $B \notin \xi_n$, then $p=(\af_{1,n}, B, \af_{1,n}) \notin \xi^\af$. Since $\xi^\af$ is an ultrafilter, there exists $q=(\af_{1,m}, C, \af_{1,m}) \in \xi^\af$ such that $pq=0$ by \cite[Lemma 12.3]{Ex1}. There are three cases to consider: if $m=n$, then $B \cap C=\emptyset$ and $C \in \xi_n$. 
Since $A, C \in \xi_n$, we see that  $\emptyset \neq A \cap C=B' \cap C \in \xi_n$. It thus follows that $B' \in \xi_n$.  
If $m < n$, then $pq=(\af_{1,n}, \theta_{\af_{m+1,n}}(C)\cap B, \af_{1,n})=0$, which implies that $\theta_{\af_{m+1,n}}(C) \cap B =\emptyset$. 
Then again,   since $A,\theta_{\af_{m+1,n}}(C) \in \xi_n $, we see that
 $\emptyset \neq A \cap \theta_{\af_{m+1,n}}(C) =B' \cap  \theta_{\af_{m+1,n}}(C) \in \xi_n$. Hence, $B' \in \xi_n$.
For the case $m>n$, we first note that $\theta_{\af_{n+1,m}}(B) \notin \xi_m$ and $\theta_{\af_{n+1,m}}(A) \in \xi_m$ since the family $\{\xi_n\}_{n \geq 0}$ is complete. 
Now consider $\tilde{p}=(\af_{1,m}, \theta_{\af_{n+1,m}}(B) ,\af_{1,m}) (\notin \xi^\af)$ instead of $p$. One then again can see that there exists $C \in \xi_m$ such that $\theta_{\af_{n+1,m}}(B) \cap C=\emptyset$.
 Then, $ \emptyset \neq  \theta_{\af_{n+1,m}}(A) \cap C = \theta_{\af_{n+1,m}}(B') \cap C \in \xi_m $. Thus, we have that $\theta_{\af_{n+1,m}}(B')  \in \xi_m$, and hence, $B' \in \xi_n$. 

One also can  show using the same arguments that if  $B' \notin \xi_n$, then $B \in \xi_n$. Thus,  $\xi_n$ is an ultrafilter.

The case $n=0$ follows from Proposition \ref{aa}.

($\Leftarrow$) Let $\eta^\bt \in \mathsf{F}$ be such that $\xi^\af \subseteq \eta^\bt$.
Then $\af=\bt$ and $\xi_n \subseteq \eta_n$ for all $n \geq 0$ by Proposition \ref{prop 1}. Since $\xi_n$ is an ultrafilter for all $n \geq 0 $ (except $\xi_0=\emptyset$), we have  $\xi_n = \eta_n$ for all $n \geq 0$ and therefore  $\xi^\af=\eta^\bt$.
\end{proof}

The results of this section give the following characterization of the ultrafilters in $E(S)$.

\begin{thm}\label{char:ultrafilters}(\cite[Theorem 5.10]{BCM1})  Let $(\CB, \CL,\theta, \CI_\af)$ be a generalized Boolean dynamical system, and $S$ be its associated inverse semigroup. Then the ultrafilters in $E(S)$ are:
\begin{enumerate}
\item[(i)] The filters of finite type $\xi^\af$ such that $\xi_{|\af|}$ is an ultrafilters in $\CI_\af$
and for each $\bt \in \CL$ there exists $A \in \xi_{|\af|}$ such that $\theta_\bt(A)=\emptyset$.
\item[(ii)] The filters of infinite type $\xi^\af$ such that $\xi_{n}$ is an ultrafilter for every $n > 0$ and $\xi_0$ is either an ultrafilter or the empty set. 
\end{enumerate}
\end{thm}

\subsection{Tight filters in \texorpdfstring{$E(S)$}{E(S)}}\label{section:tight.filters}

 We denote by $\mathsf{T}$  the set of tight filters on $E:=E(S)$ and we equip  $\mathsf{T}$  with the topology induced from the topology of pointwise convergence of character, via the bijection between tight characters and tight filters given in subsection \hyperref[Inverse semigroups]{2.2}(4). 
Note then that $\mathsf{T}$ is (homeomorphic to) the tight spectrum $\hat{E}_{tight}$ of $E$. We in this section give a complete characterization of the tight filters in $E$.

\begin{prop}\label{tight gives ultrafilter} Let $(\CB, \CL,\theta, \CI_\af)$ be a generalized Boolean dynamical system,  $S$ be its associated inverse semigroup and let  $\xi^\af$ be a tight filter in $E(S)$. Then, for every $0 \leq n \leq |\af|$, $\xi_n$ is an ultrafilter in $\CI_{\af_{1,n}}$ ($\xi_0$ may also be the empty set).
\end{prop}
 
 \begin{proof} Fix $0 \leq n \leq |\af|$ and suppose $\xi_n$ is not an ultrafilter.  There  then exists
 $A \in \xi_n$ such that $A =B \cup B' $ for $B \notin \xi_n$ and $ B' \notin \xi_n$. 
 
  Let $p=(\af_{1,n}, A, \af_{1,n})$ and $Z=\{(\af_{1,n}, B, \af_{1,n}),(\af_{1,n}, B', \af_{1,n})\}$. Then $p \in \xi^\af$ and $Z$ is a cover for $p$. Indeed, for every $q \leq p$, we must have $q=(\af_{1,n}\bt, C, \af_{1,n}\bt)$ for some $\bt \in \CL^*$ and some $C \in \CI_{\af_{1,n}\bt}$. Since $C \subseteq \theta_\bt(A)=\theta_\bt(B) \cup \theta_\bt(B')$, then $C \cap \theta_\bt(B) \neq \emptyset$ or $C \cap \theta_\bt(B') \neq \emptyset$, and hence, it follows that $q(\af_{1,n}, B, \af_{1,n}) \neq 0$ or $q(\af_{1,n}, B', \af_{1,n}) \neq 0$. It says that $Z$ is a cover for $p$. Since $Z \cap \xi^\af = \emptyset$, we see that $\xi^\af$ is not tight by Corollary \ref{tight filter:how to check}. This is a contradiction.  
 \end{proof}
 We say that a Boolean  homomorphism $\phi:\CB\to \CB'$ is {\em proper} if for every $A\in \CB'$, there exists $B\in\CB$ such that $A\subseteq \phi(B)$.
 
\begin{lem}
 	Given $\af\in\CL^{\geq 1}$ and $\bt\in\CL^*$, if $\CI_{\af\bt}\neq\{\emptyset\}$, then $\CI_{\af}\neq\{\emptyset\}$ and the map $\theta_\bt : \CI_\af \to \CI_{\af\bt}$ is a proper Boolean  homomorphism.
 \end{lem}

\begin{proof}
	Suppose that $\CI_{\af\bt}\neq\{\emptyset\}$ and that $A\in \CI_{\af\bt}\setminus\{\emptyset\}$. By Definition~\ref{def of I}, there exists $B\in\CI_{\af_1}$ such that $A\subseteq \theta_{\af_{2,n}\beta}(B)$. Since $A\neq\emptyset$ and $\theta_{\af_{2,n}\beta}(\emptyset)=\emptyset$ since $\theta_{\af_{2,n}\beta}$ is an action, we have that $B\neq\emptyset$. And since $\theta_{\af_{2,n}\beta}=\theta_{\beta}\circ\theta_{\af_{2,n}}$, we also have $\theta_{\af_{2,n}}(B)\neq\emptyset$. Hence $\theta_{\af_{2,n}}(B)\in\CI_\af\setminus\{\emptyset\}$.
	
	Notice that the above argument also shows that $\theta_\bt$ is proper since $A\subseteq \theta_\bt(\theta_{\af_{2,n}}(B))$ and $\theta_{\af_{2,n}}(B)\in\CI_\af$.
\end{proof}

From now on, for each  $\af \in  \CW^*$, 
 we write  $X_\af$ instead of $\widehat{\CI_\af}$ for the set of all ultrafilters in $\CI_\af$ to match our notations with \cite{BCM2}. Note that $X_\emptyset$ denotes the set of ultrafilters in $\CI_\emptyset=\CB$.  
  For $A \in \CI_\af$, we let $$Z(\af, A):=\{\CF \in X_\af: A \in \CF\}$$ and equip $X_\af$
 with the topology generated by $\{Z(\af, A): A\in\CI_\af\}$.
     We further define 
 $$X_\af^{sink}=\{\CF \in X_\af : \forall \bt \in \CL, \exists A \in \CF ~\text{such that}~ \theta_\bt(A) =\emptyset\}.$$

In what follows, we also need to consider the set $X_\emptyset\cup\{\emptyset\}$ with a suitable topology. If $\CB$ is unital, the topology is such that $\{\emptyset\}$ is an isolated point. If $\CB$ is not unital, then $\emptyset$ plays the role of the point at infinity in the one-point compactification of $X_\emptyset$. The following lemma describe convergence to $\emptyset$ in a unifying way.

\begin{lem}\label{lem:conv.X.empty}
A net $\{\CF_\lambda\}_{\lambda\in \Lambda}$ converges to $\emptyset$ in $X_\emptyset\cup\{\emptyset\}$ if and only if for all $A\in\CB$, there exists $\lambda_0\in\Lambda$ such that $A\notin\CF_\lambda$ for all $\lambda\geq\lambda_0$.
\end{lem}

\begin{proof}
Suppose first that $\{\CF_\lambda\}_{\lambda\in \Lambda}$ converges to $\emptyset$. In the unital case, since $\emptyset$ is isolated, the net is eventually constant and the result follows. In the non-unital case, given $A\in\CB$, we have that $Z(A)$ is compact in $X_\emptyset$ so that $(X_\emptyset\setminus Z(A))\cup \{\emptyset\}$ is an open neighborhood of $\emptyset$. Then, there exists $\lambda_0$ such that $\CF_\lambda\in (X_\emptyset\setminus Z(A))\cup \{\emptyset\}$ for all $\lambda\geq \lambda_0$, the latter implying that $A\notin\CF_\lambda$ for all $\lambda\geq\lambda_0$.

Let now $\{\CF_\lambda\}_{\lambda\in \Lambda}$ be a net such that for all $A\in\CB$, there exists $\lambda_0\in\Lambda$ such that $A\notin\CF_\lambda$ for all $\lambda\geq\lambda_0$. In the unital case, if $I$ is the unit of $\CB$, then all ultrafilters contain $I$. The property above then implies that the net $\{\CF_\lambda\}_{\lambda\in \Lambda}$ is eventually constant and equals to the empty set, and the result follows. Now suppose that $\CB$ is not unital and let $(X_\emptyset\setminus K)\cup\{\emptyset\}$ be a open neighborhood of $\emptyset$, where $K$ is a compact subset of $X_\emptyset$. By compactness of $K$, there exists $A_1,\ldots,A_n$ such that $K\subseteq Z(A_1)\cup\cdots\cup Z(A_n)=Z(A_1\cup\cdots\cup A_n)$, where the equality follows from the fact that every ultrafilter is prime. If we let $A=A_1\cup\cdots\cup A_n$ and take $\lambda_0$ as in the hypothesis, then $\CF_\lambda\in (X_\emptyset\setminus Z(A))\cup\{\emptyset\}\subseteq(X_\emptyset\setminus K)\cup\{\emptyset\}$ for all $\lambda\geq\lambda_0$. Hence the net $\{\CF_\lambda\}_{\lambda\in \Lambda}$ converges to $\emptyset$.
\end{proof}

Given $\af, \bt \in \CW^{\geq 1}$, since the  action $$\theta_\bt : \CI_\af \to \CI_{\af\bt}$$ is a proper Boolean  homomorphism, there is its dual morphism 
$$f_{\af[\bt]}: X_{\af\bt} \to X_\af$$ given by 
 $f_{\af[\bt]}(\CF)=\{A \in \CI_\af: \theta_\bt(A) \in \CF\}$,
When $\af=\emptyset$, if $\CF \in X_\bt$, then $\{A \in \CB: \theta_\bt(A) \in \CF\}$ is either an ultrafilter in $\CI_\emptyset(=\CB)$ or  the empty set. We can therefore consider 
$f_{\emptyset[\bt]}: X_\bt \to X_\emptyset \cup \{\emptyset\}.$ Notice that $f_{\af[\bt\gm]}=f_{\af[\bt]}\circ f_{\af\bt[\gm]}$ for all $\af\in\CW^{*}$ and $\beta,\gamma\in\CW^{\geq 1}$ such that $\alpha\beta\gamma\in\CW^{\geq 1}$.
If $\af\neq\emptyset$, we have that $f_{\af[\bt]}$ is a continuous function since $f_{\af[\bt]}^{-1}(Z(\af,A))=Z(\af\bt, \theta_\bt(A))$ for each $A \in \CI_\af$. For $\af=\emptyset$, we prove the continuity in the following lemma.

\begin{lem}\label{lem:f.empty.continuous}
The function $f_{\emptyset[\bt]}: X_\bt \to X_\emptyset \cup \{\emptyset\}$ is continuous.
\end{lem}

\begin{proof}
Let $\{\CF_\lambda\}_{\lambda\in \Lambda}$ be a net converging to $\CF$ in $X_\bt$. We prove that $\{f_{\emptyset[\bt]}(\CF_\lambda)\}_{\lambda\in \Lambda}$ converges to $f_{\emptyset[\bt]}(\CF)$. If $f_{\emptyset[\bt]}(\CF)\neq\emptyset$, then for $A\in\CB$, we have that
\begin{align*}
    A\in f_{\emptyset[\bt]}(\CF) & \Leftrightarrow \theta_\bt(A)\in\CF \\
    & \Leftrightarrow \exists\lambda_0:\forall\lambda\geq\lambda_0,\ \theta_\bt(A)\in\CF_\lambda \\
    & \Leftrightarrow \exists\lambda_0:\forall\lambda\geq\lambda_0,\ A\in f_{\emptyset[\bt]}(\CF_\lambda),
\end{align*}
from where the desired convergence follows.

If $f_{\emptyset[\bt]}(\CF)=\emptyset$, then for $A\in\CB$, we have that $\theta_\bt(A)\notin\CF$. Hence, there exists $\lambda_0$ such that for all $\lambda\geq\lambda_0$, $\theta_\bt(A)\notin\CF_\lambda$, that is, $A\notin f_{\emptyset[\bt]}(\CF_\lambda)$. The desired convergence then follows from Lemma \ref{lem:conv.X.empty}.
\end{proof}

\begin{remark}
The above lemma fixes a small gap in \cite{BCM1}, where the topology on $X_\emptyset\cup\{\emptyset\}$ was not described and it was not given a proof that $f_{\emptyset[\bt]}$ as defined above is continuous in the context of labeled spaces.
\end{remark}

\begin{remark}\label{rmk:complete family using f}Let  $\{\CF_n\}_{n \geq 0}$ be a complete family of ultrafilters for $\af \in \CW^{\leq\infty}$. Then the condition defining complete family is equivalent to
$$\CF_n=f_{\af_{1,n}[\af_{n+1}]}(\CF_{n+1})$$
for all $1\leq n<|\alpha|$. Also, it is easy to see that 
$$\CF_n=f_{\af_{1,n}[\af_{n+1,m}]}(\CF_m)$$
for all $1\leq n<m\leq|\af|$.
\end{remark}

The next result classifies the tight filters of finite type. To prove it,  the fact that   $\hat E_{\infty}$ is dense in  $\hat E_{tight}$ and the continuity of the map $f_{\af[\bt]}$ are mainly used. Since the idea of proof is same with \cite[Proposition 6.4]{BCM1}, it is omitted.

\begin{prop}\label{char 1:tight filter of finite type}(\cite[Proposition 6.4]{BCM1}) Let $(\CB, \CL,\theta, \CI_\af)$ be a generalized Boolean dynamical system, $S$ be its associated inverse semigroup and
  $\xi^\af$ be a  filter of finite type in $E(S)$. Then $\xi^\af$ is a tight filter if and only if $\xi_{|\af|}$ is an ultrafilter and at least one of the following condition hold:
\begin{enumerate}
\item There is a net $\{\CF_\ld\}_{\ld \in \Lambda} \subseteq X_\af^{sink}$ converging to $\xi_{|\af|}$.
\item There is a net $\{(t_\ld, \CF_\ld)\}_{\ld \in \Lambda}$, where $t_\ld \in \CL$ and $\CF_\ld \in X_{\af t_\ld}$ for each $\ld \in \Lambda$, such that $\{f_{\af[t_\ld]}(\CF_\ld)\}_{\ld \in \Lambda}$ converges to $\xi_{|\af|}$ and for every $b  \in \CL$ there is $\ld_b \in \Lambda$ such that $t_\ld \neq b$ for all $\ld \geq \ld_b$.
\end{enumerate}
\end{prop}

 Combining  Proposition \ref{char 1:tight filter of finite type},  Proposition \ref{tight gives ultrafilter} and Theorem  \ref{char:ultrafilters}, we have the following description of the tight filters in $E$.

\begin{thm}\label{char:tight}(\cite[Theorem 6.7]{BCM1})  Let $(\CB, \CL,\theta, \CI_\af)$ be a generalized Boolean dynamical system and $S$ be its associated inverse semigroup. Then the tight filters in $E(S)$ are :
\begin{enumerate}
\item[(i)] The ultrafilters of infinite type.
\item[(ii)] The filters of finite type $\xi^\af$ such that $\xi_{|\af|}$ is an ultrafilter and   $A\notin\CB_{reg}$ for all $A \in \xi_{|\af|}$.
\end{enumerate}
\end{thm}

\begin{proof}(i) If follows by Proposition \ref{tight gives ultrafilter} and Theorem  \ref{char:ultrafilters}

(ii) We prove that $\xi^\af$ is a tight filter of finite type if and only if $\xi_{|\af|}$ is an ultrafilter and  for each $A \in \xi_{|\af|}$ we have $A\notin\CB_{reg}$:

($\Rightarrow$) We only need to show that $A\notin\CB_{reg}$ for each $A \in \xi_{|\af|}$. Let  $A \in \xi_{|\af|}$.
Suppose first that  $\xi^\af$ is a tight filter such that there is a net $\{\CF_\ld\}_{\ld \in \Lambda} \subseteq X_\af^{sink}$ converging to $\xi_{|\af|}$.  If $|\Delta_A| =\infty$, there is nothing to prove. Suppose $|\Delta_A| < \infty$. Since $A \in \xi_{|\af|}$, there is $\ld$ such that $A \in \CF_\ld \in X_{\af}^{sink}$.  By the definition of $X_{\af}^{sink}$, for each $\bt \in \Delta_A$ there is $C_\bt \in \CF_\ld$ such that $\theta_\bt(C_\bt) =\emptyset$. Since  $|\Delta_A| < \infty$, if $C:=\cap C_\bt$, then $C \in \CF_\ld$. It then follows that $\theta_\bt(A \cap C)=\emptyset$ for all $\bt \in \CL$, and hence, $\emptyset \neq A \cap C \subseteq A$ and $\Delta_{A \cap C} =\emptyset $. Thus, $A \notin \CB_{reg}$.

Next, suppose that  $\xi^\af$ is a tight filter as in part (2) of Proposition \ref{char 1:tight filter of finite type}. There then exists $\ld_0$ such that $A \in f_{\af[t_\ld]}(\CF_\ld)$ for all $\ld \geq \ld_0$. It means that $\theta_{t_\ld}(A) \neq \emptyset$ for all $\ld \geq \ld_0$. Since  $\{t_\ld\}_\ld$ converges to infinity in $\CL$, we have $|\Delta_A| =\infty$. Thus,  $A\notin\CB_{reg}$.

 ($\Leftarrow$) To check $\xi^\af$ is tight, we use Proposition \ref{tight filter:how to check}. First observe that it is enough to consider $x \in \xi^\af$ of the form $x=(\af,A, \af )$. Indeed, if $x=(\af',A,\af') \in \xi^\af$, where $\af=\af'\af''$, and $Z$ is a finite cover for $x$, we define $x'=(\af, \theta_{\af''}(A),\af)$ and $Z'=\{zx': z \in Z\} \setminus \{0\}$ (which is a finite cover of $x'$). Then it is easy to see that if $Z' \cap \xi^\af \neq \emptyset$ then $Z \cap \xi^\af \neq \emptyset$.

         Consider $x=(\af,A,\af) \in \xi^\af$ and let $Z$ be a finite cover for $x$. We claim that there exists $\emptyset \neq D \in \CI_\af$ such that $(\af,D, \af) \in Z$. Since $A \in \xi_{|\af|}$, $A \notin \CB_{reg}$ by our assumption. There then exists $\emptyset \neq B \subseteq A $ such that either $|\Delta_B|=0$ or $|\Delta_B|=\infty$. If  $|\Delta_B|=\infty$, then $|\Delta_A|=\infty$. So, there are infinitely many elements of the form $y_i=(\af b_i, \theta_{b_i}(A), \af b_i)$ with distinct $b_i's$ in $\downarrow x$. Since $Z$ is a cover for $x$, we see that, for each $y_i$, there exists $z_i=(\af\bt_i, C, \af\bt_i) \in Z$ such that $y_i z_i \neq 0$.
Also, since $Z$ is finite and $\af\bt_j$ and $\af b_i$ are comparable, there must exist
 $(\af,D, \af) \in Z$ for some $\emptyset \neq D \in \CI_\af$. 
 If $|\Delta_B|=0$, then $y=(\af,B,\af) \in \downarrow x$. Thus,  an element of $Z$ which intersect $y$  must be of the form $(\af,D,\af) \in Z$ for some $\emptyset \neq D \in \CI_\af$.

Now let $D_1, \cdots, D_n$ be all the sets in $\CI_\af$ such that $(\af,D_i,\af) \in Z$. As we have seen above, $n \geq 1$. We claim that $D_1 \cup  \cdots \cup D_n \in \xi_{|\af|}$. If not, since $\xi_{|\af|}$  is an ultrafilter there exists $C \in \xi_{\af}$ such that $C \cap (D_1 \cup  \cdots \cup D_n) =\emptyset$. Now $A \cap C \in \xi_{|\af|}$ so that $A \cap C \notin \CB_{reg}$. Arguing as above it can be shown that there exists $D$ such that $(\af,D,\af) $ is in $Z$ and intersect $(\af, A\cap C, \af)$, which is a contradiction.

Finally, since $\xi_{|\af|}$  is an ultrafilter and  $D_1 \cup  \cdots \cup D_n \in \xi_{|\af|}$, we have that $D_i \in \xi_{|\af|}$ for some $i=1, \cdots, n$. It then follows that $\xi^\af \cap Z \neq \emptyset$. 
\end{proof}

\section{filter surgery in \texorpdfstring{$E(S)$}{E(S)}}\label{filter surgery}
In this section,  we extend the results in \cite[Section 4]{BCM2} to generalized Boolean dynamical systems; we generalize two operations, namely gluing paths and cutting paths on a usual directed graph $\CE$, to the context of generalized Boolean dynamical systems.
Gluing paths on $\CE$ is just concatenation of paths and cutting paths is just removing the initial segment from a path. As mentioned in \cite[Section 4]{BCM2}, in the context of generalized Boolean dynamical systems, we have an extra layer of complexity because filters in $E(S)$  are described not only by a path but also by a complete family of filters associated with it.  
 When we cut or glue paths, the Boolean algebras where each filter on the family lies change.
 
We begin with the problem of describing new filters by gluing paths. 
Recall that in the  generalized Boolean dynamical system  $(\CB, \CA, \theta, \CI_{r(\af)})$ associated to a labeled  space $(\CE,\CL,\CB)$, we  have an ultrafilter in $\CI_{r(\af\bt)}$ from  an ultrafilter $\CF \in X_{\bt}$ by cutting the elements of $\CF$ by $r(\af\bt)$ (\cite[Lemma 4.2]{BCM2}). We  revisit this case  with a slightly different perspective in the following.

\begin{remark} Let  $(\CE,\CL,\CB)$ be a  labeled space and $(\CB, \CA, \theta, \CI_{r(\af)})$ be the generalized Boolean dynamical system associated to $(\CE,\CL,\CB)$. Suppose that $\CF$ is an ultrafilter in $\CI_{r(\bt)}$ and $\af \in \CL(\CE^n)$. 
Then  it is easy to see the following.
\begin{enumerate}
\item $r(\af\bt) \in \CF$ if and only if $\CF \cap \CI_{r(\af\bt)} \neq \emptyset$ if and only if $\CF \cap \theta_{\af_{2,n}\bt}(\CI_{r(\af_1)}) \neq \emptyset$.
\item(\cite[Lemma 4.2]{BCM2}) We have that $r(\af\bt) \in \CF$ if and only if
\begin{align*} \CJ&:=\{C \cap r(\af\bt): C \in \CF\} \\
&=\CF \cap \CI_{r(\af\bt)}
 \end{align*} is an ultrafilter in $\CI_{r(\af\bt)}$.
\end{enumerate}
\end{remark}

Motivated the above remark, we have the following.

\begin{prop}\label{A}
Let $\CF \in X_\bt$ and  $\af \in \CL^{\geq 1}$  such that $\af\bt \in \CW^*$. Then we have   $\CF \cap \CI_{\af\bt}\neq \emptyset$ if and only if  $ \CF \cap \CI_{\af\bt}$
 is an ultrafilter in $\CI_{\af\bt}$.
\end{prop}

\begin{proof} ($\Rightarrow$) Clearly, $\emptyset \notin \CF \cap  \CI_{\af\bt}$. Let $A \in \CF \cap  \CI_{\af\bt}$ and $B \in \CI_{\af\bt}$ with $A \subseteq B$.  
Then, $B \in \CF$ since $\CF \ni A \subseteq B (\in \CI_{\bt})$ and  $\CF$ is a filter in $\CI_\bt$. Thus, $B \in \CF \cap  \CI_{\af\bt}$.
Let $B_1, B_2 \in \CF \cap  \CI_{\af\bt}$. Then, clearly $B_1 \cap B_2 \in \CF$. We also have that 
$B_1 \cap B_2 \in \CI_{\af\bt}$ since $\CI_{\af\bt}$ is an ideal. So, $B_1 \cap B_2 \in \CF \cap  \CI_{\af\bt}$.

 To show that $\CF \cap  \CI_{\af\bt}$ is an ultrafilter in $\CI_{\af\bt}$, let $B_1 \cup B_2 \in \CF \cap  \CI_{\af\bt}$ for some $B_1, B_2 \in \CI_{\af\bt} (\subseteq \CI_\bt)$.
    Since $B_1 \cup B_2 \in \CF$ and $\CF$ is an ultrafilter in $\CI_\bt$, either $B_1 \in \CF$ or $B_2 \in \CF$.  
  Hence,  either $B_1 \in \CF \cap  \CI_{\af\bt}$ or $B_2 \in \CF \cap  \CI_{\af\bt}$. 
  
  ($\Leftarrow$) If $\CF \cap \CI_{\af\bt} = \emptyset$, then $\CF \cap \CI_{\af\bt}$ is not a filter.
\end{proof}

\begin{remark}\label{tt}  Let $\CF \in X_\bt$ and  $\af \in \CL^{\geq 1}$ such that $\af\bt \in \CW^*$. Then $\CF \cap \CI_{\af\bt} \neq \emptyset$  if and only if $\CF \cap \theta_{\af_{2,|\af|}\bt}(\CI_{\af_1}) \neq \emptyset$. Also, if $\CF \cap \CI_{\af\bt} \neq \emptyset$, then 
$\CF \cap \theta_{\af_{2,|\af|}\bt}(\CI_{\af_1}) \subseteq \CF \cap \CI_{\af\bt}.$
\end{remark}

For $\af \in \CL^{\geq 1}$ and $\bt \in \CL^*$ such that $\af\bt \in \CW^*$, we put
$$X_{(\af)\bt}:=\{\CF \in X_\bt:\CF \cap \CI_{\af\bt}\neq \emptyset  \}.$$
Define a map $g_{(\af)\bt}: X_{(\af)\bt} \to X_{\af\bt}$  by 
$$g_{(\af)\bt}(\CF):=     \CF \cap \CI_{\af\bt}  $$
for each  $\CF \in X_{(\af)\bt}$. Also, for $\af =\emptyset$, define $X_{(\emptyset)\bt}=X_{\bt}$ and let $g_{(\emptyset)\bt}$ denote the identity function on $X_\bt$.

 The following lemmas describe properties of these sets and maps, and how they behave with respect to the maps $f_{\af[\bt]}:X_{\af\bt} \to X_\af$.
 
\begin{lem}\label{for local homeo 2} Suppose that $A \in \CI_\bt$ and $\CF \in X_{(\af)\bt}$. Then $A \in \CF$ if and only if $A \cap D \in g_{(\af)\bt}(\CF)$ for some $D \in \CF \cap \CI_{\af\bt}$.
\end{lem}

\begin{proof} ($\Rightarrow$) Since $\CF \cap \CI_{\af\bt} \neq \emptyset$, one can choose  $D \in \CF \cap \CI_{\af\bt}$. If $A \in \CF$, then $A \cap D \in \CF $ since $\CF$ is closed under finite intersection. Also, $A \cap D \in \CI_{\af\bt}$ since $\CI_{\af\bt}$ is an ideal. Thus,  $A \cap D \in g_{(\af)\bt}(\CF) $.

($\Leftarrow$) Suppose that $A \cap D \in g_{(\af)\bt}(\CF)$ for some $D \in \CF \cap \CI_{\af\bt}$. Since  $A \cap D  \in  \CF$ and $\CF$ is a filter in $\CI_\bt$, it follows that $A \in \CF$.
\end{proof}

\begin{lem}\label{prop:f} Let $\af \in \CL^{\geq 1}$   and $\bt,\gm \in \CL^*$ with $\af\bt\gm \in \CW^*$.  Then
\begin{enumerate}
\item[(i)] $f_{\bt[\gm]}(X_{(\af)\bt\gm}) \subseteq X_{(\af)\bt}$;
\item[(ii)] $f_{\bt[\gm]}^{-1}(X_{(\af)\bt}) \subseteq X_{(\af)\bt\gm}$.
\end{enumerate}
\end{lem}

\begin{proof}  (i) Put $n :=|\af| \in \N$. Given $\CF' \in X_{(\af)\bt\gm}$, one has $\theta_{\af_{2,n}\bt\gm}(A) \in \CF'$ for some $A \in \CI_{\af_1}$. Since  $\theta_{\af_{2,n}\bt\gm}(A)= \theta_{\gm}(\theta_{\af_{2,n}\bt}(A))$, we see that $$\theta_{\af_{2,n}\bt}(A) \in f_{\bt[\gm]}(\CF')=\{B \in \CI_\bt: \theta_{\gm}(B) \in \CF'\}.$$
Thus, $f_{\bt[\gm]}(\CF') \in X_{(\af)\bt}$.

(ii) If $\CF' \in X_{\bt\gm}$ such that $f_{\bt[\gm]}(\CF') \in X_{(\af)\bt}$, then one  has that
$$\theta_{\af_{2,n}\bt}(A) \in f_{\bt[\gm]}(\CF') =\{B \in \CI_\bt: \theta_\gm(B) \in \CF'\}$$ for some $A \in \CI_{\af_1}$. Thus,  $\theta_{\af_{2,n}\bt\gm}(A) \in \CF'$.
It means that $\CF' \in X_{(\af)\bt\gm}$. 
\end{proof}

\begin{lem}\label{g comp f}  Let $\af \in \CL^{\geq 1}$   and $\bt,\gm \in \CL^*$ with $\af\bt\gm \in \CW^*$.  Then
\begin{enumerate}
\item[(i)] $X_{(\af\bt)\gm} \subseteq X_{(\bt)\gm}$;
\item[(ii)] $g_{(\bt)\gm}(X_{(\af\bt)\gm}) \subseteq X_{(\af)\bt\gm}$;
\item[(iii)] If $\CF \in X_{(\af\bt)\gm}$, then $g_{(\af\bt)\gm}(\CF)=g_{(\af)\bt\gm} \circ g_{(\bt)\gm}(\CF)$;
\item[(iv)] The following diagram is commutative:
$$\xymatrix{ X_{(\af)\bt\gm} \ar[d]_{f_{\bt[\gm]}} \ar[r]^{g_{(\af)\bt\gm}} & X_{\af\bt\gm}\ar[d]^{f_{\af\bt[\gm]}}      \\
 X_{(\af)\bt} \ar[r]_{g_{(\af)\bt}} &  X_{\af\bt}.  }$$
 \end{enumerate}

Also, for  $\af \in \CL^{\geq 1}$   and $\bt \in \CL^*$  with $\af\bt \in \CW^*$, we have 
\begin{enumerate}
 \item[(v)] $g_{(\af)\bt}(X_{(\af)\bt}\cap X_{\bt}^{sink}) \subseteq X_{\af\bt}^{sink}$;
\item[(vi)] $g_{(\af)\bt}:X_{(\af)\bt} \to X_{\af\bt}$ is continuous;
\item[(vii)] $X_{(\af)\bt}$ is an open subset of $X_\bt$.
\end{enumerate}
\end{lem}

\begin{proof} (i) If $\CF \cap \CI_{\af\bt\gm} \neq \emptyset$, then  $\CF \cap \CI_{\bt\gm} \neq \emptyset$ since $\CI_{\af\bt\gm} \subseteq \CI_{\bt\gm}$ by Lemma \ref{properties of I}(ii).

(ii) Let $\CF \in X_\gm$ so that $\CF \cap \CI_{\af\bt\gm} \neq \emptyset$. Since 
$$(\CF \cap \CI_{\bt\gm} ) \cap \CI_{\af\bt\gm}=\CF \cap  ( \CI_{\bt\gm} \cap \CI_{\af\bt\gm})  = \CF \cap \CI_{\af\bt\gm} \neq \emptyset,$$ we have 
 $g_{(\bt)\gm}(\CF) \in X_{(\af)\bt\gm}$.

(iii)  For   $\CF \in X_{(\af\bt)\gm}$, we see that
$$g_{(\af)\bt\gm} ( g_{(\bt)\gm}(\CF) )=g_{(\af)\bt\gm}(\CF \cap \CI_{\bt\gm}) = (\CF \cap \CI_{\bt\gm}) \cap \CI_{\af\bt\gm} =  \CF \cap  \CI_{\af\bt\gm}=g_{(\af\bt)\gm}(\CF).$$
So, we have the result.

(iv) If   $\CF \in X_{(\af)\bt\gm}$, then  $ g_{(\af)\bt}(f_{\bt[\gm]}(\CF)) =  f_{\bt[\gm]}(\CF) \cap \CI_{\af\bt}$
and 
\begin{align*} f_{\af\bt[\gm]}(g_{(\af)\bt\gm}(\CF))
&=\{D \in \CI_{\af\bt}: \theta_{\gm}(D) \in \CF \cap \CI_{\af\bt\gm}\}.
\end{align*}
Choose $C  \in g_{(\af)\bt}(f_{\bt[\gm]}(\CF)) $. Then  $\theta_\gm(C) \in \CF $. Also, since $C \in \CI_{\af\bt}$, we have $\theta_\gm(C) \in \CI_{\af\bt\gm}$ by Lemma \ref{properties of I}(iii). Thus, 
$\theta_\gm(C) \in \CF \cap \CI_{\af\bt\gm}$. 
 It means that $ g_{(\af)\bt}(f_{\bt[\gm]}(\CF)) \subseteq f_{\af\bt[\gm]}(g_{(\af)\bt\gm}(\CF))$. Then the equality follows since both terms are ultrafilters.

 \vskip0.5pc
(v) If $\CF \in X_{(\af)\bt}\cap X_{\bt}^{sink}$, then $\CF \cap \CI_{\af\bt} \neq \emptyset$ and  for each $\gm \in \CL$, there is $D \in \CF$ such that $\theta_\gm(D)=\emptyset$. 
Choose $A \in \CF \cap \CI_{\af\bt} $. Then $A\cap D \in \CF \cap \CI_{\af\bt}$ and $\theta_\gm(A\cap D)=\theta_\gm(A) \cap \theta_\gm(D)=\emptyset$. So,   $g_{(\af)\bt}(\CF) (=\CF \cap \CI_{\af\bt}) \in X_{\af\bt}^{sink} $.

(vi) Let  $\{\CF_\ld\}_\ld \subseteq X_{(\af)\bt}$ be  a net  converging to $\CF \in X_{(\af)\bt}$. We show that $\{g_{(\af)\bt}(\CF_\ld)\}_\ld$ converges to $g_{(\af)\bt}(\CF)$. Let  $A \in g_{(\af)\bt}(\CF)$ be arbitrary.  Since $\{\CF_\ld\}_\ld$ converges to $\CF$, one can choose $\ld_0 \geq 1$ so that $\ld \geq \ld_0$ implies  $ A \in \CF_\ld$. Thus, $A \in g_{(\af)\bt}(\CF_\ld) $. 

If $B \in \CI_{\af\bt} \setminus g_{(\af)\bt}(\CF)$, then by \cite[Lemma 12.3]{Ex1} there exists $A\in g_{(\af)\bt}(\CF)$ such that $B \cap A= \emptyset$ since $g_{(\af)\bt}(\CF)$ is an ultrafilter in $\CI_{\af\bt}$. But, since  $A \in \CF$, it follows that $B \notin \CF$. There  then exists  $\ld' \geq 1$ such that $\ld \geq \ld'$ implies $B \notin \CF_\ld$. Hence, for each such $\ld$, there is $C_\ld \in \CF_\ld$ such that $B \cap C_\ld = \emptyset$. 
 Choose $D_\ld \in  \CF_\ld \cap \CI_{\af\bt} (\neq \emptyset)$. Then the element $\emptyset \neq C_\ld \cap D_\ld \in g_{(\af)\bt}(\CF_\ld) $ satisfies 
$$B \cap (C_\ld \cap D_\ld)= \emptyset,$$
from where it can be established that $B \notin g_{(\af)\bt}(\CF_\ld) $ for all $\ld  \geq \ld'$. This says that $\{g_{(\af)\bt}(\CF_\ld)\}_\ld$ converges to $g_{(\af)\bt}(\CF)$.

(vii) Suppose $\CF \in X_{(\af)\bt}$ and let $A\in \CF \cap \CI_{\af\bt} \neq \emptyset$. Then $\CF\in Z(\bt,A)\subseteq X_{(\af)\bt}$.
\end{proof}

We now consider the problem of describing new filters by cutting paths. Let $\CF \in X_{\af\bt}$  be  an ultrafitler for  $\af \in \CL^{\geq 1}$ and $\bt \in \CL^*$ such that $\af\bt \in \CW^*$. Then  $\CF (\subseteq \CI_{\af\bt} \subseteq \CI_{\bt})$ may not be a filter in $\CI_\bt$ since  $\CI_\bt$ may contain elements that are not in $\CI_{\af\bt}$ and  have a subset which is in $\CF$.
 If we add to $\CF$ these elements, the resulting set is an ultrafilter in $\CI_\bt$ as the following result shows.

\begin{prop}\label{cut path} Let $\af \in \CL^{\geq 1}$ and $\bt \in \CL^*$ be such that $\af\bt \in \CW^*$. If $\CF \in X_{\af\bt}$, then 
$\uparrow_{\CI_\bt}\CF \in X_{(\af)\bt}$.
\end{prop}

\begin{proof} 
 It is easy to see that $\uparrow_{\CI_\bt}\CF=\{B \in \CI_\bt: A \subseteq B ~\text{for some}~A \in \CF\}$ is a filter in $\CI_\bt$. To show that it is an ultrafilter, let $B_1,B_2\in\CI_\bt$ be such that $B_1\cup B_2\in \uparrow_{\CI_\bt}\CF$. Then, there exists $A\in\CF$ such that $A\subseteq B_1\cup B_2$, from where it follows that $A=(A\cap B_1)\cup (A\cap B_2)\in\CF$. Since $\CI_{\af\bt}$ is a an ideal and $\CF$ is an ultrafilter in $\CI_{\af\bt}$, we have that either $A\cap B_1\in\CF$ or $A\cap B_2\in\CF$. In the first case $B_1\in \uparrow_{\CI_\bt}\CF$ and in the second case $B_2\in \uparrow_{\CI_\bt}\CF$. Hence $\uparrow_{\CI_\bt}\CF$ is an ultrafilter.

Finally, we notice that $\CF\subseteq \CI_{\af\bt}\cap \uparrow_{\CI_\bt}\CF$, and hence $\uparrow_{\CI_\bt}\CF\in X_{(\af)\bt}$.
\end{proof}

Proposition \ref{cut path} give rises to a function $h_{[\af]\bt}:X_{\af\bt} \to X_{(\af)\bt}$ defined by 
$$h_{[\af]\bt}(\CF)=\uparrow_{\CI_\bt}\CF$$ for each ultrafilter $\CF \in X_{\af\bt}$.

We  examine the properties of these maps.

\begin{lem}\label{h comp f}  Let $\af \in \CL^{\geq 1}$   and $\bt,\gm \in \CL^*$  with $\af\bt\gm \in \CW^*$.  Then
\begin{enumerate}
\item[(i)] $h_{[\bt]\gm} \circ h_{[\af]\bt\gm}=h_{[\af\bt]\gm}$;
\item[(ii)]The following diagram is commutative:
$$\xymatrix{ X_{\af\bt\gm} \ar[d]_{f_{\af\bt[\gm]}} \ar[r]^{h_{[\af]\bt\gm}} & X_{(\af)\bt\gm}\ar[d]^{f_{\bt[\gm]}}      \\
 X_{\af\bt} \ar[r]_{h_{[\af]\bt}} &  X_{(\af)\bt}.  }$$
  \end{enumerate}

Also, for  $\af \in \CL^{\geq 1}$   and $\bt \in \CL^*$  with $\af\bt \in \CW^*$, we have 
\begin{enumerate}
 \item[(iii)]  The functions $h_{[\af]\bt}: X_{\af\bt} \to X_{(\af)\bt}$ and $g_{(\af)\bt}:X_{(\af)\bt} \to X_{\af\bt}$ are mutually inverses;
\item[(iv)] $h_{[\af]\bt}(X_{\af\bt}^{sink}) \subseteq X_{\bt}^{sink}$;
\item[(v)] $h_{[\af]\bt}:X_{\af\bt} \to X_{(\af)\bt}$ is continuous.
 \end{enumerate}
\end{lem}

\begin{proof}(i) For a given $\CF \in X_{\af\bt\gm}$, since $\CF \subseteq \uparrow_{\CI_{\bt\gm}}\CF$ one obtains 
$$h_{[\af\bt]\gm}(\CF)=\uparrow_{\CI_\gm}\CF \subseteq \uparrow_{\CI_\gm}\big( \uparrow_{\CI_{\bt\gm}}\CF\big)=h_{[\bt]\gm} \circ h_{[\af]\bt\gm}(\CF),$$
which implies these are equal since they are both ultrafilters in $\CI_\gm$.

(ii) If $\CF \in X_{\af\bt\gm}$, then 
\begin{align*} h_{[\af]\bt} \circ f_{\af\bt[\gm]}(\CF)=\{D \in \CI_\bt: \exists C \in \CI_{\af\bt} ~\text{with }~ C \subseteq D ~\text{and}~ \theta_\gm(C) \in \CF\}.
\end{align*}
For a given $D \in h_{[\af]\bt} \circ f_{\af\bt[\gm]}(\CF) $, if $C \in \CI_{\af\bt}$ is as above then $\theta_\gm(D) \subseteq \theta_\gm(C) \in \CI_{\bt\gm}$, so that
$$\theta_\gm(D) \in \uparrow_{\CI_{\bt\gm}}\theta_\gm(C) \subseteq \uparrow_{\CI_{\bt\gm}}\CF = h_{[\af]\bt\gm}(\CF)$$
and therefore $D \in f_{\bt[\gm]}\circ h_{[\af]\bt\gm}(\CF)$. Thus,   $h_{[\af]\bt} \circ f_{\af\bt[\gm]}(\CF) \subseteq f_{\bt[\gm]}\circ h_{[\af]\bt\gm}(\CF)$, and then 
these are equal since they are both ultrafilters in $\CI_\bt$.
\vskip 0.5pc

(iii) First, we show that $\CF=\uparrow_{\CI_\bt} (\CF\cap \CI_{\af\bt}) (=h_{[\af]\bt}\circ g_{(\af)\bt}(\CF))$  for $\CF \in X_{(\af)\bt}$. Clearly, $\uparrow_{\CI_\bt} (\CF\cap \CI_{\af\bt}) \subseteq \CF$. Thus, 
they are equal since both are ultrafilters.

Secondly, we show that $\CF = (\uparrow_{\CI_\bt}\CF) \cap \CI_{\af\bt} (=g_{(\af)\bt}\circ h_{[\af]\bt}(\CF))$ 
 for $\CF \in X_{\af\bt}$.
Clearly, we see that $\CF \subseteq (\uparrow_{\CI_\bt}\CF) \cap \CI_{\af\bt}$ for $\CF \in X_{\af\bt}$. Hence, they are equal since both are ultrafilters.

(iv)   Suppose $\CF \in X_{\af\bt}^{sink}$. Then for each $\gm \in \CL$, there exists $D \in \CF (\in X_{\af\bt})$ such that $\theta_\gm(D) =\emptyset$. Since $D \in \CF \subseteq \uparrow_{\CI_\bt}\CF=h_{[\af]\bt}(\CF)$, it says that for each $\gm \in \CL$, there exists $D \in h_{[\af]\bt}(\CF)$ such that $\theta_\gm(D) =\emptyset$. So, $h_{[\af]\bt}(\CF) \in X_{\bt}^{sink}$.

(v) Consider a net $\{\CF_\ld\}_\ld \subseteq X_{\af\bt}$ that converges to $\CF \in X_{\af\bt}$, and let $D \in \CI_\bt$ be arbitrary.
If $D \in h_{[\af]\bt}(\CF)$, there exists $C \in \CF$ such that $C \subseteq D$.
The convergence of the $\CF_\ld$ says then that there is $\ld_0 \geq 1$ such that $\ld \geq \ld_0$ implies $C \in \CF_\ld$, and hence
$$D \in \uparrow_{\CI_\bt}C \subseteq \uparrow_{\CI_\bt} \CF_\ld= h_{[\af]\bt}(\CF_\ld).$$
On the other hand, if $D \in \CI_\bt \setminus h_{[\af]\bt}(\CF)$, then there exists $C \in h_{[\af]\bt}(\CF)$ such that $C \cap D =\emptyset$ by \cite[Lemma 12.3]{Ex1}. 
What has been just shown above says that there exists $\ld_0 \geq 1$ such that $\ld \geq \ld_0 $ implies $C \in h_{[\af]\bt}(\CF_\ld)$. Thus, $D \notin h_{[\af]\bt}(\CF_\ld)$ if $\ld \geq \ld_0$.  
Thus, the net $\{h_{[\af]\bt}(\CF_\ld)\}_\ld$
converges to $h_{[\af]\bt}$, and hence, $h_{[\af]\bt}$ is continuous. 
\end{proof}
The following will be used to prove Proposition \ref{local homeo on T}.

\begin{lem}\label{for local homeo} Suppose that $A \in \CI_{\af\bt}$ and $\CF \in X_{\af\bt}$. Then $A \in \CF$ if and only if $A \in h_{[\af]\bt}(\CF)$.
\end{lem}

\begin{proof} ($\Rightarrow$) It is clear. 

($\Leftarrow$)
Since $\CF=g_{(\af)\bt}(h_{[\af]\bt}(\CF))=h_{[\af]\bt}(\CF)\cap  \CI_{\af\bt}$ by Lemma \ref{h comp f}(iii), it follows that $A  \in \CF$.
\end{proof}

The functions $g_{(\af)\bt}$ and $h_{[\af]\bt}$ can be used to produce tools for working with filters in $E(S)$.
Let us begin with a function for gluing paths to a filter $\xi$ in $E(S)$ using the maps $g_{(\af)\bt}$.
Given  $\af \in \CL^{\geq 1}$ and $\bt \in \CW^{\leq \infty}$, 
we first explain how to construct a complete family of ultrafilters for $\af\bt$
 starting from a complete family  $\{\CF_k\}_k:=\{\CF_k\}_{0 \leq k \leq |\bt|}$ of ultrafilters for $\bt$.

Consider  the case $\bt=\emptyset$. In this case the complete family $\{\CF_k\}_k$ for $\emptyset$ consist of a single filter $\CF_0 \subseteq \CB$. Define
$$\CJ_{|\af|}:=g_{(\af)\emptyset}(\CF_0)=\CF_0 \cap \CI_\af$$
and, for $0 \leq i < |\af| (=: n)$, we recursively set 
\begin{align*} \CJ_{n-1}&:=f_{\af_{1,n-1}[\af_n]}(\CJ_{n}) \\
\CJ_{n-2}&:=f_{\af_{1,n-2}[\af_{n-1}]}(\CJ_{n-1}) \\
& \vdots \\
\CJ_0&:=f_{\emptyset[\af_1]}(\CJ_1) \\
&(=\{A \in \CB: \theta_{\af_1}(A) \in \CJ_{1}\}).
\end{align*}
Now, for the case where $\bt \neq \emptyset$, define for $1 \leq k \leq |\bt|$ (or $k < |\bt|$ if $\bt$ is infinite), 
$$\CJ_{|\af|+k}:=g_{(\af)\bt_{1,k}}(\CF_k).$$

\begin{remark} Note that, in order to this construction make sense, one must have that $\CF_n \in 
X_{(\af)\bt_{1,n}}$ for all $n$, which follows from requiring $\CF_1 \in X_{(\af)\bt_1}$. In fact, if $\CF_1 \in X_{(\af)\bt_1}$, then $
\CF_1 \cap \theta_{\af_{2,|\af|}\bt_1}(\CI_{\af_1}) \neq \emptyset$. Since $\{\CF_n\}_n$ is complete, one can see that $\CF_k \cap \theta_{\af_{2,|\af|}\bt_{1,k}}(\CI_{\af_1}) \neq \emptyset $ for  $1 \leq k \leq |\bt|$ (or $k < |\bt|$ if $\bt$ is infinite).
\end{remark}
Also, letting $n :=|\af|$, we recursively set 
\begin{align*} \CJ_{n}&:=f_{\af[\bt_1]}(\CJ_{n+1}) \\
\CJ_{n-1}&:=f_{\af_{1,n-1}[\af_{n}]}(\CJ_{n}) \\
& \vdots \\
\CJ_0&:=f_{\emptyset[\af_1]}(\CJ_1). 
\end{align*}

\begin{remark}\label{R1} If it is the case that $\CF_0 \neq \emptyset$ in the complete family $\{\CF_n\}_n$, one has $\CJ_{|\af|}=g_{(\af)\emptyset}(\CF_0)$ since 
\begin{align*}  \CJ_{n}=f_{\af[\bt_1]}(\CJ_{n+1}) &=f_{\af\emptyset[\bt_1]} \circ g_{(\af)\emptyset\bt_1}(\CF_1) \\
&=g_{(\af)\emptyset} \circ f_{\emptyset[\bt_1]}(\CF_1)=g_{(\af)\emptyset}(\CF_0).
\end{align*}
\end{remark}

\begin{remark} Since  $f_{\af[\bt\gm]}=f_{\af[\bt]}\circ f_{\af\bt[\gm]}$, we have that
\begin{align*} \CJ_{n-1}&=f_{\af_{1,n-1}[\af_n]}(f_{\af[\bt_1]}(\CJ_{|\af|+1})) \\
&=f_{\af_{1,n-1}[\af_n\bt_1]}(\CJ_{|\af|+1}).
\end{align*}
Similarly, $\CJ_i= f_{\af_{1,i}[\af_{i+1,|\af|}\bt_1]}(\CJ_{|\af|+1})$
for each $0 \leq i < |\af|-1$.
\end{remark}
 We then see that the resulting family $\{\CJ_i\}_i$ is complete for $\af\bt$.

\begin{prop}\label{after gluing} (\cite[Proposition 4.11]{BCM2}) Let $\af \in \CL^{\geq 1}$. 
\begin{enumerate}
\item[(i)] If $\CF \in X_{(\af)\emptyset}$, then $\{\CJ_i\}_{0 \leq i \leq |\af|}$ as above is a complete family of ultrafilters for $\af$.
\item[(ii)] If $\bt \in \CL^{\leq \infty} \setminus \{\emptyset\}$ and $\{\CF_k\}_k$ is a complete family of ultrafilter for $\bt$ such that $\CF_1 \in X_{(\af)\bt_1}$, then 
$\{\CJ_i\}_i$ as above is a complete family of ultrafilters for $\af\bt$.
\end{enumerate}
\end{prop}

\begin{proof} The proof is identical to \cite[Proposition 4.11]{BCM2}. 
\end{proof}

Continuing the discussion above, let $\mathsf{T}_\af$ be the set of all tight filters  in $E(S)$ associated with  $\af \in \CW^{\leq \infty}$.
For $\bt \in \CW^{\leq \infty}$,  let 
$\mathsf{T}_{(\af)\bt}$ denote the subset of $\mathsf{T}_\bt$ given by 
\begin{align*}\mathsf{T}_{(\af)\bt}&:=\{\xi \in \mathsf{T}_\bt: \xi_0 \in X_{(\af)\emptyset}\}\\
&=\{\xi \in \mathsf{T}_\bt: \xi_0  \cap \CI_\af\neq \emptyset    \}.
\end{align*}
Observe that, for $\bt \neq \emptyset$, $\xi_0 \in X_{(\af)\emptyset}$ is equivalent to  $\xi_1 \in X_{(\af)\bt_1}$. 

 The following theorem says that   a filter $\eta^{\af\bt}$ associated with the resulting pair  $(\af\bt, \{\CJ_i\}_i)$  is also tight.  
\begin{thm}\label{gluing gives tight}(\cite[Theorem 4.12]{BCM2}) Let $\af \in \CL^{\geq 1}$ and $\bt \in \CW^{\leq \infty}$. If $\xi \in \mathsf{T}_{(\af)\bt}$ and $\{\CJ_i\}_i$ is the complete family for $\af\bt$ constructed as above from $\{\xi_n\}_n$, then the filter $\eta \in \mathsf{F}_{\af\bt}$ associated with $\{\CJ_i\}_i$ is tight. 
\end{thm}

\begin{proof} The proof is identical to \cite[Theorem 4.12]{BCM2}. 
\end{proof}
Through the above discussion, one  can define a gluing map 
$$G_{(\af)\bt}: \mathsf{T}_{(\af)\bt} \to \mathsf{T}_{\af\bt}$$ taking a tight filter $\xi \in \mathsf{T}_{(\af)\bt} $
to the tight filter $\eta$ given by Theorem \ref{gluing gives tight}. Also, for $\af = \emptyset$, define 
$\mathsf{T}_{(\emptyset)\bt}=\mathsf{T}_\bt$ and let $G_{(\emptyset)\bt}$ be the identity function on $\mathsf{T}_\bt$.

\vskip 1pc

 Next, we define a function for removing paths from a filter $\xi \in E(S)$ using the maps $h_{[\af]\bt}$. First, we explain  how to get a new complete family of ultrafilters for $\bt$ from a complete family of ultrafilters for $\af\bt$ in below.

\begin{lem}\label{after cutting}(\cite[Lemma 4.14]{BCM2}) Let $\af \in \CL^{\geq 1}$ and $\bt \in \CL^{\leq \infty}$ such that $\af\bt \in \CW^{\leq \infty}$. If   $\{\CF_n\}_n$ is the complete family for $\af\bt$, then 
$$\{h_{[\af]\bt_{1,n}}(\CF_{|\af|+n})\}_n$$
is a complete family of ultrafilters for $\bt$.
\end{lem}

\begin{proof} The proof is identical to \cite[Lemma 4.14]{BCM2}. 
\end{proof}

\begin{thm}\label{cutting gives tight}(\cite[Theorem 4.14]{BCM2}) Let $\af \in \CL^{\geq 1}$ and $\bt \in \CL^{\leq \infty}$ such that $\af\bt \in \CW^{\leq \infty}$. If $\xi \in \mathsf{T}_{\af\bt}$, then the filter $\eta \in \mathsf{F}_\bt$ associated with the complete family $\{h_{[\af]\bt_{1,n}}(\xi_{n+|\af|})\}_n$ for $\bt$ is a tight filter. 
\end{thm}

\begin{proof} The proof is identical to \cite[Theorem 4.14]{BCM2}.
\end{proof}

By the above theorem, one can therefore define a cutting map  
$$H_{[\af]\bt}: \mathsf{T}_{\af\bt} \to \mathsf{T}_{(\af)\bt}$$
by 
$$H_{[\af]\bt}(\xi^{\af\bt})=\eta^\bt$$ where for all $n $ with $0 \leq n \leq |\bt|$,
 \begin{align}\label{H-map}\eta^\bt_n=h_{[\af]\bt_{1,n}}(\xi_{n+|\af|}) \in X_{(\af)\bt_{1,n}}.\end{align}
 For $\af= \emptyset$, define $H_{[\emptyset]\bt}$ to be the identity function over $\mathsf{T}_\bt$.

\begin{lem}\label{G comp G}(\cite[Lemma 4.14 and Lemma 4.16]{BCM2}) Let $\af,\bt \in \CL^{\geq 1}$ and $\gm \in \CL^{\leq \infty}$ such that $\af\bt\gm\in\CW^{\leq\infty}$. Then, 
\begin{enumerate}
\item[(i)] $\mathsf{T}_{(\af\bt)\gm} \subseteq \mathsf{T}_{(\bt)\gm}$;
\item[(ii)] $G_{(\af\bt)\gm}=G_{(\af)\bt\gm}\circ G_{(\bt)\gm}$;
\item[(iii)] $H_{[\bt]\gm} \circ H_{[\af]\bt\gm}= H_{[\af\bt]\gm}.$
\end{enumerate}
\end{lem}
 
  \begin{proof} (i) It follows by Lemma \ref{g comp f}(i).

(ii)  It is same with  \cite[Lemma 4.13 (ii)]{BCM2}.

(iii) It is immediate from Lemma \ref{h comp f}(i).
\end{proof}

\begin{thm}\label{GH=id}(\cite[Theorem 4.17]{BCM2}) Let $\af \in \CL^{\geq 1}$ and $\bt \in \CW^{\leq \infty}$ such that $\af\bt\in \CW^{\leq \infty}$. Then $H_{[\af]\bt} \circ G_{(\af)\bt}$ and $G_{(\af)\bt} \circ H_{[\af]\bt}$ are the identity maps over $\mathsf{T}_{(\af)\bt}$ and $\mathsf{T}_{\af\bt}$, respectively. 
\end{thm}

\begin{proof} It is identical to \cite[Theorem 4.17]{BCM2}. 
\end{proof}

\section{A Boundary path groupoid having the tight spectrum \texorpdfstring{$\mathsf{T}$}{T} as the unit space}\label{section:tight.groupoid}

In this section, we provide a  groupoid model for a generalized Boolean dynamical system and show that it  is isomorphic to the tight groupoid $\CG_{tight}:=\CG_{tight}(S_{(\CB, \CL,\theta, \CI_\af) })$ arising from the inverse semigroup $S_{(\CB, \CL,\theta, \CI_\af) }$  given in section  \hyperref[An inverse semigroup]{3.1}.
We then define a local homeomorphism on the tight spectrum and  show that the defined groupoid is isomorphic to the Renault-Deaconu groupoid for this local homeomorphism. Most of results in this section can be obtained by the same arguments used in \cite{BCM3}. So, we only provide a proof for the results which we need to modify.

In what follows we fix a generalized Boolean dynamical system $(\CB, \CL,\theta, \CI_\af)$ and put  $S:=S_{(\CB, \CL,\theta, \CI_\af) } $ and $E:=E(S)$.

\begin{prop}\label{boundary path groupoid of tight filters}(\cite[Proposition 3.4]{BCM3}) Let $(\CB, \CL,\theta, \CI_\af)$ be a generalized Boolean dynamical system and $S$ be its associated inverse semigroup. Define 
$$\Gamma(\CB, \CL,\theta, \CI_\af)=\{(\eta^{\af\gm}, |\af|-|\bt|, \xi^{\bt\gm}) \in \mathsf{T} \times \Z \times \mathsf{T} : H_{[\af]\gm}(\eta^{\af\gm})=H_{[\bt]\gm}(\xi^{\bt\gm}) \}.$$
Then, $\Gamma(\CB, \CL,\theta, \CI_\af)$ is a groupoid with product given by 
$$(\eta, m ,\xi)(\xi, n, \rho)=(\eta, m+n, \rho)$$
and an inverse given by 
$$(\eta, m , \xi)^{-1}=(\xi, -m, \eta).$$
\end{prop} 

\begin{proof} The proof is identical with \cite[Proposition 3.4]{BCM3}. 
\end{proof}

If it is clear  which generalized Boolean dynamical system we are working with, we just write $\Gamma$ for $\Gamma(\CB, \CL,\theta, \CI_\af)$. The reason we think that the groupoid given in Proposition \ref{boundary path groupoid of tight filters} as a boundary path groupoid is that we can model a graph as labeled space, and therefore as a generalized Boolean dynamical system, in which case the tight spectrum is homeomorphic to the boundary path space of the graph. Also, in the graph case, $\Gamma$ is isomorphic to the boundary path groupoid of the graph.

Recall that for each idempotent $e \in E$, we let $D_e:=\{ \phi \in \hat E_{tight}: \phi(e)=1\}$ and  $$\Omega=\{ (s,\phi) \in S \times \hat{E}_{tight} : \phi \in D_{s^*s} \},$$
and the standard action $\rho$ of $S$ on $\hat{E}_{tight}$  is given by $\rho_s(\phi)(e)=\phi(s^*es) $ (see Definition \ref{def:tight groupoid} for more details).

\begin{remark}(\cite[Remark 3.1]{BCM3}) Let a nonzero element $s=(\mu,A,\nu)\in S$ and $\phi\in \hat{E}_{tight} $. If $\xi^\af$ is the filter associated with $\phi$, then $\phi(s^*s)=1$ if and only if $s^*s=(\nu,A,\nu) \in \xi^\af$ and if and only if $\nu$ is a beginning of $\af$ and $A \in \xi^\af_{|\nu|}$.
\end{remark}

\begin{lem}(\cite[Proposition 3.2]{BCM3})\label{equivalence class} Let $s=(\mu, A, \nu)$, $t=(\bt, B, \gm)$ be two non-zero element in $S$ and $\phi$ be a tight character associated with a filter $\xi^\af$ on $E$. Suppose that $(s, \phi), (t, \phi) \in \Omega$ and that $\nu$ is a beginning of $\gm$ with $\gm=\nu\gm'$. Then $(s, \phi) \sim (t, \phi)$ if and only if $\bt=\mu\gm'$.
\end{lem}

\begin{proof}
By replacing $r(-,\af)$ to $\theta_\af(-)$ in the proof of \cite[Proposition 3.2]{BCM3}, we have the result.  
\end{proof}

\begin{lem}\label{abc}Let $\CF \in X_{\af\bt}$. If $A \in \CF$ and $C \in \uparrow_{\CI_\bt}\CF$, then $A \cap C \in \CF$.
\end{lem}

\begin{proof} Note first that  $A \cap C \in \CI_{\af\bt}$ since $\CI_{\af\bt}$ is an ideal.  Also, there is $X \in \CF$ such  that $X \subseteq C$. Since $A \cap X \subseteq A \cap C$ and $A \cap X \in \CF$, we have $A \cap C \in \CF$.
\end{proof}

\begin{lem}\label{well defined} Let $(t, \phi) \in \Omega$ with $t=(\bt,A,\gm)$ and $\phi \in  \hat{E}_{tight}$ be a character associated with $\xi^\af$ so that $\af=\gm\af'$ for some $\af' \in \CW^{\leq \infty}$. Then $\eta=(G_{(\bt)\af'} \circ H_{[\gm]\af'})(\xi^\af)$ is well-defined and $\eta$ is the filter associated with $\rho_t(\phi)$. 
\end{lem}

\begin{proof} 
 To see that $\eta=(G_{(\bt)\af'} \circ H_{[\gm]\af'})(\xi^\af)$ is well-defined, first note that $(\bt,A,\gm) $ is an element of $S$ so that $\emptyset \neq A \in \CI_{\bt}\cap \CI_\gm,$ and   that  \begin{align}\label{H_0} H_{[\gm]\af'}(\xi^\af)_0=h_{[\gm]\emptyset}(\xi_{|\gm|})=\uparrow_{\CI_\emptyset} \xi_{|\gm|}=\uparrow_{\CB} \xi_{|\gm|}. \end{align}
If $\bt =\emptyset$, then the domain $G_{(\bt)\af'} (=id_{\mathsf{T}_{\af'}})$ is $\mathsf{T}_{\af'}$, which contains  the range $\mathsf{T}_{(\gm)\af'}$ of $H_{[\gm]\af'}$. So, $G_{(\emptyset)\af'}( H_{[\gm]\af'}(\xi^\af))=H_{[\gm]\af'}(\xi^\af)$ is well-defined.
If $\bt \neq \emptyset$, 
then $A \in \xi^\af_{|\gm|} $.
Thus, $A \in H_{[\gm]\af'}(\xi^\af)_0$ by the equation (\ref{H_0}). Since $A \in \CI_\bt$, we see that
 $H_{[\gm]\af'}(\xi^\af)_0 \cap  \CI_\bt \neq \emptyset$. 
It means that $H_{[\gm]\af'}(\xi^\af)_0  \in X_{(\bt)\emptyset} (=\{\CF \in X_\emptyset:\CF \cap \CI_\bt\neq \emptyset\})$.
It thus follows that $H_{[\gm]\af'}(\xi^\af) \in \mathsf{T}_{(\bt)\af'} (=\{\eta \in \mathsf{T}_{\af'}: \eta_0 \in X_{(\bt)\emptyset}\})$. Since $\mathsf{T}_{(\bt)\af'}$  is the domain of $G_{(\bt)\af'}$, we see that $G_{(\bt)\af'}( H_{[\gm]\af'}(\xi^\af))$ is well-defined.

Let $\psi$ be the character associated with $\eta=(G_{(\bt)\af'} \circ H_{[\gm]\af'})(\xi^\af)$. We show that $\psi=\rho_t(\phi) $. Choose $(\dt,D, \dt) \in E$. We first observe that
\begin{align*} \rho_t(\phi)(\dt,D, \dt) &=\phi ((\gm,A, \bt)(\dt,D, \dt)(\bt,A,\gm)) \\
&=\left\{\begin{array}{ll}
  \phi(\gm\dt', \theta_{\dt'}(A)\cap D), \gm\dt')  & \hbox{if\ }                                      \dt=\bt\dt'\\
  \phi(\gm, A \cap \theta_{\bt'}(D), \gm) & \hbox{if\ } \bt=\dt\bt' \\                       
                                               0 & \hbox{otherwise}
                      \end{array}
                    \right. \\  
 &= \left\{\begin{array}{ll}
  \Big[ \dt' ~\text{is a beginning of~ }\af'~\text{and}~ \theta_{\dt'}(A)\cap D \in \xi^\af_{|\gm\dt'|} \Big]  & \hbox{if\ }                     \dt=\bt\dt'\\
   \Big[ A \cap \theta_{\bt'}(D) \in \xi^\af_{|\gm|} \Big] & \hbox{if\ }\bt=\dt\bt'  \\
            0 & \hbox{otherwise},
                      \end{array}
                    \right.
  \end{align*}
 and that  
$$\psi(\dt,D, \dt)=[\dt ~\text{is a beginning of }~ \bt\af' ~\text{and}~D \in \eta_{|\dt|}],$$
 where $[~]$ represents the Boolean function that returns 0 if the argument is false and 1 if it is true. 
 
If $\dt=\bt\dt'$, then 
\begin{align*} \eta_{|\dt|}&=(g_{(\bt)\dt'} \circ h_{[\gm]\dt'})(\xi_{|\gm\dt'|}) \\
&=g_{(\bt)\dt'}\big( \uparrow_{\CI_{\dt'}} \xi_{|\gm\dt'|}\big) \\
&=  \big( \uparrow_{\CI_{\dt'}} \xi_{|\gm\dt'|}\big) \cap \CI_{\bt\dt'}.
\end{align*}
We claim that $D \in \eta_{|\dt|}$ if and only if $D \cap \theta_{\dt'}(A) \in \xi_{|\gm\dt'|}$.
 If $D \in \eta_{|\dt|}$, then clearly $D  \in \uparrow_{\CI_{\dt'}} \xi_{|\gm\dt'|}$.
   Since $A \in \xi_{|\gm|}$, we have $\theta_{\dt'}(A ) \in \xi_{|\gm\dt'|}$, and hence,  $D \cap \theta_{\dt'}(A)  \in  \xi_{|\gm\dt'|} $
 by Lemma \ref{abc}.

On the other hand,  if $D \cap \theta_{\dt'}(A) \in \xi_{|\gm\dt'|}$ for $D \in \CI_{\dt}(= \CI_{\bt\dt'})$,  since $D \in \CI_{\bt\dt'} \subseteq \CI_{\dt'}$ and $D \cap \theta_{\dt'}(A) \subseteq D$, then $D \in \uparrow_{\CI_{\dt'}} \xi_{|\gm\dt'|}$. Hence,  $D \in \eta_{|\dt|}$.

Now suppose that $\bt=\dt\bt'$. Then 
\begin{align*} \eta_{|\dt|}&= (f_{\dt[\bt']}\circ g_{(\bt)\emptyset}\circ h_{[\gm]\emptyset})(\xi_{|\gm|}) \\
&=(f_{\dt[\bt']}\circ g_{(\bt)\emptyset})(\uparrow_{\CB}\xi_{|\gm|})\\
&=f_{\dt[\bt']}\Big(\big( \uparrow_{\CB}\xi_{|\gm|}\big) \cap \CI_\bt \Big)\\
&=\{D \in \CI_{\dt}: \theta_{\bt'}(D) \in \big(\uparrow_{\CB}\xi_{|\gm|}\big) \cap \CI_\bt \}.
\end{align*}

We claim that $D \in \eta_{|\dt|}$ if and only if $A \cap \theta_{\bt'}(D) \in \xi_{|\gm|}$. 
If $D \in \eta_{|\dt|}$, then clearly $\theta_{\bt'}(D) \in  \uparrow_{\CB}\xi_{|\gm|} $. Then, $A \cap \theta_{\bt'}(D) \in \xi_{|\gm|}$ by Lemma \ref{abc}.

On the other hand,  if $A \cap \theta_{\bt'}(D) \in \xi_{|\gm|}$, we have $\theta_{\bt'}(D)  \in \uparrow_{\CB}\xi_{|\gm|}$.  
Since $D \in \CI_{\dt}$, $\theta_{\bt'}(D)  \in \CI_{\dt\bt'} (=\CI_{\bt})$ by Lemma \ref{properties of I}(iii).
 Hence, $D \in \eta_{|\dt|}$.
 
We conclude that $\psi(\delta,D,\delta)=\rho_t(\phi)(\delta,D,\delta)$ for all $(\delta,D,\delta)\in E$, so that $\psi=\rho_t(\phi)$.
\end{proof}

The following is a generalization of \cite[Theorem 3.7]{BCM3}.

\begin{thm}\label{tight iso boundary} Let $(\CB, \CL,\theta, \CI_\af)$ be a generalized Boolean dynamical system and $S$ be its associated inverse semigroup. Then the map
$$ \Phi: \CG_{tight} \to \Gamma $$
 given by 
$$ [t, \phi] \mapsto ((G_{(\bt)\af'}\circ H_{[\gm]\af'})(\xi^\af), |\bt|-|\gm|, \xi^\af)$$
is a well-defined  isomorphisms of groupoids
where $\xi^\af$ is the filter associated with $\phi$ and 
  $t=(\bt, A, \gm) \in S$ is such that $(\gm,A, \gm) \in \xi^\af$, so that $\af=\gm\af'$ for some $\af' \in \CW^{\leq \infty}$.
\end{thm}

\begin{proof} For $t$ and $\phi$ as in the statement, $(G_{(\bt)\af'}\circ H_{[\gm]\af'})(\xi^\af)$ is well-defined by Lemma \ref{well defined}. Also, 
since $H_{[\bt]\af'}= G_{(\bt)\af'}^{-1}$ , we have $$(H_{[\bt]\af'} \circ G_{(\bt)\af'}\circ H_{[\gm]\af'})(\xi^\af)=H_{[\gm]\af'}(\xi^\af),$$
so that $((G_{(\bt)\af'}\circ H_{[\gm]\af'})(\xi^\af), |\bt|-|\gm|, \xi^\af) \in \Gamma$.

Now, let $s=(\mu,B, \nu) \in S$ be such that $\phi \in D_{s^*s}$ and $[s, \phi]=[t, \phi]$. By Proposition \ref{equivalence class}, it then sufficient to suppose that $\gm=\nu\gm'$ and $\bt=\mu\gm'$ for some $\gm' \in \CL^*$. 
Then $$|\bt|-|\gm|=|\mu\gm'|-|\nu\gm'|=|\mu|-|\nu|$$
and by Lemma \ref{G comp G}(ii), Lemma \ref{G comp G}(iii) and Theorem \ref{GH=id},
\begin{align*}(G_{(\bt)\af'}\circ H_{[\gm]\af'})(\xi^\af)&=(G_{(\mu)\gm'\af'}\circ G_{(\gm')\af'}\circ H_{[\gm']\af'}\circ H_{[\nu]\gm'\af'})(\xi^\af)\\
&=(G_{(\mu)\gm'\af'}\circ H_{[\nu]\gm'\af'})(\xi^\af),
\end{align*}
so that $\Phi$ dose not depend on the representative.

 We show that $\Phi$ is onto.  If $(\eta^{\tilde{\af}}, m ,\xi^\af) \in \Gamma$, then there exists $\bt, \gm \in \CW^*$ and $\af' \in \CW^{\leq \infty}$  such that $\tilde{\af}=\bt\af'$, $\af=\gm\af'$, $m =|\bt|-|\gm|$ and $H_{[\gm]\af'}(\xi^\af)=H_{[\bt]\af'}(\eta^{\tilde{\af}})$. Let $\phi$ be the filter associated with $\xi^\af$. Suppose $\gm \neq \emptyset$. Since $H_{[\gm]\af'}(\xi^\af)=H_{[\bt]\af'}(\eta^{\tilde{\af}})$, we have
$$\uparrow_{\CB} \xi_{|\gm|} =h_{[\gm]\emptyset}(\xi_{|\gm|})=h_{[\bt]\emptyset}(\eta_{|\bt|})=\uparrow_{\CB} \eta_{|\bt|}.$$  
Here, we note that  $\CF:=\uparrow_{\CB} \xi_{|\gm|} =\uparrow_{\CB} \eta_{|\bt|}$ satisfies that $\CF \cap \CI_\gm \neq \emptyset$ and 
$\CF \cap \CI_\bt \neq \emptyset$. 
Choose $\emptyset \neq D \in  \CF \cap \CI_\gm \cap \CI_\bt$ and define $t=(\bt, D , \gm )$. 
Note that  $t^*t=(\gm,D ,\gm) \in \xi^\af_{|\gm|}$ since $D \in (\uparrow_{\CB} \xi_{|\gm|})\cap \CI_\gm=g_{(\gm)\emptyset}(h_{[\gm]\emptyset}(\xi_{|\gm|}))=\xi_{|\gm|}$.
Then, by Theorem \ref{GH=id},
\begin{align*}\Phi[t, \phi]&= ((G_{(\bt)\af'}\circ H_{[\gm]\af'})(\xi^\af), |\bt|-|\gm|, \xi^\af) \\
&=((G_{(\bt)\af'}\circ H_{[\bt]\af'})(\eta^{\tilde{\af}}), |\bt|-|\gm|, \xi^\af)\\
&= (\eta^{\tilde{\af}}, m ,\xi^\af).
\end{align*}
If $\gm=\emptyset$ and $\af' \neq \emptyset$, define $\gm'=\af_1'$, $\bt'=\bt\af_1'$ and $\af''=\af'_{2,|\af'|}$, and repeat the above argument.
 If $\gm=\af'=\emptyset$, then $h_{[\bt]\emptyset}(\eta_{|\bt|})=\xi_0$. For $B\in\eta_{|\bt|}$, we have that $B \in \xi_0 \cap \CI_\bt$, so that $t=(\bt, B, \emptyset)\in S$. Then   $t^*t=(\emptyset, B, \emptyset) \in \xi_0$ and, once again, we  can see that  $\Phi[t, \phi]=(\eta^{\tilde{\af}}, m ,\xi^\af)$.

Showing that $\Phi$ is injective, $(\Phi \times \Phi)\Big(\CG^{(2)}_{tight}\Big) \subseteq \Gamma^{(2)}$ and $\Phi$ preserves multiplication are identical to \cite[Theorem 3.7]{BCM3}. So, we omit it. 
\end{proof}

\subsection{Topology on \texorpdfstring{$\Gamma$}{Gamma}}
We now describe the topology on $\Gamma$ induced by the isomorphism given in Theorem \ref{tight iso boundary}. We  start with reviewing  the topology on $\mathsf{T}$.

\begin{remark}
Using the bijection between filters and characters as well as the topology on characters given by pointwise convergence, we see that a basis of compact-open sets for the induced topology on $\mathsf{T}$ is given by sets of the form
\[V_{e:e_1\ldots,e_n}=\{\xi\in\mathsf{T}:e\in\xi,e_1\notin\xi,\ldots,e_n\notin\xi\},\]
where $e \in E$ and $\{e_1, \cdots, e_n\}$ is a finite (possibly empty) subset of $E$. See \cite{Lawson2012} for more details.
\end{remark}

We define a collection of subsets on $\Gamma$ as follows. 
Given  $s=(\mu,A,\nu) \in S$ and $e, e_1, \cdots, e_n \in E$, let 
$$Z_{s,e:e_1, \cdots,e_n}:=\{(\eta^{\mu\gm}, |\mu|-|\nu|, \xi^{\nu\gm}) \in \Gamma:  \xi \in V_{e:e_1, \cdots,e_n} ~\text{and}~ H_{[\mu]\gm}(\eta)=H_{[\nu]\gm}(\xi) \}.$$

\begin{remark} In the above definition, if $e=s^*s$, we denote it by $Z_{s:e_1, \cdots,e_n}$. We also allow $n$ to be zero and in  this case we denote it simply by $Z_s$ when $e=s^*s$, namely 
$$Z_s:=\{(\eta^{\mu\gm}, |\mu|-|\nu|, \xi^{\nu\gm}) \in \Gamma:  \xi \in V_{s^*s} ~\text{and}~ H_{[\mu]\gm}(\eta)=H_{[\nu]\gm}(\xi) \}.$$
\end{remark}

\begin{prop}\label{topology on Gamma}(\cite[Proposition 4.4]{BCM3}) The sets $Z_{s,e:e_1, \cdots,e_n}$ defined above form a basis of compact-open sets for the topology on $\Gamma$ induced by the map $\Phi$ of   Theorem \ref{tight iso boundary} from the topology on $\CG_{tight}$ given by the sets $\Theta(s,\CU)=\{[s, \phi] \in \CG_{tight}: \phi \in \CU\}$.  
\end{prop}

\begin{proof} It is same with \cite[Proposition 4.4]{BCM3}.
\end{proof}

\begin{cor}(\cite[Corollary 4.6]{BCM3}) $\Gamma$ with the above topology is Hausdorff.
\end{cor}

\begin{proof} It is easy to check that $S_{(\CB,\CL, \theta, \CI_\af)}$ is $E^*$-unitary (see \cite[Proposition 4.6]{BCM3}). Then the result follows immediately from \cite[Corollary 10.9]{Ex1}.
\end{proof}

\subsection{\texorpdfstring{$\Gamma$}{Gamma} as a Renault-Deaconu groupoid}\label{section:Renault-Deaconu}

Here, we  show that $\Gamma$ can be seen as a Renault-Deaconu groupoid. 
 We define, for each $n \geq 1$, $\mathsf{T}^{(n)}=\{\xi^\af \in \mathsf{T}: \af \in \CW^{\leq \infty}, |\af| \geq n \}$ and $\sigma:\mathsf{T}^{(1)} \to \mathsf{T}$ by $\sigma(\xi^\af)=H_{[\af_1]\gm}(\xi)$ if $\af=\af_1\gm$.

\begin{prop}\label{local homeo on T} Let $\mathsf{T}^{( 1)}$ and $\sigma$ as above. Then
\begin{enumerate} 
\item[(i)]  $\mathsf{T}^{( 1)}$ is open;
\item[(ii)] $\sigma$ is a local homeomorphism.
\end{enumerate}
\end{prop}

\begin{proof}  
(i) If $\xi = \xi^\af$ with $|\af| \geq 1$, then $(\af_1, A, \af_1) \in \xi$ for some $A \in \CI_{\af_1}$. Hence, 
$$\mathsf{T}^{(1)}= \bigcup_{\af \in \CL, A \in \CI_\af} V_{(\af,A, \af)},$$
which is an open set.

(ii) Let $\xi^{\af} \in \mathsf{T}^{( 1)}$ and choose $A \in \CI_{\af_1}$ such that $\xi \in  V_{(\af_1, A, \af_1)}$. We show that $\sm_{\af_1,A}:=\sigma|_{V_{(\af_1,A,\af_1)}}$ is a homeomorphism between $V_{(\af_1,A,\af_1)}$ and $V_{(\emptyset, A, \emptyset)}$. To see tha that $\sm_{\af_1,A}$ is injective, suppose that $\sm_{\af_1,A}(\xi)=\zeta^\gm=\sm_{\af_1,A}(\eta)$.  Then $\xi=\xi^{\af_1\gm}$, $\eta=\eta^{\af_1\gm}$ and $H_{[\af_1]\gm}(\xi^{\af_1\gm})=\sm_{\af_1,A}(\xi)=\sm_{\af_1,A}(\eta)=H_{[\af_1]\gm}(\eta^{\af_1\gm})$. Since $H_{[\af_1]\gm}$ is injective, we have  $\xi=\eta$.

For the surjectivity, let $\zeta=\zeta^\gm \in V_{(\emptyset, A, \emptyset)}$ where $A \in \CI_{\af_1}$. Since $(\emptyset, A, \emptyset)\in  \zeta$, $A \in \zeta_0$, and hence $\zeta_0 \in X_{(\af_1)\emptyset}(=\{\CF \in X_{\emptyset} : \CF \cap \CI_{\af_1} \neq \emptyset\})$.  So, $\zeta \in \mathsf{T}_{(\af_1)\gm}(=\{\eta \in \mathsf{T}_{\gm}: \eta_0 \in X_{(\af_1)\emptyset}\})$. Take $\xi=G_{(\af_1)\gm}(\zeta)$, so that $\xi \in V_{(\af_1,A,\af_1)}$ and $\sigma_{\af_1,A}(\xi)=H_{[\af_1]\gm}(\xi)=\zeta$.

To show that $\sm_{\af_1,A}$ is a homeomorphism, one can use  Lemma \ref{for local homeo} to see that 
\begin{align*}&\sm_{\af_1,A}(V_{(\af_1\gm, B, \af_1\gm):(\af_1\gm\bt_1, B_1, \af_1\gm\bt_1), \cdots, (\af_1\gm\bt_n, B_n, \af_1\gm\bt_n)}) \\
&=\{H_{[\af_1]\gm}(\xi): \xi \in V_{(\af_1\gm, B, \af_1\gm):(\af_1\gm\bt_1, B_1, \af_1\gm\bt_1), \cdots, (\af_1\gm\bt_n, B_n, \af_1\gm\bt_n)} \}\\
&=V_{(\gm,B, \gm):(\gm\bt_1, B_1, \gm\bt_1), \cdots, (\gm\bt_n, B_n, \gm\bt_n)}
\end{align*}
for an arbitrary basic open set $V_{(\af_1\gm, B, \af_1\gm):(\af_1\gm\bt_1, B_1, \af_1\gm\bt_1), \cdots, (\af_1\gm\bt_n, B_n, \af_1\gm\bt_n)}$ of $V_{(\af_1,A,\af_1)}$, where $B \subseteq \theta_\gm(A)$. Also, we can use Lemma \ref{for local homeo 2} to conclude that 
\begin{align*}\sm_{\af_1,A}^{-1}&(V_{(\emptyset, A, \emptyset)} \cap V_{(\gm,B, \gm):(\gm\bt_1, B_1, \gm\bt_1), \cdots, (\gm\bt_n, B_n, \gm\bt_n)}) \\
&=\{G_{(\af_1)\gm}(\xi): \xi \in V_{(\emptyset, A, \emptyset)} \cap V_{(\gm,B, \gm):(\gm\bt_1, B_1, \gm\bt_1), \cdots, (\gm\bt_n, B_n, \gm\bt_n)} \}\\
&= V_{(\af_1\gm,B \cap \theta_{\gm}(A), \af_1\gm):(\af_1\gm\bt_1, B_1 \cap \theta_{\gm\bt_1}(A), \af_1\gm\bt_1), \cdots, (\af_1\gm\bt_n, B_n \cap \theta_{\gm\bt_n}(A), \af_1\gm\bt_n)}
\end{align*}
for an arbitrary basic open set 
$V_{(\emptyset, A, \emptyset)} \cap V_{(\gm,B, \gm):(\gm\bt_1, B_1, \gm\bt_1), \cdots, (\gm\bt_n, B_n, \gm\bt_n)} $ of $V_{(\emptyset, A, \emptyset)}$, where $B \subseteq \theta_\gm(A)$.
Hence,  $\sm_{\af_1,A}$ and $\sm_{\af_1,A}^{-1}$ are open mapping. So,  $\sm_{\af_1,A}$ is a homeomorphism.
\end{proof}

Notice that for $n \in \N$, $\operatorname{dom}(\sigma^n)=\mathsf{T}^{(n)}$. Also, if $\xi=\xi^{\af\bt}$, then $\sigma^{|\af|}(\xi)=H_{[\af]\bt}(\xi)$ by Lemma \ref{G comp G}(iii).  This implies that 
$$\Gamma=\{(\eta, m-n, \xi):  m,n \in \N, \eta \in \operatorname{dom}(\sigma^m), \xi  \in  \operatorname{dom}(\sigma^n) ~\text{and}~ \sigma^m(\eta)=\sigma^n(\xi) \},$$
that is, $\Gamma$ is the Renault-Deaconu groupoid associated with $\sigma$. 

 We now prove that the topology given in Proposition \ref{topology on Gamma} is the same as the one generated by the basic open sets 
 \begin{align}\label{basis by RD}\CV(X,m,n, Y):=\{(\eta,m-n,\xi): (\eta,\xi) \in X \times Y ~\text{and}~\sm^m(\eta)=\sm^n(\xi)\},\end{align}
 where $X$ (resp. Y) is an open subset of $\operatorname{dom}(\sm^m)$ (resp. $\operatorname{dom}(\sm^n)$) for which $\sm^m|_X$ (resp. $\sm^n|_Y$) is injective. 
 
 \begin{lem}\label{RD is finer} If $(\af,A,\bt) \in S$ and $e_1=(\bt\dt_1, B_1, \bt\dt_1), \cdots, e_n=(\bt\dt_n, B_n, \bt\dt_n) \in E$ are such that $B_i \subseteq \theta_{\dt_i}(A)$, then $\sm^{|\bt|}$ is injective restricted to $V_{s^*s:e_1, \cdots, e_n}$,  $\sm^{|\af|}$ is injective restricted to $V_{ss^*:f_1, \cdots, f_n}$ and 
 \begin{equation}\label{Z equals V}Z_{s:e_1, \cdots, e_n}=\CV(V_{ss^*:f_1, \cdots, f_n}, |\af|, |\bt|, V_{s^*s:e_1, \cdots, e_n}),\end{equation}
 where $f_1=(\af\dt_1, B_1, \af\dt_1), \cdots, (\af\dt_n, B_n, \af\dt_n)$.
 \end{lem}

 \begin{proof} For injectivity, it is sufficient to show that $\sm^{|\bt|}$ is injective in $V_{s^*s}$. This is indeed the case, for $\eta, \xi \in V_{s^*s}$, if $\sm^{|\bt|}(\eta)=\sm^{|\bt|}(\xi)=\zeta^{\gm}$, then $\eta=G_{(\bt)\gm}(\zeta)=\xi$ by Theorem \ref{GH=id}.
 
 By the definition of $S$, $A \in \CI_\af \cap \CI_\bt$. Choose $A' \in \CI_{\af_1}$ and $B' \in \CI_{\bt_1}$ so that $A \subseteq \theta_{\af_{2, |\af|}}(A')\cap \theta_{\bt_{2,|\bt|}}(B')$. Then we see that 
 $$B_i \subseteq \theta_{\dt_i}(A) \subseteq  \theta_{\af_{2, |\af|}\dt_i}(A')\cap \theta_{\bt_{2,|\bt|}\dt_i}(B').$$
 From this, it follows that for $(\eta^{\af\gm}, |\af|-|\bt|, \xi^{\bt\gm}) \in \Gamma$, 
 $$  \eta \in V_{ss^*: f_1, \cdots, f_n}   ~\text{if and only if}~\xi \in V_{s^*s: e_1, \cdots, e_n}.$$
  In fact,  since $\eta^{\af\gm}=G_{(\af)\gm}(H_{[\bt]\gm}(\xi^{\bt\gm}))$, we have $\eta_{|\af|}=g_{(\af)\emptyset}(h_{[\bt]\emptyset}(\xi_{|\bt|}))$. Thus, 
 if $\xi \in V_{s^*s: e_1, \cdots, e_n} $, then $A \in \eta_{|\af|}$, and hence, $ss^*=(\af, A, \af) \in \eta$. One also can see that  $f_1, \cdots, f_n \notin \eta$. 
 Similarly, we have the other implication. The equality of \eqref{Z equals V} then follows from the above equivalence and the corresponding definitions of the sets.
 \end{proof}

 \begin{prop} The topology on $\Gamma$ given by Proposition \ref{topology on Gamma} is the same as the topology generated by the basic open set given by (\ref{basis by RD}).
 \end{prop}
 
 \begin{proof} Since we have Lemma \ref{RD is finer}, it is sufficient to show that for each $(\eta, m-n, \xi) \in \CV(X,m,n, Y)$, there exists $s \in S$ and $e_1, \cdots, e_n \in E$ such that $(\eta, m-n, \xi) \in Z_{s:e_1, \cdots, e_n} \subseteq   \CV(X,m,n, Y)$.
 Given such $(\eta, m-n, \xi)$, 
  there exist $\af, \bt, \gm \in \CW^{\leq \infty}$ such that $|\af|=m$, $|\bt|=n$, $\eta=\eta^{\af\gm}$, $\xi=\xi^{\bt\gm}$ and $H_{[\af]\gm}(\eta)=H_{[\bt]\gm}(\xi)$. Since $\xi \in Y$ and $Y$ is open, there exist $e, e_1, \cdots, e_n \in E$ such that $\xi \in V_{e:e_1, \cdots, e_n} \subseteq Y$. This implies that $e$ is of the form $(\bt\dt, A, \bt\dt)$ for some beginning of $\dt$ of $\gm$.
  Since $H_{[\af]\gm}(\eta)=H_{[\bt]\gm}(\xi)=\zeta$, we have $\zeta_{|\dt|}=h_{[\af]\dt}(\eta_{|\af\dt|})=h_{[\bt]\dt}(\xi_{|\bt\dt|})$. Since $A \in \xi_{|\bt\dt|}$, $A \in \uparrow_{\CI_{\dt}} \xi_{|\bt\dt|}= \uparrow_{\CI_{\dt}} \eta_{|\af\dt|}$. So, there exists $A' \in \eta_{|\af\dt|} (\subseteq \CI_{\af\dt})$ such that $A' \subseteq A$. We also note that $A' \in \CI_{\bt\dt}$. Now, take $s=(\af\dt, A',  \bt\dt)$, so that $Z_{s:e_1, \cdots, e_n}$ is the desired set. 
 \end{proof}

\section{The \texorpdfstring{$C^*$}{C*}-algebra of generalized Boolean dynamical system as a groupoid \texorpdfstring{$C^*$}{C*}-algebra}\label{section:C*-isomorphism}

Fix $(\CB,\CL,\theta,\CI_\af)$ a generalized Boolean dynamical system. By the results of the previous sections, we have that $  C^*(\Gamma)    \cong C^*(\CG_{tight}).$ We now prove that each of them is isomorphic to the $C^*$-algebra of generalized Boolean dynamical system by showing that $ C^*(\Gamma)  \cong C^*(\CB,\CL,\theta, \CI_\af)$.
 We first recall that, since $\Gamma$ is an $\acute{e}$tale groupoid, 
 the corresponding Haar system is given by counting measure and, for $f,g \in C_c(\Gamma)$,  
 their product is given by 
 $$(f*g)(\eta,n,\xi)=\sum_{\zeta,m:(\eta,m,\zeta)\in \Gamma}f(\eta, m, \zeta)g(\zeta, n-m, \xi).$$

\begin{prop}\label{representation in C(G)}   Consider the following functions in $C_c(\Gamma)$ viewed as elements of $C^*(\Gamma)$:
\begin{align*} P_A&=\chi_{Z_{(\emptyset, A, \emptyset)}},\\
S_{\af,B}&=\chi_{Z_{(\af,B,\emptyset)}}, 
\end{align*}
where $A \in \CB$, $\af \in \CL$ and $B \in \CI_\af$.
Then  $\{P_A,\ S_{\alpha,B}: A\in\CB,\ \alpha\in\CL ~\text{and}~\ B\in\CI_\alpha\}$ is a $(\CB,\CL,\theta,\CI_\alpha)$-representation.
\end{prop}

\begin{proof}
Let  $A,A'\in\mathcal{B}$, $\alpha,\alpha'\in\mathcal{L}$, $B\in\mathcal{I}_\alpha$ and $B'\in\mathcal{I}_{\alpha'}$. 
Since $\xi_0$ is a filter in $\CB$, $\emptyset \notin \xi_0$ and hence, $P_\emptyset=0$.
 Fix an arbitrary element $(\eta,n, \xi) \in \Gamma$.  First, observe that
\begin{align*}(P_A * P_{A'})(\eta,n,\xi)&=\sum_{\zeta,m:(\eta,m,\zeta)\in \Gamma} P_A(\eta,m,\zeta)P_{A'}(\zeta,n-m,\xi) \\
&=\left\{
   \begin{array}{ll}
   1   & \hbox{if\ }  ~m=n=0,~ \eta=\zeta=\xi,~ A \in \xi_0 ~\text{and}~ A' \in \xi_0,   \\
    0  & \hbox{if\ } ~\text{otherwise},
   \end{array}
\right.
\end{align*}
and that 
\begin{align*}P_{A \cap A'}(\eta,n,\xi)
&=\left\{
   \begin{array}{ll}
   1   & \hbox{if\ }  ~n=0,~\eta=\xi ~\text{and}~ A\cap A' \in \xi_0,   \\
    0  & \hbox{if\ } ~\text{otherwise}.
   \end{array}
\right.
\end{align*}
Since $\xi_0$ is a filter,  $A \in \xi_0$ and $A' \in \xi_0$ if and only if $A \cap A' \in \xi_0$.
It thus follows that  $P_A * P_B=P_{A\cap B}$.

The equality $P_{A \cup A'}=P_A+P_{A'}-P_{A\cap A'}$ follows from the fact that $\xi_0$ is an ultrafilter and therefore a prime filter, so that $A \cup B \in \xi_0$ if and only if $A \in \xi_0$ or $B \in \xi_0$.

We  see that 
\begin{align*}(&P_A  *S_{\af,B})(\eta, n, \xi)\\
&=\sum_{\zeta,m:(\eta,m,\zeta)\in \Gamma}P_A(\eta,m,\zeta)S_{\af,B}(\zeta,n-m, \xi)\\
&=\left\{
   \begin{array}{ll}
   1   & \hbox{if\ }  ~m=0, n=1, \zeta=\eta=\eta^{\af\gm}, A \in \eta_0, H_{[\af]\gm}(\eta)=\xi ~\text{and}~ B \in \xi_0, \\
    0  & \hbox{if\ } ~\text{otherwise}.
   \end{array}
\right.
\end{align*}
Notice that $\eta_1\subseteq \xi_0$ and $\theta_\af(A)\in\CI_\af$, we have that $A\in\eta_0$ if and only if $\theta_\af(A)\in\xi_0$. Then again, since $\xi_0$ is a filter in $\CB$, $\theta_{\af}(A) \in \xi_0$ and $B \in \xi_0$ if and only if $\theta_{\af}(A) \cap B \in \xi_0$.
Thus, $P_A *S_{\af,B}=S_{\af, \theta_{\af}(A) \cap B}$.

Also, we have 
\begin{align*}(&S_{\af,B}*P_{\theta_\af(A)})(\eta,n, \xi)\\
&=\sum_{\zeta,m:(\eta,m,\zeta)\in \Gamma} S_{\af,B}(\eta,m,\zeta)P_{\theta_\af(A)}(\zeta,n-m, \xi)\\
&=\left\{
   \begin{array}{ll}
   1   & \hbox{if\ }  ~m=n=1, \eta=\eta^{\af\gm}, H_{[\af]\gm}(\eta)=\zeta=\xi ~\text{and}~ B, \theta_\af(A) \in \xi_0, \\
    0  & \hbox{if\ } ~\text{otherwise}.
   \end{array}
\right.
\end{align*}
Then again, $B \in \xi_0$ and $\theta_\af(A) \in \xi_0$ if and only if $B \cap \theta_\af(A) \in \xi_0$. It then easy to see that   $S_{\af,B}*P_{\theta_\af(A)}=S_{\af, \theta_\af(A)\cap B}$. Hence $P_A *S_{\af,B}=S_{\af,B}*P_{\theta_\af(A)}$.

We have that 
\begin{align*}(&S_{\af,B}^* * S_{\af',B'})(\eta,n,\xi)\\
&=\sum_{\zeta,m:(\eta,m,\zeta)\in \Gamma} S_{\af,B}^*(\eta,m,\zeta)S_{\af',B'}(\zeta,n-m,\xi) \\
&=\sum_{\zeta,m:(\eta,m,\zeta)\in \Gamma} S_{\af,B}(\zeta, -m, \eta)S_{\af',B'}(\zeta,n-m,\xi) \\
&= \left\{
   \begin{array}{ll}
   1   & \hbox{if\ }  ~\af=\af',  m=-1, n=0,  \eta=H_{[\af]\gm}(\zeta^{\af\gm})=\xi, ~\text{and}~ B, B' \in \xi_0, \\
    0  & \hbox{if\ } ~\text{otherwise}.
   \end{array}
\right.
\end{align*}
On the other hand,  
\begin{align*}\delta_{\af,\af'}P_{B \cap B'}(\eta,n,\xi)&= \left\{
   \begin{array}{ll}
   1   & \hbox{if\ }  ~\af=\af',  n=0,  \eta=\xi, ~\text{and}~ B \cap B' \in \xi_0, \\
    0  & \hbox{if\ } ~\text{otherwise}.
   \end{array}
\right.
\end{align*}
Since $\xi_0$ is a filter,  $B \in \xi_0$ and $B' \in \xi_0$ if and only if $B \cap B' \in \xi_0$. Thus, $S_{\af,B}^* * S_{\af',B'}=\delta_{\af,\af'}P_{B \cap B'}$.

Finally, for the last relation, let $A \in \CB_{reg}$. 
We first observe that 
\begin{align*} &\sum_{\af \in \Delta_A} S_{\af, \theta_\af(A)}*S_{\af, \theta_\af(A)}^*(\eta, n ,\xi) \\
& =\left\{
   \begin{array}{ll}
   S_{\af, \theta_\af(A)}*S_{\af, \theta_\af(A)}^*(\eta, n ,\xi)   & \hbox{if\ }  \eta=\eta^{\af\gm} ~\text{for some}~ \af \in \Delta_A,\\
    0  & \hbox{if\ } ~\text{otherwise}, \\
   \end{array}
\right. \\
& =\left\{
   \begin{array}{ll}
  1   & \hbox{if\ }  \eta=\eta^{\af\gm}, n=0, \xi=\xi^{\af\gm}, ~\text{and}~ H_{[\af]\gm}(\eta)=H_{[\af]\gm}(\xi),\\
    0  & \hbox{if\ } ~\text{otherwise}. \\
   \end{array}
\right. 
\end{align*} 
 On the other hand,  $P_A(\eta,n,\xi)=1$ if and only if $ n=0, \eta=\xi$, and $ A \in \eta_0$. Here, since $A \in \CB_{reg}$, if $\eta=\eta^\dt$, then $A \in \eta_0$ implies that $\dt\neq \emptyset$ by Theorem \ref{char:tight}, and in this case $\dt=\af\gm$ for some $\af$ such that $\af \in \Delta_A$. 
Thus, for $A \in \CB_{reg}$, we have $P_A=\sum_{\af \in \Delta_A} S_{\af, \theta_\af(A)}*S_{\af, \theta_\af(A)}^* $. 
So, we are done.
\end{proof}

By the universal property of $C^*(\CB,\CL,\theta,\CI_\af)$, there is a $*$-homomorphism
\begin{align*} \pi: C^*(\CB,\CL,\theta,\CI_\af) \to C^*(\Gamma)
\end{align*}
such that $\pi(p_A)=P_A$ and $\pi(s_{\af,B})=S_{\af,B}$ for $A \in \CB$, $\af \in \CL$ and $B \in \CI_\af$.
We shall show that the map $\pi$ is an isomorphism. To do that, we need a  lemma.

\begin{lem} For all $(\af, A, \bt) \in S$, we have $S_{\af,A}*S_{\bt,A}^*=\chi_{Z_{(\af,A,\bt)}}$.
\end{lem}

\begin{proof} We see  that 
\begin{align*}(&S_{\af,A}*S_{\bt,A}^*)(\eta,n,\xi)\\
&=\sum_{\zeta,m:(\eta,m,\zeta)\in \Gamma} S_{\af,A}(\eta,m,\zeta)S_{\bt,A}^*(\zeta,n-m,\xi) \\
&=\sum_{\zeta,m:(\eta,m,\zeta)\in \Gamma}S_{\af,A}(\eta,m,\zeta)S_{\bt,A}(\xi,m-n,\zeta) \\
&= \left\{
   \begin{array}{ll}
   1   & \hbox{if\ }  ~m=1, n=0, A\in  \zeta_0,   \eta=\eta^{\af\gm}, \xi=\xi^{\bt\gm},~\text{and}~ H_{[\af]\gm}(\eta)=H_{[\bt]\gm}(\xi)=\zeta,\\
    0  & \hbox{if\ } ~\text{otherwise}
   \end{array}
\right.\\
&=\chi_{Z_{(\af,A,\bt)}}.
\end{align*}
For the last equality, we observe that for $A \in \CI_{\bt}$, 
 $A\in  \zeta_0$ if and only if $(\bt,A,\bt) \in \xi$. 
 In fact, since $H_{[\bt]\gm}(\xi)=\zeta$, we have $G_{(\bt)\gm}(\zeta)=\xi$. Then, we see that  $\xi_{|\bt|}=G_{(\bt)\gm}(\zeta)_{|\bt|}=g_{(\bt)\emptyset}(\zeta_0)=\zeta_0 \cap \CI_\bt$ by Remark \ref{R1}. Thus, it easy to see that  $A\in  \zeta_0$ if and only if $(\bt,A,\bt) \in \xi$. 
\end{proof}
\begin{prop}(\cite[Corollary 5.3 and Proposition 5.7]{BCM3}) The map $\pi$ is an isomorphism.
\end{prop}

\begin{proof} Using the gauge invariant uniqueness theorem (\cite[Corollary 6.2]{CaK2}), one can see that $\pi$ is injective.

By replacing $r(-,\af)$ to $\theta_\af(-)$ in the proof of \cite[Proposition 5.7]{BCM3}, one can have that  $\pi$ is surjective. 
\end{proof}

Combining all together, we have the following. It generalizes \cite[Theorem 5.8]{BCM3}.

\begin{thm}\label{C*-isom}Let  $(\CB,\CL,\theta, \CI_\af)$ be a generalized Boolean dynamical system.  Then, we have 
$$C^*(\CB,\CL,\theta, \CI_\af) \cong C^*(\Gamma)   \cong C^*(\CG_{tight}).$$
\end{thm}

\section{Boundary path groupoids from  topological correspondences associated with generalized Boolean dynamical systems}\label{section:topological.correspondence}

In this section, we provide another type of a  boundary path groupoid for a  generalized Boolean dynamical system and examine a relation between it and the groupoid $\Gamma$ defined in  section 5. We shall first construct a topological correspondence $E_{(\CB,\CL,\theta,\CI_\alpha)}$ from   a generalized Boolean dynamical system $(\CB, \CL,\theta, \CI_\af)$.   
We then define a notion of  the boundary path space $\partial E$  of the topological correspondence as \cite[Definition 3.1]{KL2017}. As a result, we have a Renault-Deaconu groupoid $\Gamma(\partial E, \sm_E)$  associated to a  shift map $\sm_E$ on the boundary path space. 
  We  next prove  that the tight spectrum $\mathsf{T} $ is homeomorphic to the boundary path space $\partial E$, and that  $\Gamma(\partial E, \sm)$ is isomorphic to   the groupoid $\Gamma(\CB,\CL, \theta, \CI_\af)$  given in section 5  
 as  topological groupoids.

\subsection{A topological correspondence}

Let $(\CB, \CL,\theta, \CI_\af)$ be a generalized Boolean dynamical system.  Recall, from section \ref{section:tight.filters}, that $X_\emptyset=\WCB$ is equipped with  the topology generated by $\{Z(A): A\in\CB\}$, where 
$Z(A)=\{\xi\in\widehat{\CB}:A\in\xi\}$, and that 
$X_\af=\widehat{\CI_\af}$ is equipped with the topology generated by  $\{Z(\af, A): A\in\CI_\af\}$, where $Z(\af, A):=\{\xi \in \widehat{\CI_{\af}}: A \in \xi\}$.
 
To define a topological correspondence from the generalized Boolean dynamical system $(\CB, \CL,\theta, \CI_\af)$, we let $E^0_{(\CB,\CL,\theta,\CI_\alpha)}=X_\emptyset$ as a topological space, and we let
\[
	F^0_{(\CB,\CL,\theta,\CI_\alpha)}:=X_\emptyset\cup\{\emptyset\}
\]
and equip $F^0_{(\CB,\CL,\theta,\CI_\alpha)}$ with the topology described in Section \ref{section:tight.filters}.

We also let
\begin{equation*}
	E^1_{(\CB,\CL,\theta,\CI_\alpha)}:=\bigl\{e^\alpha_\eta:\alpha\in\CL,\ \eta\in X_\alpha\bigr\}
\end{equation*}
and equip $E^1_{(\CB,\CL,\theta,\CI_\alpha)}$ with the topology generated by 
\begin{equation}\label{basis on E^1}
\mathbb{V}:=\bigcup_{\alpha\in\CL} \{Z^1(\af, B):B\in\CI_\alpha\},
\end{equation}
where $$Z^1(\af, B):=\{e^\af_\eta: \eta \in X_\af  , B \in \eta\}.$$
Notice that $E^1_{(\CB,\CL,\theta,\CI_\alpha)}$ is homeomorphic to the disjoint union $\bigsqcup_{\af\in\CL}X_\af$ (here, we are considering the disjoint union of topological spaces that are not necessarily mutually disjoint).
\begin{prop}\label{def:topological correspondence} Let $(\CB,\CL,\theta, \CI_\af)$ be a generalized Boolean dynamical system and let $E^0:=E^0_{(\CB,\CL,\theta,\CI_\alpha)}$, $F^0:=F^0_{(\CB,\CL,\theta,\CI_\alpha)}$ and $E^1:=E^1_{(\CB,\CL,\theta,\CI_\alpha)}$ be as above. If we define the maps
 $d:E^1\to E^0$ and $r:E^1\to F^0$ by 
\[
	d(e^\alpha_\eta)=h_{[\alpha]\emptyset}(\eta) \text{ and } r(e^\alpha_\eta)=f_{\emptyset[\af]}(\eta)
\]
 then $(E^1, d,r)$ is a topological correspondence from $E^0$ to $F^0$.
\end{prop}

\begin{proof}
By construction, we have that $E^0$, $F^0$ and $E^1$ are  totally disconnected, locally compact Hausdorff spaces. To prove that $d$ is a local homeomorphism, we use the above observation that $E_1$ is homeomorphic to disjoint union $\bigsqcup_{\af\in\CL}X_\af$. In each part of this disjoint union, $d$ behaves as the map $h_{[\af]\emptyset}:X_{\af}\to X_{(\af)\emptyset}$, which is a homeomorphism between $X_{\af}$ and open subset of $X_\emptyset$ by Lemmas \ref{g comp f} and \ref{h comp f}. Similarly, we have that $r$ is continuous because in each part of the disjoint it behaves as $f_{\emptyset[\af]}$, which is continuous by Lemma \ref{lem:f.empty.continuous}. Hence, $(E^1, d,r)$ is a topological topological correspondence from $E^0$ to $F^0$.
\end{proof}

The construction of a C*-algebra associated with a topological graph in \cite{Ka2004} uses a C*-correspondence over the algebra $C_0(E^0)$. Motivated by what we will see below, we consider a similar C*-algebra for the above topological correspondence by considering a C*-correspondence over $C_0(E^0)=C_0(X_\emptyset)$. We notice that $E^0=X_\emptyset$ is a open subset of $F^0=X_\emptyset\cup\{\emptyset\}$ that contains $d(E^1)$, so that it is an ideal of $C_0(F^0)=C(X_\emptyset\cup\{\emptyset\})$. We also notice that the construction of the Hilbert module done in \cite{Ka2004} does not use the information on $r$, which is used later to define a left action.

As in \cite{Ka2004}, for each $p\in C(E^1)$, we let $\left<p,p\right>:X_\emptyset\to [0,\infty]$ be the function defined by $\left<p,p\right>(v):=\sum_{e\in d^{-1}(v)}|p(e)|^2$. Then, the set $C_d(E^1):=\{p\in C(E^1):\left<p,p\right>\in C_0(X_\emptyset)\}$ has a structure of Hilbert module given by
\[\left<p,q\right>(v)=\sum_{e\in d^{-1}(v)}\overline{p(e)}q(e),\]
and
\[(pa)(e):=p(e)a(d(e)),\]
where $p,q\in C_d(E^1)$, $a\in C_0(X_\emptyset)$, $v\in E^0$ and $e\in E^1$. Also, since $C_0(X_\emptyset)$ can be seen as an ideal of $C_0(X_\emptyset\cup\{\emptyset\})$, we can define a left action by
\[(ap)(e):=a(r(e))p(e)\]
where $p\in C_d(E^1)$, $a\in C_0(X_\emptyset)$ and $e\in E^1$.

We now recall the construction of a C*-correspondence from a generalized Boolean dynamical system such that its C*-algebra is isomorphic to $C^*(\CB,\CL,\theta,\CI_\af)$ as done in \cite{CaK2}. We change the notation here in order not to confuse with some other notation already established in this paper. For the coefficient algebra, we take $C_0(\WCB)=C_0(X_\emptyset)$ (the subalgebra $\CA(\CB,\CL,\theta)$ defined in \cite{CaK2} actually coincides with $C_0(\WCB)$ by \cite[Proposition~2.14]{COP}). For each $A\in\CB$, we let $\chi_A=1_{Z(A)}$, and for each $\af\in\CL$, define
\[M_{\af}:=\overline{\operatorname{span}}\{\chi_A:A\in\CI_\af\}.\]
We observe that $M_{\af}\cong C_0(X_{\af})$ via a map that sends $m_\af\in M_\af$ to $m_\af\circ h_{[\alpha]\emptyset}$. Since $M_{\af}$ is an ideal of $C_0(X_{\emptyset})$, it has a natural Hilbert module structure. We then let
\begin{align*}
    M(\CB,\CL,\theta,\CI_\af) & :=\bigoplus_{\af\in\CL} M_{\af}\\
    & =\left\{(m_\af)_{\af\in\CL}\in \prod_{\af\in\CL}M_\af:\sum_{\af\in\CL}m_\af^*m_\af\in C_0(X_\emptyset)\right\}
\end{align*}
as a direct sum of Hilbert modules, with inner product given by
\[\left<(m_\af)_{\af\in\CL},(n_\af)_{\af\in\CL}\right>=\sum_{\af\in\CL}m^*_\af n_\af\]
and right action given by
\[(m_\af)_{\af\in\CL}\cdot a=(m_\af a)_{\af\in\CL},\]
for $(m_\af)_{\af\in\CL},(n_\af)_{\af\in\CL}\in M(\CB,\CL,\theta,\CI_\af)$ and $a\in C_0(X_\emptyset)$. The left action is such that for $A\in\CB$ and $(m_\af)_{\af\in\CL}\in M(\CB,\CL,\theta,\CI_\af)$
\[\chi_A\cdot (m_\af)_{\af\in\CL}=(\chi_{\theta_\af(A)}m_\af)_{\af\in\CL}.\]
Observe that $\chi_{\theta_\af(A)}=\chi_A\circ f_{\emptyset[\af]}$ for $A\in\CB$ and $\af\in\CL$. The left action is then given more generally by
\[a\cdot (m_\af)_{\af\in\CL}=((a\circ f_{\emptyset[\af]})m_\af)_{\af\in\CL},\]
where $a\in C_0(X_\emptyset)$ and $(m_\af)_{\af\in\CL}\in M(\CB,\CL,\theta,\CI_\af)$.

\begin{prop}\label{prop:iso.corr}
The C*-correspondences $C_d(E^1)$ and $M(\CB,\CL,\theta,\CI_\af)$ defined above are isomorphic.
\end{prop}

\begin{proof}
Since $E^1$ can be identified with the disjoint union $\bigsqcup_{\af\in\CL} X_\af$, we can identify a function $p\in C(E^1)$ with the element $(p_\af)\in\prod_{\af\in\CL}M_\af$, where $p_\af$ is such that $p_\af(\xi)=p(e^{\af}_{\xi})$  for $\xi\in X_\af$. It is straightforward to check that $\left<p,p\right>\in C_0(X_\emptyset)$ is equivalent to $\sum_{\af\in\CL}p_\af^*p_\af\in C_0(X_\emptyset)$, and that via this identification the left and right actions are the same.
\end{proof}

\begin{remark}\label{rmk:iso.corr.alg}
If we let $\CO(E)$ be the the C*-algebra associated with the topological correspondence $E=(E^1,d,r)$, that is the Cuntz-Pimsner algebra of the C*-correspondence $C_d(E^1)$ over $C_0(X_\emptyset)$, then we have that $\CO(E)\cong\CO_{M(\CB,\CL,\theta,\CI_\af)}\cong C^*(\CB,\CL,\theta,\CI_\af)$, where the last isomorphism is given by \cite[Corollary 5.6]{CaK2}.
\end{remark}

\subsection{Another boundary path groupoid}

Given a topological correspondence $E=(E^1,d,r)$ from $F^0$ to $E^0$, we define the following subsets of $F^0$ (\cite[Section~1]{Ka2004}):
\begin{align*}F_{sce} & := F^0 \setminus \overline{r(E^1)}, \\
F^0_{fin}&:=\{v \in F^0: \exists V~\text{neighborhood of}~v ~\text{such that}~ r^{-1}(V)~\text{is compact} \},\\
F^0_{rg}&:=F^0_{fin} \setminus \overline{F^0_{sce}},\\
F^0_{sg} & :=F^0 \setminus F^0_{rg}.
\end{align*}

Although stated for topological graphs in \cite{Ka2004}, the following result also holds for topological correspondences.

\begin{prop}\label{prop:reg.vertex}(\cite[Proposition~2.8]{Ka2004})
For each $v\in F^0$, we have that $v\in F_{rg}$ if and only if there exists a compact neighborhood $V$ of $v$ such that $r^{-1}(V)$ is compact and $r(r^{-1}(V))=V$.
\end{prop}

Going back to the topological correspondence of Proposition~\ref{def:topological correspondence}, we recall that $E^0$ is an open subset of $F^0$. We also consider the sets $E^0_{rg}=F^0_{rg}\cap E^0$ and $E^0_{sg}=F^0_{sg}\cap E^0$. Although we do not have a topological graph, the condition that $E^0$ is an open subset of $F^0$ allows us to define a path space similarly to what is done in \cite{KL2017,Yeend2006}.

For $n \geq 2$, we denote by $E^n$ the space of paths of length $n$, that is,
\[E^{n}:=\{(e_1,\ldots,e_n)\in \prod_{i=1}^n E^1:d(e_i)=r(e_{i+1})(1\leq i<n)\}\]
which we regard as a subspace of the product space $\prod_{i=1}^n E^1$. Define the {\it finite path space} $E^*= \sqcup_{n=0}^{\infty} E^n$ with the disjoint union topology. 
Define the {\it infinite path space} as
\[E^{\infty}:=\{(e_i)_{i\in\mathbb{N}}\in \prod_{i=1}^\infty E^1:d(e_i)=r(e_{i+1})(i\in\mathbb{N})\}.\]

For elements $e_1,e_2\in E^1$, they are of the form $e_1=e^{\af_1}_{\eta_1}$ and $e_2=e^{\af_2}_{\eta_2}$ for $\af_1,\af_2\in\CL$, $\eta_1\in X_{\af_1}$ and $\eta_2\in X_{\af_2}$. The condition $d(e_1)=r(e_2)$ is then equivalent to $h_{[\af_1]\emptyset}(\eta_1)=f_{\emptyset[\af_2]}(\eta_2)$. In particular, if $A\in\eta_1$, then $\theta_{\af_2}(A)\in \eta_2$.

For an element, $(e_k)_{k=1}^n\in E^n$, we let $d((e_k)_{k=1}^n)=d(e_n)$ and $r((e_k)_{k=1}^n)=r(e_1)$ if $n\geq 1$. For $v\in E^0$, we let $r(v)=d(v)=v$. For infinite paths, we only define the range, namely, if $(e_k)_{k=1}^\infty\in E^\infty$, we define $r((e_k)_{k=1}^\infty)=r(e_1)$.

\begin{dfn}\label{def:boundary path space}  Let  $(\CB,\CL,\theta, \CI_\af)$ be a generalized Boolean dynamical system and $E=(E^1,d,r)$ be the associated  topological correspondence. The {\it boundary path space} of $E$ is defined by 
$$\partial E :=E^{\infty} \sqcup \{(e_k)_{k=1}^n \in E^* : d((e_k)_{k=1}^n )  \in E^0_{sg}\}.$$
We denote by $\sm_E:\partial E\setminus E^0_{sg}\to \partial E$ the shift map that removes the first edge for paths of length greater of equal to 2. For elements $\mu$ of length 1, $\sm_E(\mu)=d(\mu)$.
For a subset $S \subset E^*$, denote by 
$$\CZ(S)=\{\mu \in \partial E:~\text{either} ~r(\mu) \in S, ~\text{or there exists}~  1 \leq i \leq |\mu| ~\text{such that}~ \mu_1 \cdots \mu_i \in 
S \}.$$
We endow $\partial E$ with the topology generated by the basic open sets $\CZ(U)\cap \CZ(K)^c$, where $U$ is an open set of $E^*$ and $K$ is a compact set of $E^*$.
\end{dfn}

This topology can also be describe by convergence of nets. The following lemma follows from \cite[Theorem 3.10]{Cas2020} and generalizes the description of convergence of sequences in $\partial E$ found in \cite[Lemma 4.8]{KL2017}.

\begin{lem}\label{convergence on partial E}
A net $\{\mu^\lambda\}_{\lambda\in \Lambda}$ converges to $\mu$ in $\partial E$ if and only if
	\begin{enumerate}
		\item $\{r(\mu^\lambda)\}_{\lambda\in\Lambda}$ converges to $r(\mu)$;
		\item\label{item:conv1} for all $1\leq k \leq|\mu|$ with $k\neq\infty$, there exists $\lambda_0\in \Lambda$ such that for all $\lambda\geq \lambda_0$, $|\mu^\lambda|\geq k$ and $(\mu_1^\lambda,\ldots,\mu_k^\lambda)_{\lambda\geq \lambda_0}$ converges to $(\mu_1,\ldots,\mu_{k})$ in $E^k$;
		\item\label{item:conv3} if $|\mu|<\infty$, then for any compact $K\subseteq E^1$, there exists $\lambda_0\in \Lambda$ such that for all $\lambda\geq \lambda_0$, either $|\mu^\lambda|=|\mu|$, or $|\mu^\lambda|>|\mu|$ and $\mu_{|\mu|+1}^\lambda\notin K$.
	\end{enumerate}
\end{lem}

\begin{remark}
We could have defined a topological graph by putting $E^0=F^0$. There are two reasons why we use a topological correspondence instead of a topological graph. The first is the C*-correspondence found in the previous section, where we use $C_0(X_\emptyset)$ as the coefficient algebra instead of $C_0(X_\emptyset\cup\{\emptyset\})$. The second reason is that we want the boundary path space of the topological correspondence to be homeomorphic to the tight spectrum of the inverse semigroup associated with a generalized Boolean dynamical system. If we were to allow $\emptyset$ to be an element of $E^0$, for it not to belong to the boundary space, we would have to prove that $\emptyset\in E^0_{rg}$, however, this is not always the case, for instance if there is an infinite amount of elements $\af\in\CL$ such that $\emptyset$ belongs to the image of $f_{\emptyset[\af]}$.
\end{remark}

\begin{lem}\label{lem:word from path}
Let $(e^{\af_k}_{\eta_k})_{k=1}^n\in E^*$, where $1\leq n$. Then $\af_1\cdots\af_n\in\CW^*$. Moreover for all $1\leq m\leq n$, we have that $\eta_m\cap \CI_{\af_{1,m}}$ is an ultrafilter in $\CI_{\af_{1,m}}$.
\end{lem}

\begin{proof}
We prove by induction on $n$. If $n=1$, then $\af_1\in\CL\subseteq\CW^*$ and $\eta_1\cap\CI_{\af_1}=\eta_1$ is an ultrafilter in $\CI_{\af_1}$ by definition.

Suppose the result true for a fixed $n$, and let $(e^{\af_k}_{\eta_k})_{k=1}^{n+1}\in E^*$. By the induction hypothesis there exists $A\in\eta_n\cap \CI_{\af_{1,n}}$. As observed above, since $d(e^{\af_n}_{\eta_n})=r(e^{\af_{n+1}}_{\eta_{n+1}})$, we have that $\theta_{\af_{n+1}}(A)\in\eta_{n+1}$. By Lemma \ref{properties of I}, $\theta_{\af_{n+1}}(A)\in\CI_{\af_{1,n+1}}$ so that $\CI_{\af_{1,n+1}}\neq\{\emptyset\}$ and $\af_{1,n+1}\in\CW^*$. The fact that $\eta_{n+1}\cap \CI_{\af_{1,n+1}}$ is an ultrafilter follows from Proposition \ref{A}.
\end{proof}

\begin{lem}\label{lem:sing.vertex}
Let $\xi\in X_\emptyset$. Then $\xi\in E^0_{sg}$ if and only if $A\notin\CB_{reg}$ for all $A\in\xi$.
\end{lem}

\begin{proof}
Notice that for $A\in\CB$, we have that $r^{-1}(Z(A))=\bigcup_{\af\in\Delta_A} Z^1(\af,\theta_\af(A))$.

Suppose first that there exists $A\in\xi\cap\CB_{reg}$, then $Z(A)$ is a compact neighborhood of $\xi$ such that $r^{-1}(Z(A))$ is compact because it is a finite union of compact sets.

We claim that $r(r^{-1}(Z(A))=Z(A)$. Let $\eta\in Z(A)$ and $\af\in\Delta_A$. The set $\{B\in\CI_\af:\theta_\af(C)\subseteq B\text{ for some }C\in\eta\}$ is a filter in $\CI_\af$ and therefore it is contained in an ultrafilter $\zeta\in X_{\af}$. We have that $r(e^{\af}_{\zeta})=f_{\emptyset[\af]}(\zeta)=\{D\in\CB:\theta_\af(D)\in\zeta\}\supseteq\eta$. Since $\eta$ is an ultrafilter, we get $r(e^{\af}_{\zeta})=\eta$.
By Proposition~\ref{prop:reg.vertex}, we have that $\xi\in E^0_{rg}$.

Now suppose that $\xi\in E^0_{rg}$. By Proposition~\ref{prop:reg.vertex}, there exists a compact neighborhood $V$ of $\xi$ such that $r^{-1}(V)$ is compact and $r(r^{-1}(V))=V$. Then, there exists $A\in\xi$ such that $\xi\in Z(A)\subseteq V$. By the conditions on $V$, we also have that $r^{-1}(Z(B))$ is compact and $r(r^{-1}(Z(B))=Z(B)$ for every $\emptyset\neq B\subseteq A$. The first part implies that $\lambda_B<\infty$, while the second part implies that $\lambda_B>0$. It follows that $A\in\CB_{reg}$.
\end{proof}

\begin{thm}\label{thm:iso.tight.boundary}
Let $(\CB,\CL,\theta,\CI_\af)$ a generalized Boolean dynamical system, $\mathsf{T}$ the tight spectrum of its inverse semigroup and $\partial E$ the boundary path space of its topological correspondence. Then, there exists a homeomorphism $\phi:\mathsf{T}\to \partial E$ such that $\phi\circ\sigma=\sm_E\circ\phi|_{\mathsf{T}^{(1)}}$.
\end{thm}

\begin{proof}
Given a tight filter $\xi^\af$ in $E(S)$ with $|\af|\geq 1$, we put $\eta_1=\xi_1$ and $\eta_n=h_{[\af_{1,n-1}]\af_n}(\xi_n)$ for $2\leq n\leq|\af|$. Notice that for each $n$, we have that $\eta_n\in X_{\af_n}$ and $d(e^{\af_n}_{\eta_n})=h_{[\af_{1,n}]\emptyset}(\xi_n)=H_{[\af_{1,n}]\af_{n+1,|\af|}}(\xi^\af)_0$.

Define $\phi:\mathsf{T}\to\partial E$ by
\[\phi(\xi^\af)=\begin{cases}
\xi_0 & \text{if } \af=\emptyset, \\
(e^{\af_n}_{\eta_n})_{n=1}^{|\af|} &\text{if }|\af|\geq 1.
\end{cases}\]

\begin{itemize}
    \item The map $\phi$ is well defined.
\end{itemize}

If $\af=\emptyset$, we have that $\xi_0\in E^0_{sg}$ by Theorem \ref{char:tight} and Lemma \ref{lem:sing.vertex}, and hence $\phi(\xi^\af)\in\partial E$. If $1\leq|\af|<\infty$, then $d(e_{\eta_n}^{\af_n})=H_{[\af_{1,n}]\af_{n+1,|\af|}}(\xi^\af)_0\in E^0_{sg}$ by Theorems \ref{char:tight} and \ref{cutting gives tight} as well as Lemma \ref{lem:sing.vertex}. Moreover, given $1\leq i<|\af|$, using Lemma \ref{h comp f}, we have that
\begin{align*}
    d(e^{\af_n}_{\eta_n}) &=h_{[\af_{1,n}]\emptyset}(\xi_n)\\
    &=h_{[\af_{1,n}]\emptyset}(f_{\af_{1,n}[\af_{n+1}]}(\xi_{n+1}))\\
    &=f_{\emptyset[\af_{n+1}]}(h_{[\af_{1,n}]\af_{n+1}}(\xi_{n+1}))\\
    &=f_{\emptyset[\af_{n+1}]}(\eta_{n+1}) \\
    &=r(e^{\af_{n+1}}_{\eta_{n+1}}).
\end{align*}
Hence $\phi(\xi^\af)\in\partial E$. The above computation also show that $\phi(\xi^\af)\in\partial E$ if $|\af|=\infty$.

\begin{itemize}
    \item The map $\phi$ is injective.
\end{itemize}
By Theorem \ref{filter-bijective-pair}, a filter $\xi\in\mathsf{T}$ is completely determined by the pair $(\af,\{\xi_n\}_{n=0}^{|\af|})$. Since the maps $h$ are injective by Lemma \ref{h comp f}, we get that $\phi$ is injective.

\begin{itemize}
    \item The map $\phi$ is surjective.
\end{itemize}

Let $\mu\in\partial E$. If $|\mu|=0$, then $\mu\in X_{\emptyset}$ and we can let $\xi$ be the filter in $E(S)$ associated with the pair $(\emptyset,\{\mu\})$. Then $\xi\in\mathsf{T}$ by Lemma \ref{lem:sing.vertex} and Theorem \ref{char:tight}.

Suppose now that $1\leq|\mu|<\infty$, so that $\mu=(e^{\af_k}_{\eta_k})_{k=1}^{|\mu|}$. By Lemma \ref{lem:word from path}, we have that $\af_1\ldots\af_{|\mu|}\in\CW^*$ and $\eta_k\cap\CI_{\af_{1,k}}$ is an ultrafilter in $\CI_{\af_{1,k}}$ for all $1\leq k\leq |\mu|$. Define a family of filters $\{\CF_k\}_{k=0}^{|\mu|}$ by
\[\CF_k=\begin{cases}
f_{\emptyset[\af_1]}(\eta_1) & \text{if }k=0,\\
\eta_k\cap\CI_{\af_{1,k}} & \text{if }k>0.
\end{cases}\]
Observe that for $k>1$, we have that $\CF_k=g_{(\af_1\ldots\af_{k-1})\af_k}(\eta_k)$, and for $k=1$, we have that $\CF_1=\eta_1$. We show that $\{\CF_k\}_{k=0}^{|\mu|}$ is complete for $\af_1\ldots\af_{|\mu|}$. For that, we use Remark \ref{rmk:complete family using f}. By definition, we get $\CF_0=f_{\emptyset[\af_1]}(\eta_1)=f_{\emptyset[\af_1]}(\CF_1)$. Now, fix $1\leq k<|\mu|$, then by Lemmas \ref{g comp f} and \ref{h comp f} and the fact that $\mu$ is a path on the topological correspondence, we get
\begin{align*}
    f_{\af_1\ldots\af_k[\af_{k+1}]}(\CF_{k+1}) & = f_{\af_1\ldots\af_k[\af_{k+1}]}(g_{(\af_1\ldots\af_{k})\af_{k+1}}(\eta_{k+1}))\\
    &=g_{(\af_1\ldots\af_{k})\emptyset}(f_{\emptyset[\af_{k+1}]}(\eta_{k+1})) \\
    &=g_{(\af_1\ldots\af_{k})\emptyset}(h_{[\af_k]\emptyset}(\eta_k)) \\
    &=g_{(\af_1\ldots\af_{k-1})\af_k}(g_{(\af_{k})\emptyset}(h_{[\af_k]\emptyset}(\eta_k))) \\
    &=g_{(\af_1\ldots\af_{k-1})\af_k}(\eta_k)\\
    &=\CF_k.
\end{align*}
By Theorem \ref{filter-bijective-pair}, the pair $(\af_1\ldots\af_{|\mu|},\{\CF_k\}_{k=0}^{|\mu|})$ is associated with a filter $\xi$ in $E(S)$, which is tight by Theorem \ref{char:tight} and Lemma \ref{lem:sing.vertex}.

Finally, suppose that $|\mu|<\infty$, so that $\mu=(e^{\af_k}_{\eta_k})_{k=1}^{\infty}$. Notice that, for each $n\in\mathbb{N}$, we have that $\mu=(e^{\af_k}_{\eta_k})_{k=1}^{n}\in E^n$. By Lemma \ref{lem:word from path}, we get $\af_1\ldots\af_n\in\CW^*$ and hence $\af_1\af_2\ldots\CW^{\infty}$. Analogously to the finite length case, we can build a complete family family for $\af_1\af_2\ldots$ and a filter associated with the corresponding pair which is tight by Theorem \ref{char:tight}.

\begin{itemize}
    \item The map $\phi$ is continuous.
\end{itemize}

Let $\{\xi^{(\lambda)}\}_{\lambda\in\Lambda}$ be a net converging to $\xi$ in $\mathsf{T}$. We show that $\{\mu^\lambda\}_{\lambda\in\Lambda}:=\{\phi(\xi^{(\lambda)})\}_{\lambda\in\Lambda}$ converges to $\mu:=\phi(\xi)$ using the characterization of convergence on $\partial E$ given by Lemma \ref{convergence on partial E}. We let $\af$ be the word associated with $\xi$ and, for each $\lambda\in\Lambda$, we let $\af^{(\lambda)}$ be the word associated with $\xi^{(\lambda)}$.

(1) If $r(\mu)=\xi_0\neq\emptyset$, then there exists $\lambda_0\in\Lambda$ such that $r(\mu^\lambda)=\xi_0^{(\lambda)}\neq\emptyset$ for all $\lambda\geq\lambda_0$. It is straightforward to check that $(\xi_0^{(\lambda)})_{\lambda\geq\lambda_0}$ converges to $\xi_0$ and hence $\{r(\mu^\lambda)\}_{\lambda\in\Lambda}$ converges to $r(\mu)$. \ If $r(\mu)=\xi_0=\emptyset$, then $|\af|\geq 1$ and using any $A\in\xi_1$ we see that there exists $\lambda_0$ such that $\af^{(\lambda)}_1=\af_1$ for all $\lambda\geq\lambda_0$. That $\{r(\mu^\lambda)\}_{\lambda\in\Lambda}$ converges to $r(\mu)$ then follows from the continuity of $f_{\emptyset[\af_1]}$.

(2) Suppose that $1\leq k\leq|\mu|=|\af|$ and $k\neq\infty$. Taking any $A\in\xi_k$, we see that there exists $\lambda_0\in\Lambda$ such that $\af_{1,k}$ is a beginning of $\af^{(\lambda)}$ for all $\lambda\geq\lambda_0$, so that $|\mu^\ld| \geq k$ for all $\ld\geq\ld_0$. Also we see that $(\xi_i^{(\lambda)})_{\lambda\geq\lambda_0}$ converges to $\xi_i$ in $\CI_{\af_{1,i}}$ for all $1\leq i\leq k$. By the continuity of the $h$ maps (see Lemma \ref{h comp f}), we conclude that $(\mu_1^\lambda,\ldots,\mu_k^\lambda)_{\lambda\geq \lambda_0}$ converges to $(\mu_1,\ldots,\mu_{k})$ in $E^k$. 

(3) Suppose that $|\af|=|\mu|<\infty$ and let $K\subseteq E^1$ be compact. By the definition of the topology on $E^1$ and compactness of $K$, there exists $Z^1(\beta_1,B_1),\ldots,Z^1(\beta_n,B_n)$ basic compact-open subsets of $E^1$ such that $K\subseteq Z^1(\beta_1,B_1)\cup\cdots\cup Z^1(\beta_n,B_n)$. Fix $A\in\xi_{|\af|}$. Using the convergence of $\{\xi^{(\lambda)}\}_{\lambda\in\Lambda}$ to $\xi$, we can find $\lambda_0$ such that $(\af,A,\af)\in\xi^{(\lambda)}$ and $(\af\beta_i,\theta_{\beta_i}(A)\cap B_i,\af\beta_i)\notin\xi^{(\lambda)}$ for all $\lambda\geq\lambda_0$. Hence, if $\lambda\geq\lambda_0$, we have that $|\mu^{\lambda}|=|\af^{(\lambda)}|\geq |\af|=|\mu|$. And if $|\mu^\lambda|>|\mu|$, then $\mu^{\lambda}_{|\mu|+1}=e^{\af^{(\lambda)}_{|\mu|+1}}_{\eta^{(\lambda)}_{|\mu|+1}}\notin K$ because either $\af^{(\lambda)}_{|\mu|+1}\neq \beta_i$ for all $1\leq i\leq n$, or if $\af^{(\lambda)}_{|\mu|+1}= \beta_i$ for some $1\leq i\leq n$, then $B_i\notin\xi_{|\mu|+1}^{(\lambda)}\subseteq \eta_{|\mu|+1}^{(\lambda)}$.

\begin{itemize}
    \item The map $\phi^{-1}$ is continuous.
\end{itemize}

Let $\{\mu^\ld\}_{\ld\in\Ld}$ be a net converging to $\mu$ in $\partial E$. We show that the net $\{\xi^\lambda\}_{\lambda\in\Lambda}:=\{\phi^{-1}(\mu^{(\lambda)})\}_{\lambda\in\Lambda}$ converges to $\xi:=\phi^{-1}(\mu)$ in $\mathsf{T}$. We write $\mu=(e^{\af_k}_{\eta_k})_{k=1}^{|\mu|}$ and $\mu^\ld=(e^{\af^{(\ld)}_k}_{\eta^{(\ld)}_k})_{k=1}^{|\mu^\ld|}$ whenever they have positive length and we write $\mu=\eta_0$, $\mu^\ld=\eta_0^{(\ld)}$, $\af=\emptyset$ and $\af^{(\lambda)}=\emptyset$ whenever they are paths of length zero. Then $\af$ is the word associated with $\xi$ and, for each $\lambda\in\Lambda$, $\af^{(\lambda)}$ is the word associated with $\xi^{(\lambda)}$.

Let $(\bt,B,\bt)\in E(S)$ and suppose first that $(\bt,B,\bt)\in\xi$. We have to find $\ld_0\in\Ld$ such that for all $\ld\geq\ld_0$, we have that $(\bt,B,\bt)\in\xi^{(\lambda)}$.

If $\beta=\emptyset$, then $B\in \xi_0=r(\mu)$. Because $\{r(\mu^\ld)\}_{\ld\in\Ld}$ converges to $r(\mu)$, there exists $\ld_0$ such that $B\in r(\mu^\ld)=\xi_0^{(\ld)}$ for all $\ld\geq\ld_0$. Hence $(\beta,B,\beta)\in\xi^{(\lambda)}$ for all $\ld\geq\ld_0$.

If $\beta\neq\emptyset$, then $\beta$ is a beginning of $\af$, so that $1\leq|\bt|\leq|\af|$. For each $1\leq i\leq|\bt|$, let $A_i\in\eta_i$. We may assume that $A_{|\bt|}=B$. Consider the open subset $U:=\CZ((Z^1(\af_1,A_1)\times\cdots\times Z^1(\af_{|\bt|},A_{|\bt|}))\cap E^{|\bt|})$ of $\partial E$. Then, there exists $\ld_0\in\Ld$  such that $\mu^\ld\in U$ for all $\ld\geq\ld_0$. This implies that for all $\ld\geq\ld_0$, we have that $\bt$ is the beginning of $\af^{(\ld)}$ and $B\in\eta_{|\bt|}^{(\ld)}$. Also, since $B\in\CI_{\bt}$, we see that $B\in\xi_{|\bt|}^{(\af)}=\eta_{|\bt|}^{(\af)}\cap\CI_{\bt}$ for all $\ld\geq\ld_0$. Hence $(\beta,B,\beta)\in\xi^{(\lambda)}$ for all $\ld\geq\ld_0$.

Suppose second that $(\bt,B,\bt)\notin\xi$. Now, we have to find $\ld_0\in\Ld$ such that for all $\ld\geq\ld_0$, we have that $(\bt,B,\bt)\notin\xi^{(\lambda)}$. We consider a few cases.

If $\bt=\emptyset$, then $r(\mu)=\xi_0$ belongs to the open subset $F^0\setminus Z(B)$ of $F^0$. Because $\{r(\mu^\ld)\}_{\ld\in\Ld}$ converges to $r(\mu)$, there exists $\ld_0$ such that $\xi_0^{(\ld)}=r(\mu^\ld)\in F^0\setminus Z(B)$ for all $\ld\geq\ld_0$. Hence $(\bt,B,\bt)\notin\xi^{(\lambda)}$ for all $\lambda\geq\lambda_0$.

If $1\leq|\bt|\leq|\af|$ and $\bt$ is not a beginning of $\af$, then we can find $\ld_0\in\Ld$ such that $\af^{(\ld)}_{1,|\bt|}=\af_{1,|\bt|}\neq\bt$ for all $\ld\geq\ld_0$. Hence $(\bt,B,\bt)\notin\xi^{(\lambda)}$ for all $\lambda\geq\lambda_0$.

If $1\leq|\bt|\leq|\af|$ and $\bt$ is a beginning of $\af$, then we can find $\ld_1\in\Ld$ such that $\af^{(\ld)}_{1,|\bt|}=\af_{1,|\bt|}=\bt$ for all $\ld\geq\ld_1$. Moreover $\{\mu^{\ld}_{1,|\bt|}\}_{\ld\geq\ld_1}$ converges to $\mu_{1,|\bt|}$. In particular $\{\eta^{(\ld)}_{|\bt|}\}_{\ld\geq\ld_1}$ converges to $\eta_{|\bt|}$. By Lemma \ref{g comp f}, the maps $g$ are continuous so that $\{\xi^{(\ld)}_{|\bt|}\}_{\ld\geq\ld_1}$ converges to $\xi_{|\bt|}$. Notice that $B\notin \xi_{|\bt|}$ so that there exists $\ld_0\geq\ld_1$ such that $B\notin \xi^{(\ld)}_{|\bt|}$ for all $\ld\geq\ld_0$. Hence $(\bt,B,\bt)\notin\xi^{(\lambda)}$ for all $\lambda\geq\lambda_0$.

If $|\bt|>|\af|$ and $\af$ is not a beginning of $\bt$, then we can find $\ld_0\in\Ld$ such that $\af^{(\ld)}_{1,|\af|}=\af\neq\bt_{1,|\af|}$ for all $\ld\geq\ld_0$. Hence $(\bt,B,\bt)\notin\xi^{(\lambda)}$ for all $\lambda\geq\lambda_0$.

Finally, we suppose that $|\bt|>|\af|$ and $\af$ is a beginning of $\bt$. Since $B\in\CI_{\bt}$, there exists $C\in\CB$ such that $B\subseteq \theta_{\beta_{2,|\beta|}}(C)$. Define $B':=\theta_{\beta_{2,|\af|+1}}(C)\in\CI_{\beta_{1,|\af|+1}}\subseteq \CI_{\beta_{|\af|+1}}$ and observe that $K:=Z^1(\bt_{|\af|+1},B')$ is a compact subset of $E^1$. By (3) of Lemma \ref{convergence on partial E}, there exists $\ld_0\in\Ld$ such for all $\ld\geq\ld_0$, we have that $|\mu^\ld|=|\mu|$, or $|\mu^\ld|>|\mu|$ and $\mu_{|\mu|+1}\notin K$. In the first case $(\beta,B,\beta)\notin\xi^{(\ld)}$ because $|\af^{(\ld)}|=|\mu^{\ld}|=|\mu|=|\af|<|\bt|$. In the second case, if $\bt$ is not beginning of $\af^{(\ld)}$, then $(\beta,B,\beta)\notin\xi^{(\ld)}$ by comparing the words. If $\bt$ is a beginning of $\af^{(\ld)}$, then $B'\notin\eta^\ld_{|\af|+1}$ since $\mu^{\ld}_{|\af|+1}\notin K$. This implies that $B'\notin\xi^{(\ld)}_{|\af|+1}$ so that $\theta_{|\af|+2,|\bt|}(B')=\theta_{\beta_{2,|\beta|}}(C)\notin\xi^{(\ld)}_{|\bt|}$. Because $B\subseteq \theta_{\beta_{2,|\beta|}}(C)$, we have that $(\beta,B,\beta)\notin \xi^{(\ld)}$. Hence $(\bt,B,\bt)\notin\xi^{(\lambda)}$ for all $\lambda\geq\lambda_0$.

\begin{itemize}
    \item $\phi\circ\sigma=\sm_E\circ\phi|_{\mathsf{T}^{(1)}}$.
\end{itemize}

Let $\xi^\af\in \mathsf{T}^{(1)}$. Suppose first that $|\af|=1$. In this case, $\sigma(\xi)=H_{[\af]\emptyset}(\xi)$ and $\phi(\sigma(\xi))=H_{[\af]\emptyset}(\xi)_0=h_{[\af]\emptyset}(\xi_1)$. On the other hand $\phi(\xi)=e^\af_{\xi_1}$ and $\sm_E(\phi(\xi))=d(e^\af_{\xi_1})=h_{[\af]\emptyset}(\xi_1)$. Hence $\phi(\sigma(\xi))=\sm_E(\phi(\xi))$.

Suppose now that $|\af|>1$ and write $\af=\af_1\bt$ where $\af_1\in\CL$ and $\bt\in\CW^{\leq\infty}$. In this case, $\sigma(\xi)=H_{[\af_1]\beta}(\xi)$ and $\phi(\sigma(\xi))=(e^{\beta_n}_{\zeta_n})_{n=1}^{|\bt|}$, where $\zeta_1=H_{[\af_1]\beta}(\xi)_1=h_{[\af_1]\beta_1}(\xi_2)$ and \begin{align*}
    \zeta_n &=h_{[\beta_1,n-1]\beta_n}(H_{[\af_1]\beta}(\xi)_n) \\
    &=h_{[\beta_1,n-1]\beta_n}(h_{[\af_1]\beta_{1,n}}(\xi_{n+1}) \\
    &=h_{[\af_1\beta_1,n-1]\beta_n}(\xi_{n+1})
\end{align*} for $2\leq n\leq|\bt|$. On the other hand $\phi(\xi)=(e^{\gamma_n}_{\eta_n})_{n=1}^{|\beta|+1}$, where $\gamma_1=\af_1$ and $\eta_1=\xi_1$, and for $2\leq n\leq|\beta|+1$, $\gamma_n=\bt_{n-1}$ and $\eta_n=h_{[\af_1\bt_{1,n-2}]\bt_{n-1}}(\xi_n)$. Applying $\sm_E$ to $\phi(\xi)$ and comparing with the above expression, we see that $\phi(\sigma(\xi))=\sm_E(\phi(\xi))$. 

\end{proof}

\begin{cor}
Let $(\CB,\CL,\theta,\CI_\af)$ a generalized Boolean dynamical system and $E=(E^1,d,r)$ its topological correspondence. The shift map $\sm_E:\partial E\setminus E^0_{sg}\to\partial E$ is a local homeomorphism. Moreover, if $\Gamma(\partial E,\sm_E)$ is the Renault-Deaconu groupoid associated with $\sm_E$, then $\Gamma(\partial E,\sm_E)\cong\Gamma(\CB,\CL,\theta,\CI_\af)$ and $C^*(\Gamma(\partial E,\sm_E))\cong C^*(\CB,\CL,\theta,\CI_\af)$.
\end{cor}

\begin{proof}
By the results of Section \ref{section:Renault-Deaconu}, $\Gamma(\CB,\CL,\theta,\CI_\af)$ is the Renault-Deaconu of the map $\sigma$. It is then straightforward to prove that $\Psi: \Gamma(\CB, \CL, \theta, \CI_\af) \to  \Gamma(\partial E, \sigma)$ defined by 
$$\Psi((\xi^{\af\gm}, |\af|-|\bt|, \eta^{\bt\gm}))=(\phi(\xi^{\af\gm}),|\af|-|\bt|, \phi(\eta^{\bt\gm} )), $$ where $\phi: \mathsf{T} \to \partial E $  the homeomorphism given in Proposition \ref{thm:iso.tight.boundary}, is a well defined isomorphism of topological groupoids. The isomorphism $C^*(\Gamma(\partial E,\sm_E))\cong C^*(\CB,\CL,\theta,\CI_\af)$ follows from Theorem \ref{C*-isom}. 
\end{proof}

\begin{remark}
We believe that the results of \cite{KL2017,Yeend2006} can be adapted to prove that $C^*(\Gamma(\partial E,\sm_E))\cong\CO(E)$ directly. Using the isomorphism of Remark \ref{rmk:iso.corr.alg}, we could then get a groupoid model for $C^*(\CB,\CL,\theta,\CI_\af)$ in a different way.
\end{remark}

\end{document}